\documentclass[11pt,leqno,oneside]{amsbook}
\usepackage[T1]{fontenc}
\usepackage[utf8]{inputenc}
\usepackage{lmodern,amssymb,booktabs,longtable,array,microtype,needspace}
\usepackage[letterpaper,textwidth=6.15in,textheight=8.65in,centering,headheight=14pt,headsep=20pt]{geometry}
\usepackage{enumitem}
\usepackage[hidelinks,bookmarksnumbered]{hyperref}
\makeindex
\newtheorem{theorem}{Theorem}[chapter]
\newtheorem{lemma}[theorem]{Lemma}
\newtheorem{proposition}[theorem]{Proposition}

\theoremstyle{definition}
\newtheorem{definition}[theorem]{Definition}
\newtheorem{example}[theorem]{Example}
\theoremstyle{remark}

\numberwithin{equation}{chapter}
\newcommand{\one}{\mathbf{1}}
\newcommand{\two}{\mathbf{2}}
\newcommand{\setm}[2]{\{\,#1\mid #2\,\}}
\newcommand{\Ap}[1]{A^{\times #1}}
\newcommand{\Bp}[1]{B^{\times #1}}

\newcommand{\xp}[2]{{#1}^{\times #2}}
\newcommand{\K}{\mathcal K}
\newcommand{\OO}{\mathcal O}
\DeclareMathOperator{\height}{height}
\DeclareMathOperator{\Con}{Con}
\DeclareMathOperator{\Tow}{Tow}
\DeclareMathOperator{\Max}{Max}

\DeclareMathOperator{\Der}{Der}
\DeclareMathOperator{\Red}{Red}
\DeclareMathOperator{\poly}{poly}
\DeclareMathOperator{\wt}{wt}
\DeclareMathOperator{\comp}{comp}
\DeclareMathOperator{\iso}{iso}
\DeclareMathOperator{\id}{id}

\DeclareMathOperator{\hdeg}{deg_h}
\DeclareMathOperator{\coord}{coord}
\DeclareMathOperator{\sig}{sig}

\setlist[enumerate]{itemsep=3pt,topsep=5pt}
\hypersetup{pdftitle={Generalizing the $A^A$ Problem for Finite Ordered Sets: Reconstruction through Twelve Occurrences},pdfauthor={George Gratzer}}
\begin{document}
\frontmatter
\title[Generalizing the $A^A$ Problem]{\fontsize{26}{32}\selectfont
Generalizing the $A^A$ Problem\newline for Finite Ordered Sets\\[14pt]
{\Large Reconstruction through Twelve Occurrences}}
\thanks{Research manuscript. October 2, 2026.}
\begin{abstract}
Let X and Y be finite ordered sets
and let $X^Y$ denote the ordered set of order-preserving maps from $Y$
to $X$. 

Let $t$ be a~term formed from a~single variable by exponentiation.
We prove that, for every term $t$ with at most twelve variable
occurrences and arbitrary finite ordered sets $A,B$,
$t(A)\cong t(B)$ implies that $A\cong B$. 

General constructions and
reconstruction principles are developed first; the second part gives
the occurrence-by-occurrence proofs. 

At twelve occurrences the proof
covers 4,766 interchange normal forms, representing all 58,786 binary
parenthesizations, including 40 exceptional forms. Two additional
unbounded reconstruction families, complete finite audit data, and
457 exact height-comparison certificates are included.
\end{abstract}
\author{George Gr\"atzer}
\date{October 2, 2026}
\maketitle

\tableofcontents

\chapter*{Preface}
This monograph studies reconstruction from exponential expressions in one
finite ordered set. Its principal conclusion is that every exponential
term with at most twelve occurrences of the variable determines its
input up to isomorphism. The number twelve counts occurrences, not
input elements. The input cardinality is unrestricted.

The exposition has two parts. Part~I develops the constructions and
general results: arithmetic of components, canonical exponent
coordinates, auxiliary orders, maximum-height intervals, height
reduction, a~class reconstructed at every occurrence count, strong
reconstruction, forest recovery, polynomial perturbations, and
substitution. Part~II proves the occurrence-bounded assertions, with
separate treatment of the exceptional branches. The last theorem is
cumulative: every smaller occurrence count has been proved, not merely
assumed to follow from the case of twelve.

The foundational $A^A$ argument comes from the joint manuscript of
Jonathan David Farley and George Gr\"atzer. The present exposition
incorporates that argument and the subsequent occurrence-count
manuscripts into a~single system of notation and internal references.
No reconstruction theorem from that manuscript sequence is used as an
unproved external premise. The classical structural inputs of Hashimoto,
McKenzie, and Farley are stated explicitly and cited where used.

Some arguments recover an entire disjoint sum without identifying its
displayed summands. This distinction is essential. An isomorphism of
ordered sets does not label the components by the formula that produced
them. Likewise, a~reduction chain of a~known shape is not automatically
recognizable when two possible chains have different lengths. The
exceptional proofs deal explicitly with these collisions.

The accompanying audit is part of the working research record. It
checks normal-form enumeration, normalization of every binary
parenthesization, individual multiplicities, case assignments, forest
signatures, and exact integer-polynomial comparison certificates. It is
not a~proof-assistant verification of the structural theorems or of the
entire mathematical argument. The distinction between a~mathematical
proof and an executed finite audit is maintained throughout.

All mathematical text and printed tables are included in the single
\LaTeX\ source. The clean archive also contains the audit programs and
their executed results.

\chapter*{Notation}
\begin{longtable}{@{}p{.24\textwidth}p{.70\textwidth}@{}}
\toprule
Symbol & Meaning\\\midrule
$X+Y$ & Disjoint sum.\\
$XY$, $X\times Y$ & Cartesian product.\\
$X^Y$ & Order-preserving maps $Y\to X$, ordered pointwise.\\
$X^{\times k}$ & Cartesian power; $X^{\times0}=\one$.\\
$kX$ & Disjoint sum of $k$ copies of $X$.\\
$\one$, $0$ & One-element ordered set and formal omitted sum.\\
$\height(X)$ & Maximum chain length, counting strict steps.\\
$\comp(X)$, $\iso(X)$ & Numbers of connected components and isolated points.\\
$\poly(X)$ & Component polynomial of $X$.\\
$\Con(P,Y)$ & Constant-containing component construction.\\
$(E_D,F_D)$ & Canonical core and exponent coordinate of a~prime factor.\\
$\Tow(P)$ & Sum of canonical exponent coordinates of a~connected $P$.\\
$\Max(X)$ & Selected maximum-height interval.\\
$\Der(X)$ & $\Tow(\Max(X))$.\\
$\Red(Q)$ & $\Max(\Tow(Q))$, for a~bounded positive-height $Q$.\\
$L,R,S,J$ & $\Max(A),\Tow(L),\Max(R),\Tow(S)$.\\
$\K$, $\OO$ & Componentwise contractible inputs and their complement.\\
$|t|$, $\wt(D)$ & Occurrence count of a~term and weight of a~forest.\\
$\sig(D)$ & Length and terminal Cartesian exponent of a~forest chain.\\
$E_{r,C}$ & $A^{\,A^{\times r}\times A^{C(A)}}$.\\
$N,p,f,g$ & Cardinalities of $A,P,F,G$.\\*
$P,F,G$ & $A^A,A^{A^{\times2}},A^{A^A}$.\\*
\bottomrule
\end{longtable}
\mainmatter
\part{General constructions and general results}

\chapter{The reconstruction problem}\label{ch:problem}

An exponential expression remembers more than its cardinality. The
question here is whether it remembers the ordered set from which it was
built. For ordered sets $A,B$, the simplest nontrivial instance asks
whether
\[
 A^A\cong B^B\quad\Longrightarrow\quad A\cong B.
\]
The general finite-occurrence version replaces $A^A$ by an arbitrary
fixed parenthesization of copies of $A$, using exponentiation alone.

\section{Conventions}
\index{conventions!ordered sets}\index{notation!conventions}
All ordered sets are finite and, by definition, nonempty. An ordered set is \emph{connected} when its
comparability graph is connected. It is \emph{nontrivial} when it has
more than one element. The one-element ordered set is $\one$.
The symbol $0$ denotes an omitted summand; it is not an additional
ordered set. In a~sum or product indexed by $\setm{u_i}{i\in I}$,
every index contributes, including repetitions.
\index{ordered!set}\index{connected!ordered set}\index{formal!zero}

The exponent $X^Y$ is the set of order-preserving maps from $Y$ to $X$,
with $f\leq g$ when $f(y)\leq g(y)$ for every $y\in Y$.
An isomorphism of exponents in this book is an isomorphism of ordered
sets, not an isomorphism of endomorphism monoids. Composition of maps
is not part of the structure being reconstructed.
\index{exponentiation}\index{endomorphism!ordered set}

\begin{definition}\label{D:term}
\index{exponential!term!definition}
An \emph{exponential term} is defined recursively: $x$ is a~term, and
$u^v$ is a~term when $u$ and $v$ are terms. Put
\[
 x(A)=A,\qquad (u^v)(A)=u(A)^{v(A)},\qquad
 |x|=1,\quad |u^v|=|u|+|v|.
\]
A~term $t$ is \emph{reconstructing} on a~class of inputs if
$t(A)\cong t(B)$ implies $A\cong B$ for inputs in that class.
Without a~specified class, all finite ordered sets are intended.
\end{definition}
\index{term!exponential}\index{occurrence!count}\index{reconstruction}

The main theorem, proved in Part~II, applies to every $t$ with
$1\leq|t|\leq12$. The same term is used on the two sides. Different
parenthesizations are not asserted to have isomorphic values.
For instance, the two three-occurrence terms are
$(A^A)^A\cong A^{A\times A}$ and $A^{A^A}$.

\section{Exponential calculus}
\index{exponential!calculus}\index{exponentiation!calculus}\index{interchange!identity}
\begin{proposition}\label{P:calculus}
\index{exponential!calculus!laws}
For ordered sets $X,Y,Z$,
\begin{equation}\label{E:calculus}
\begin{aligned}
 X^{Y+Z}&\cong X^Y\times X^Z,\\
 (X\times Y)^Z&\cong X^Z\times Y^Z,\\
 (X^Y)^Z&\cong X^{Y\times Z},\\
 X^{\one}&\cong X,\qquad \one^X\cong\one.
\end{aligned}
\end{equation}
Consequently $(X^Y)^Z\cong(X^Z)^Y$.
\end{proposition}
\begin{proof}
Restriction to the two summands gives the first isomorphism, with
inverse obtained by joining the two maps. Coordinate projections give
the second isomorphism. For the third, send $f:Z\to X^Y$ to
$\widehat f(y,z)=f(z)(y)$. If $(y,z)\leq(y',z')$, then
$f(z)(y)\leq f(z)(y')\leq f(z')(y')$, so $\widehat f$ is
order-preserving. Conversely, an order-preserving map on $Y\times Z$
gives an order-preserving map $Z\to X^Y$ by taking its sections.
These correspondences preserve and reflect the pointwise order.
The unit identities follow directly from the definitions. Commuting
the two Cartesian factors in the third identity proves interchange.
\end{proof}
\index{interchange!identity}\index{exponential!laws}

If $Y$ is connected, an order-preserving map $Y\to X_1+\cdots+X_a$
lands in a~single summand. Indeed, comparable domain points have values
in the same summand, and a~comparability path joins any two domain
points. Thus
\[
 (X_1+\cdots+X_a)^Y\cong X_1^Y+\cdots+X_a^Y
 \qquad(Y\text{ connected}).
\]
There is no such formula for a~disconnected exponent without first
splitting that exponent into its connected components.

\section{Retractions and antichains}
\index{retract!term retraction}\index{antichain!exponential terms}\index{evaluation!map}
\begin{lemma}\label{L:term-retract}
\index{retract!term retraction}
Every input $A$ is a~retract of every term value $t(A)$ and of every
Cartesian product of term values.
\end{lemma}
\begin{proof}
For any $X^Y$, the constant-map map $c:X\to X^Y$ is an order embedding.
Evaluation $e(f)=f(y_0)$ at a~fixed $y_0\in Y$ is order-preserving and
$ec=\id_X$. Induction on $t=u^v$ composes this retraction onto $u(A)$
with the inductively obtained retraction onto $A$. For a~product, use
one factor and fix elements in all other coordinates. The projection
onto the chosen factor is a~retraction, and composition finishes the
proof.
\end{proof}
\index{retract}\index{constant!map}\index{evaluation!map}

\begin{lemma}\label{L:antichains}
\index{antichain!reconstruction}
Every exponential term reconstructs antichains, and an antichain input
cannot have the same term value as a~positive-height input.
\end{lemma}
\begin{proof}
A~term evaluated on an antichain is an antichain, by induction. The
converse follows from Lemma~\ref{L:term-retract}. Its cardinality on the
$N$-element antichain is obtained by evaluating the same term with
ordinary positive-integer exponentiation. This cardinality is strictly
increasing in $N$. The induction step is
\[
 u(N)^{v(N)}<u(M)^{v(N)}\leq u(M)^{v(M)}\qquad(N<M),
\]
using strict increase of $u$, positivity of $v(N)$, and monotonicity of
positive-integer exponentiation in its exponent. Hence the term value
uniquely determines $N$ and the antichain.
\end{proof}
\index{antichain}

\section{Height and components}
\index{height!exponential terms}\index{connected!components!arithmetic}\index{component!polynomial}
The height $\height(X)$ is the greatest value of $|C|-1$ over chains
$C$ in $X$. Write $\comp(X)$ for the number of connected components and
$\iso(X)$ for the number of one-element components.
\begin{proposition}\label{P:height-components}
\index{height!component formula}
\begin{equation}\label{b:E:heights}
\begin{aligned}
 \height(X+Y)&=\max\{\height(X),\height(Y)\},\\
 \height(X\times Y)&=\height(X)+\height(Y),\\
 \height(X^Y)&=|Y|\height(X).
\end{aligned}
\end{equation}
Moreover,
\begin{equation}\label{b:E:component-arithmetic}
\begin{aligned}
 \comp(X+Y)&=\comp(X)+\comp(Y),&
 \comp(XY)&=\comp(X)\comp(Y),\\
 \iso(X+Y)&=\iso(X)+\iso(Y),&
 \iso(XY)&=\iso(X)\iso(Y).
\end{aligned}
\end{equation}
\end{proposition}
\begin{proof}
A~chain in a~sum lies in one summand. A~strict step in a~product chain
changes at least one coordinate, and the total number of changes is at
most the sum of the two heights. Interleaving maximum chains attains
this bound.

A~chain of maps changes the value at a~fixed domain point at most
$\height(X)$ times, so its length is at most $|Y|\height(X)$. For the
reverse inequality, choose a~maximum chain $x_0<\cdots<x_h$ in $X$ and
a~linear extension of $Y$. Starting from the constant $x_0$ map, change
values from $x_{j-1}$ to $x_j$ in reverse linear-extension order, one
point at a~time. Every intermediate map is order-preserving. This gives
$|Y|h$ strict steps.

The components of a~product are precisely the products of components.
A~product component has one element exactly when both factors do.
The component and isolated-point formulas follow.
\end{proof}
\index{height}\index{connected!components}\index{isolated!point}

\chapter{Arithmetic of finite ordered sets}\label{b:S:arithmetic}
Hashimoto's theorem~\cite{Hashimoto} gives unique factorization of a
connected ordered set into connected nontrivial directly indecomposable
factors. Introduce an indeterminate $z_D$ for each such factor type $D$.
Encode a~connected ordered set by the monomial of its factor
multiplicities, and encode an arbitrary ordered set $X$ by the sum of
its component monomials, with multiplicity. Denote the resulting
polynomial by $\poly(X)$. Then
\begin{equation}\label{b:E:polynomial}
\begin{aligned}
 \poly(X+Y)&=\poly(X)+\poly(Y),\\
 \poly(X\times Y)&=\poly(X)\poly(Y),\\
 \poly(X)=\poly(Y)&\quad\Longleftrightarrow\quad X\cong Y.
\end{aligned}
\end{equation}
These polynomials lie in an integral domain over $\mathbb Z$; any
particular argument involves only finitely many indeterminates.
Products therefore cancel. Disjoint sums cancel by their component
multiplicities. Also, $kX$ determines $X$ by division of those
multiplicities by $k$.

Give $z_D$ the positive weight $\height(D)$. For a~nonzero signed
polynomial $q$, write $\hdeg(q)$ for its greatest weighted degree. Then
\begin{equation}\label{b:E:weighted-degree}
 \hdeg(\poly(X))=\height(X),\qquad
 \hdeg(q_1q_2)=\hdeg(q_1)+\hdeg(q_2).
\end{equation}
The second equality follows because the product of the highest
homogeneous parts is nonzero.

\section{Cancellation and roots}
\index{cancellation!Cartesian product}\index{cancellation!disjoint sum}\index{Cartesian!root!uniqueness}
\begin{lemma}\label{L:arithmetic-cancellation}
\index{cancellation!arithmetic lemma}
Products and disjoint sums cancel. If $k\geq1$, each of $kX$ and
$X^{\times k}$ determines $X$. More generally, a~fixed nonconstant
polynomial $q$ with nonnegative integer coefficients determines its
ordered-set argument.
\end{lemma}
\begin{proof}
An actual ordered set has a~nonzero component polynomial. Product
cancellation follows by cancellation of this polynomial in an integral
domain. Sum cancellation is cancellation of component multiplicities,
and $kX$ is recovered by dividing every multiplicity by $k$.

For the polynomial assertion, put $a=\poly(A)$ and $b=\poly(B)$. Then
$q(a)-q(b)=(a-b)H(a,b)$, where $H$ has nonnegative coefficients and is
nonzero. Nonzeroness follows, for example, by substituting $1$ for every
factor variable: both $a$ and $b$ then become positive integers and
$q$ is nonconstant. Thus $q(a)=q(b)$ forces $a=b$. Taking $q(z)=z^k$
proves Cartesian-root uniqueness.
\end{proof}
\label{b:C:polynomial-injective}
\index{cancellation!product}\index{cancellation!sum}\index{Cartesian!root}\index{polynomial!injectivity}

\section{Three evaluations of the component polynomial}
\index{component!polynomial!evaluations}\index{component!polynomial!weighted degree}\index{polynomial!component polynomial}
Substitution $z_D=|D|$ gives $|X|$. Substitution $z_D=1$ gives
$\comp(X)$. The constant coefficient is $\iso(X)$. Thus one polynomial
simultaneously records cardinality, component counts, and isolated
points. None of these three specializations is injective by itself.
The full polynomial is needed for cancellation.

For example, let $C$ be a~connected nontrivial directly indecomposable
ordered set and put $X=2\one+C$. Then
\[
 \poly(X)=2+z_C,\qquad
 \poly(X\times X)=4+4z_C+z_C^2.
\]
The square has four isolated points, four components isomorphic to $C$,
and one component isomorphic to $C\times C$. This elementary example
also shows why a~coefficient counts components with multiplicity rather
than merely recording the presence of a~component type.
\index{component!polynomial}\index{multiplicity}

\begin{lemma}\label{L:degree-obstruction}
\index{degree!obstruction!weighted degree}
An ordered set of positive height cannot divide a~nonzero integer in
the signed component-polynomial ring. If $|A|=|B|$ and $A\ncong B$,
then $\hdeg(\poly(A)-\poly(B))\geq1$.
\end{lemma}
\begin{proof}
Every factor variable has a~positive height weight. A~nonzero integer
has weighted degree zero, while degrees add on nonzero products. This
proves the first assertion. For the second, a~nonzero degree-zero
difference would be a~nonzero integer. Cardinality evaluation would
then give $|A|-|B|\ne0$, a~contradiction.
\end{proof}
\index{weighted!degree}\index{degree!obstruction}

\chapter{Canonical exponent coordinates}\label{ch:coordinates}
For connected $P$ and arbitrary $Y$, let $\Con(P,Y)$ be the connected
component of $P^Y$ containing the maps constant on each component of
$Y$. All such maps belong to one component because $P$ is connected.
The usual exponential laws restrict to
\begin{equation}\label{b:E:Con-laws}
\begin{aligned}
 \Con(P\times Q,Y)&\cong\Con(P,Y)\times\Con(Q,Y),\\
 \Con(P,Y+Z)&\cong\Con(P,Y)\times\Con(P,Z),\\
 \Con(\Con(P,Y),Z)&\cong\Con(P,Y\times Z).
\end{aligned}
\end{equation}
Here $P,Q$ are connected. These identities and the following structural
theorem are the exponent-coordinate input from McKenzie and Farley;
see~\cite{McKenzie,McKenzieZigzag,Farley}, and specifically Lemma~7 of the preprint
version~\cite{FarleyPreprint}.

\begin{theorem}[Canonical exponent coordinates]\label{b:T:coordinates}
\index{canonical!coordinates!theorem}
Every connected nontrivial directly indecomposable ordered set $D$ has a
unique pair $(E_D,F_D)$, up to isomorphism, such that $F_D$ is connected,
$E_D$ is connected, nontrivial, and absolutely $\Con$-indecomposable, and
\[
 D\cong\Con(E_D,F_D).
\]
Absolute $\Con$-indecomposability means that $E_D$ is connected and
directly indecomposable, and a~representation
$E_D\cong\Con(U,V)$ forces $V\cong\one$ and $U\cong E_D$.
For connected $Y$, the ordered set $\Con(D,Y)$ is directly
indecomposable and its canonical pair is $(E_D,F_D\times Y)$.
\end{theorem}
The height construction above stays in the constant-containing
component. Consequently
\begin{equation}\label{b:E:coordinate-height}
 \height(D)=|F_D|\height(E_D),\qquad
 \height(F_D)<|F_D|\leq\height(D).
\end{equation}
This strict decrease will make an auxiliary recursive ordering
well-defined.

For connected $P$ with prime factors $D_1,\ldots,D_s$, define
\[
 \Tow(P)=F_{D_1}+\cdots+F_{D_s},\qquad \Tow(\one)=0.
\]
The canonical theorem gives
\begin{equation}\label{b:E:Tow-laws}
\begin{aligned}
 \Tow(P\times Q)&\cong\Tow(P)+\Tow(Q),\\
 \Tow(\Con(P,Y))&\cong\Tow(P)\times Y.
\end{aligned}
\end{equation}
For the second identity, split $Y$ into its connected components and
apply the last assertion of Theorem~\ref{b:T:coordinates} to each prime
factor of $P$. If $P$ is nontrivial, $\Tow(P)$ is an actual ordered set,
and \eqref{b:E:coordinate-height} gives
\begin{equation}\label{b:E:Tow-height}
 \height(\Tow(P))<\height(P).
\end{equation}
If $P$ is bounded, every map $Y\to P$ lies between the constant bottom
and top maps, so $P^Y=\Con(P,Y)$.

\section{Uniqueness and the source of the coordinates}
\index{canonical!coordinates!uniqueness}\index{canonical!coordinates!source}
Theorem~\ref{b:T:coordinates} is the classical structural input, not an
induction assumption about exponential reconstruction. Its exponent
coordinate must be connected. Otherwise its component decomposition
would split $\Con(E_D,F_D)$ into at least two nontrivial Cartesian
factors.

For completeness, the uniqueness reduction can be made explicit. The
connected refinement theorem for $\Con$, as stated in
\cite[Theorem~4]{FarleyPreprint}, says that an equality
$\Con(E,F)\cong\Con(E',F')$ with connected entries yields connected
$H,U,V,Z$ such that
\[
 E\cong\Con(H,U),\quad E'\cong\Con(H,V),\quad
 F\cong V\times Z,\quad F'\cong U\times Z.
\]
If $E,E'$ are absolutely $\Con$-indecomposable, then
$U\cong V\cong\one$. Hence $E\cong E'$ and $F\cong F'$.
The coordinate transformation under a~connected exponent follows from
\[
 \Con(\Con(E,F),Y)\cong\Con(E,F\times Y)
\]
and this uniqueness. These are precisely the properties used below.
Existence and preservation of direct indecomposability are supplied by
the cited structural theorem.
\index{canonical!coordinates}\index{absolute!indecomposability}\index{Tow!construction}

\section{Why the coordinate sum lowers height}
\index{canonical!coordinates!height reduction}\index{height!reduction!coordinate sum}
If $P=D_1\times\cdots\times D_s$ is nontrivial, each canonical core
$E_{D_i}$ has positive height. Consequently
\[
 \height(F_{D_i})<|F_{D_i}|
 \leq |F_{D_i}|\height(E_{D_i})
 =\height(D_i)\leq\height(P).
\]
Taking the maximum over the summands of $\Tow(P)$ proves the strict
height drop. Notice that the coordinate sum may be disconnected even
when $P$ is bounded. Keeping that sum, rather than selecting only one
of its components, is essential for the cancellation identities.

\chapter{Two compatible auxiliary orders}\label{ch:orders}

We use two different orders on isomorphism types. The first selects
intervals and resolves the height ties in the finite case analysis. The
second gives a~strict-monotonicity theorem on componentwise contractible
inputs. The two roles must not be interchanged.
\index{auxiliary!ordering}\index{tie-breaking}
\section{Greatest unequal coordinates}
\index{auxiliary!ordering!greatest unequal coordinate}\index{lexicographic!comparison}
\begin{lemma}\label{L:lexicographic}
\index{lexicographic!comparison!lemma}
Let functions from a~chain to a~chain differ pairwise at only finitely
many positions. Compare two distinct functions at their greatest
unequal position. This defines a~linear order.
\end{lemma}
\begin{proof}
Totality and asymmetry are immediate. For $f<g<h$, let $c$ be the
greater of the greatest unequal positions for $f,g$ and for $g,h$.
Above $c$ all three agree. At $c$ one has
$f(c)\leq g(c)\leq h(c)$, with at least one strict inequality.
Thus $c$ is the greatest unequal position for $f,h$, and $f<h$.
\end{proof}
\section{The interval-selection order}
\index{auxiliary!ordering!interval selection}\index{maximum-height!interval}
We now fix the ordering used throughout the remaining proof. First
compare connected types by height, then cardinality, and finally by the
multiplicity at the last unequal prime factor in a~fixed enumeration of
prime types. Call this preliminary order $\leq_1$. Multiplication by a
connected type preserves strict comparisons: heights add, cardinalities
multiply, and the same factor-multiplicity vector is added.

Next order prime types by their canonical pairs $(E_D,F_D)$, comparing
both entries in $\leq_1$, with the core first. Finally, compare connected
types at their greatest unequal prime factor in this pair order. Denote
this last order by $\preceq$.

\begin{lemma}\label{b:L:selection-order}
\index{interval-selection!order!lemma}
Multiplication by a~connected ordered set preserves strict
$\preceq$-comparisons. So does $P\mapsto\Con(P,Y)$ for fixed $Y$.
For bounded nontrivial $L$,
\begin{equation}\label{b:E:exponent-tie}
 \height(X)>\height(Y)\quad\Longrightarrow\quad L^Y\prec L^X.
\end{equation}
In particular, $L^{\times j}\prec L^X$ whenever $\height(X)>0$ and
$j\geq1$.
\end{lemma}
\begin{proof}
Products add the same prime multiplicities. For connected $Y$, the
canonical theorem replaces each pair $(E,F)$ by $(E,F\times Y)$.
This is an injective order-preserving replacement because multiplication
preserves $\leq_1$. For arbitrary $Y$, split it into connected components
and compare the corresponding products one factor at a~time.

For \eqref{b:E:exponent-tie}, take the greatest core occurring among the
prime factors of $L$, and let $b$ be the greatest height of a~coordinate
with that core. In $L^X$, the greatest coordinate height at this core is
$b+\height(X)$; in $L^Y$, it is $b+\height(Y)$. The preliminary order
compares these heights first, so a~prime of $L^X$ exceeds every prime
of $L^Y$ with that core. No smaller core can alter the comparison.
Taking $Y=j\one$ proves the last assertion.
\end{proof}

\section{A~recursive ordered-ring order}
\index{auxiliary!ordering!recursive ordered-ring order}\index{ordered!ring!recursive order}
The order $\triangleleft$ constructed here is distinct from $\preceq$.
It is used only for strict monotonicity, not for selecting the intervals
in the finite classification.
\begin{lemma}\label{b:L:ordered-ring}
\index{ordered!ring!ordering lemma}
There is a~total order $\triangleleft$ on finite ordered-set types that
is the restriction of a~total ordered-ring structure on the signed
component polynomials and has these properties:
\begin{equation}\label{b:E:Con-monotonicity}
\begin{aligned}
 P\triangleleft Q&\ \Longrightarrow\
       \Con(P,Y)\triangleleft\Con(Q,Y),\\
 Y\triangleleft Z&\ \Longrightarrow\
       \Con(P,Y)\triangleleft\Con(P,Z)
       \quad\text{if $P$ is nontrivial}.
\end{aligned}
\end{equation}
In the first line $P,Q$ are connected; in the second $P$ is connected.
All actual ordered sets represent positive elements of this ordered ring.
\end{lemma}
\begin{proof}
Fix a~total order of the absolutely $\Con$-indecomposable core types.
Compare prime types $D$ by their pairs $(E_D,F_D)$, first by the core
and then by the connected type $F_D$. Compare connected types by the
multiplicity at their greatest unequal prime factor. The unit $\one$
is smaller than every nontrivial connected type.

This is a~recursive definition by height. To compare prime types of
height at most $h$, \eqref{b:E:coordinate-height} requires only comparisons
of connected exponent coordinates of height below $h$. The resulting
prime order then orders connected products of height at most $h$.
The definitions agree at successive heights. Totality and transitivity
follow from lexicographic comparison and the greatest-unequal-factor
rule.

A~nonzero signed component polynomial is positive if the coefficient of
its greatest monomial is positive. The monomial order is invariant under
multiplication, because multiplication adds the same vector of prime
multiplicities. The leading monomial and coefficient of a~product are
the products of the leading monomials and coefficients. Thus this rule
defines a~totally ordered integral domain.

For a~core $E$ and connected $P$, let $\coord_E(P)$ be the sum of
$\poly(F_D)$ over the prime factors of $P$ with core $E$, with
multiplicity. Set $\coord_E(\one)=0$. The connected-type order is
obtained by finding the greatest core at which these polynomials differ,
and comparing the two polynomials there. The canonical theorem gives
\begin{equation}\label{b:E:coordinates-product}
 \coord_E(\Con(P,Y))=\coord_E(P)\poly(Y).
\end{equation}
For connected $Y$ this is the coordinate multiplication rule. For
arbitrary $Y$, use its component decomposition and
\eqref{b:E:Con-laws}.

If $P\triangleleft Q$, multiply the nonzero difference at the greatest
unequal core by the positive polynomial $\poly(Y)$. Its sign is
unchanged, proving the first line of \eqref{b:E:Con-monotonicity}.
For the second, use the greatest core occurring in the nontrivial $P$.
Its nonzero positive coordinate polynomial multiplies
$\poly(Z)-\poly(Y)$ and preserves its sign.
\end{proof}

\chapter{Selected intervals and height reduction}\label{b:S:selection}
Define $\Max(X)$ to be the $\preceq$-greatest isomorphism type among
intervals of height $\height(X)$ in $X$. This is intrinsic, bounded,
and unique up to isomorphism.

\begin{lemma}\label{b:L:Max}
\index{Max!construction!lemma}
For arbitrary ordered sets,
\begin{equation}\label{b:E:Max-laws}
 \Max(X\times Y)\cong\Max(X)\times\Max(Y),\qquad
 \Max(X^Y)\cong\Max(X)^Y.
\end{equation}
\end{lemma}
\begin{proof}
Every interval in a~product is a~product of intervals. The height formula
shows that a~maximum-height interval uses maximum-height intervals in
both factors. Product compatibility of $\preceq$ selects the two
$\Max$ factors.

For the exponent assertion, let $[f,g]$ be an interval of maximum
height in $X^Y$, and write $h=\height(X)$. The coordinatewise bound gives
\[
 |Y|h=\height([f,g])
 \leq\sum\setm{\height([f(y),g(y)])}{y\in Y}\leq |Y|h.
\]
Thus each $[f(y),g(y)]$ has height $h$. Its bottom is minimal in $X$
and its top maximal, since otherwise a~maximum chain could be extended.
Comparable values of $f$ are therefore equal, and the same holds for
$g$. Both maps are constant on each connected component $Y_i$ of $Y$.
Consequently
\[
 [f,g]\cong\prod\setm{I_i^{Y_i}}{1\leq i\leq\comp(Y)},
\]
where each $I_i$ is a~maximum-height interval of $X$. Conversely, every
such product is an interval of maximum height. The intervals $I_i$ are
bounded, so $I_i^{Y_i}=\Con(I_i,Y_i)$. Exponent and product compatibility
in Lemma~\ref{b:L:selection-order} select $I_i=\Max(X)$ in every factor.
\end{proof}

Put
\[
 \Der(X)=\Tow(\Max(X)),\qquad \Red(Q)=\Max(\Tow(Q)),
\]
where $\Red$ is applied only to bounded ordered sets of positive height.
The preceding identities imply
\begin{equation}\label{b:E:Der-laws}
 \Der(X\times Y)\cong\Der(X)+\Der(Y),\qquad
 \Der(X^Y)\cong\Der(X)\times Y.
\end{equation}
There is no additivity assertion for $\Der(X+Y)$.

For a~positive-height input $A$, abbreviate
\begin{equation}\label{b:E:LRS}
 L=\Max(A),\quad R=\Tow(L),\quad S=\Max(R),\quad J=\Tow(S).
\end{equation}
The ordered set $R$ is actual. If $J\ne0$, then
\begin{equation}\label{b:E:separation}
 \height(J)<\height(S)=\height(R)<\height(L)=\height(A).
\end{equation}

\begin{lemma}[Reduction and terminal recovery]\label{b:L:reduction}
\index{terminal!recovery!reduction lemma}
For an ordered set $D$ and an integer $k\geq1$,
\begin{equation}\label{b:E:reductions}
\begin{aligned}
 \Red(\Max(A^D))&\cong S\times\Max(D),\\
 \Red(S\times L^D)&\cong S\times\Max(D),\\
 \Red(S\times L^{\times k})&\cong S.
\end{aligned}
\end{equation}
For known $k$, the isomorphism type $S\times L^{\times k}$ determines
$L$, $R$, and $S$.
\end{lemma}
\begin{proof}
For the first identity, apply $\Tow$ to $\Max(A^D)=L^D$ to get
$R\times D$, and then apply $\Max$. For the second, the $\Tow$-image is
$J+(R\times D)$. Its second summand has greater height than the first
when $J$ is present, by \eqref{b:E:separation}. Its selected interval is
$S\times\Max(D)$. Taking $D=k\one$ proves the third identity.

From $S\times L^{\times k}$, recover $S$ by the third identity, cancel
it, take the unique Cartesian $k$th root, and compute $R=\Tow(L)$.
Every application of $\Red$ has a~positive-height argument. We never
apply $\Red$ to a~terminal $S$ unless its positive height is separately
known.
\end{proof}

\section{What the construction determines}
\index{terminal!recovery}\index{derived!sum}\index{selected!interval!determination}
The choice of $\Max(X)$ is intrinsic up to isomorphism after the
auxiliary order has been fixed. A~particular interval inside a~chosen
representation of $X$ need not be unique. Every statement below uses
only its isomorphism type, so this distinction causes no ambiguity.
\index{maximum-height!interval}\index{Max!construction}\index{derived!sum}\index{height!reduction}\index{terminal!recovery}

The formula $\Der(XY)=\Der(X)+\Der(Y)$ resembles a~logarithm. The
formula $\Der(X^Y)=\Der(X)Y$ is its exponential counterpart. It is
important that $\Der(X+Y)$ is not the sum of the derived sums: one first
selects a~maximum-height interval of the disjoint sum. Consequently,
a~recovery argument involving a~mixed sum must prove which interval
is selected, or must prove that different selections cannot produce
the same invariants.

\begin{example}
Suppose $L$ is a~nontrivial bounded interval and $R=\Tow(L)$ is an
antichain. Then $S=\Max(R)=\one$ and $J=0$. The terminal recovery
$SL^{\times k}\mapsto S$ remains valid: its final output is the
one-element ordered set. No additional application of $\Red$ is
permitted or needed. This explains why stopping at a~known terminal
stage is part of every reduction argument.
\end{example}

\chapter{Inputs reconstructed at every occurrence count}\label{b:S:K}
\begin{definition}\label{b:D:contractible}
\index{contractibility!definition}
A~connected ordered set $P$ is \emph{contractible} if there is a~finite
sequence of order-preserving self-maps starting at $\id_P$, ending at a
constant map, and with consecutive maps comparable in the pointwise
order. Let $\K$ be the class whose connected components are contractible,
and let $\OO$ be its complement among finite ordered sets.
\end{definition}
This is the only meaning of contractible used here. A~least or a~greatest
element makes a~connected ordered set contractible.

\begin{lemma}\label{b:L:K-closure}
\index{contractibility!closure}
The class $\K$ is closed under sums, products, exponentiation, and
retracts. If $P$ is contractible, then $P^Y$ is contractible and equals
$\Con(P,Y)$ for every $Y$.
\end{lemma}
\begin{proof}
Postcomposition with a~contraction sequence for $P$ contracts $P^Y$:
the maps $f\mapsto h_i\circ f$ start at its identity, end at a~constant
self-map, and consecutive maps are comparable. This also proves
connectedness and equality with $\Con(P,Y)$.
Products of contractible ordered sets contract by changing one coordinate
at a~time. Sums preserve the componentwise condition.

If $X=P_1+\cdots+P_a\in\K$ and $Y=Y_1+\cdots+Y_b$ is the component
decomposition of $Y$, then
\begin{equation}\label{b:E:power-components}
 X^Y\cong\prod\setm{
      P_1^{Y_j}+\cdots+P_a^{Y_j}}{1\leq j\leq b}.
\end{equation}
A~map on a~connected $Y_j$ lands in one component of $X$, and different
target components cannot be joined by a~comparability path. Every
component of every factor on the right is contractible. The product
assertion proves closure under exponentiation.

Suppose $i:Q\to P$ and $r:P\to Q$ are order-preserving and
$r i=\id_Q$. A~contraction $h_i$ of $P$ induces a~contraction
$r h_i i$ of $Q$. For a~retract of a~disconnected member of $\K$,
restrict to a~component of the retract and the component containing
its image. The retraction sends that whole component back into the
chosen component. This proves retract closure.
\end{proof}

\begin{lemma}\label{b:L:K-monotone}
\index{strict!monotonicity!contractible inputs}
For fixed $Y$, full exponentiation $X\mapsto X^Y$ is strictly increasing
on $\K$ in the order $\triangleleft$. For fixed $X\in\K$ of positive
height, $Y\mapsto X^Y$ is strictly increasing on all ordered sets.
\end{lemma}
\begin{proof}
For connected $Y_j$, the replacement of a~contractible connected type
$P$ by $P^{Y_j}=\Con(P,Y_j)$ is injective and strictly increasing by
Lemma~\ref{b:L:ordered-ring}. Applied to component polynomials, it
preserves the sign at the greatest unequal monomial. Thus
$X\triangleleft X'$ gives a~strict comparison between the corresponding
sums in each factor of \eqref{b:E:power-components}. Products of positive
factors preserve these strict comparisons, proving the first assertion.

For the second, let $P$ be the greatest connected component type of $X$.
It is nontrivial because $X$ has positive height. The greatest component
type of $X^Y$ is $P^Y=\Con(P,Y)$, by
\eqref{b:E:power-components}. If $Y\triangleleft Z$, the second line of
\eqref{b:E:Con-monotonicity} strictly increases this greatest component
type. Hence $X^Y\triangleleft X^Z$.
\end{proof}

\begin{theorem}\label{b:T:K}
\index{reconstruction!contractible inputs}
Every exponential term is strictly increasing on $\K$. For arbitrary
inputs $A,B$, if $t(A)\cong t(B)$ and either input belongs to $\K$,
then $A\cong B$. Moreover, $t(A)\in\K$ if and only if $A\in\K$.
\end{theorem}
\begin{proof}
For positive-height inputs in $\K$, use induction on the term.
If $t=u^v$ and $A\triangleleft B$, the induction hypothesis and
Lemma~\ref{b:L:K-monotone} give
\[
 u(A)^{v(A)}\triangleleft u(B)^{v(A)}
                     \triangleleft u(B)^{v(B)}.
\]
The smaller values lie in $\K$ and have positive height by closure and
the constant-map embeddings. For antichains, term values are antichains,
and their cardinalities are the corresponding integer exponentials,
strictly increasing in the input cardinality. An antichain is smaller
than every positive-height ordered set in $\triangleleft$, and this
remains true of their term values. Thus the same induction covers all
inputs in $\K$.

Every input $A$ is a~retract of $t(A)$. For $t=u^v$, constant maps embed
$u(A)$ in $u(A)^{v(A)}$, and evaluation at a~fixed element of $v(A)$
retracts onto $u(A)$. Compose with the inductive retraction onto $A$.
Closure under exponentiation and retracts now gives the asserted
equivalence of membership in $\K$. If either input is in $\K$ and the
term values are isomorphic, both inputs are in $\K$. Strict increase
forces them to be isomorphic.
\end{proof}

Every connected ordered set with at most three elements has a~least or
a~greatest element. Therefore
\begin{equation}\label{b:E:O-inputs}
 A\in\OO\quad\Longrightarrow\quad |A|\geq4,
       \qquad\height(A)>0.
\end{equation}
No term-size bound was used in this section.

\section{Scope of the unrestricted theorem}
\index{reconstruction!unrestricted theorem}\index{contractibility!scope}
Theorem~\ref{b:T:K} places no bound on the occurrence count. It includes
all bounded inputs and all disjoint sums of contractible connected
inputs. It also reflects membership in the class: an input outside
$\K$ cannot acquire an isomorphic term value to an input inside $\K$.
Thus later proofs may restrict both inputs to $\OO$ without losing
mixed pairs of inputs.
\index{contractibility}\index{componentwise!contractible inputs}\index{strict!monotonicity}

The word contractible is used only in the explicit pointwise-comparable
self-map sense of Definition~\ref{b:D:contractible}. No equivalence with
another topological convention is required. Retractions, contractions
of products, and contractions by postcomposition have all been proved
in this sense.

\chapter{Cardinality bounds and retractions}\label{ch:bounds}
\section{Basic bounds}\label{b:S:bounds}
Use the following terms throughout:
\begin{equation}\label{b:E:small-terms}
\begin{aligned}
 P&=A^A,& F&=A^{\Ap2},& G&=A^P,\\
 H_1&=A^{\Ap3},& H_2&=A^{A\times P},&
 H_3&=A^F,& H_4&=A^G.
\end{aligned}
\end{equation}
Write $N=|A|$, $p=|P|$, $f=|F|$, $g=|G|$, $h_i=|H_i|$, and
$\ell=\height(A)>0$.

\begin{lemma}\label{b:L:bounds}
\index{cardinality!bounds!basic lemma}
For positive-height $A$ and $c=\comp(A)$,
\begin{equation}\label{b:E:basic-bounds}
\begin{aligned}
 |A^Y|&\geq |Y|+N-1,\\
 p&\geq N^c+N^{c-1},\\
 \comp(P)&\geq c^c.
\end{aligned}
\end{equation}
If $A\in\OO$, then
\begin{equation}\label{b:E:noncontractible-bound}
 \comp(P)\geq c^c+c^{c-1}.
\end{equation}
If $N\geq4$, then
\begin{equation}\label{b:E:strong-cardinality}
 p\geq N^2+N,\qquad f\geq p+N-1,\qquad g\geq p+N-1.
\end{equation}
Finally, $P$ is a~retract of $G$, so $\comp(G)\geq\comp(P)$.
\end{lemma}
\begin{proof}
An ordered set $Y$ has at least $|Y|+1$ upsets, as seen from a~linear
extension. Maps into a~fixed two-element chain in $A$, followed by the
other $N-2$ constant maps, give the first inequality.

For the second, choose a~nontrivial connected component $Q$ of $A$.
On $Q$ use any of the $N$ constant maps or its inclusion into $A$.
On every other component use an arbitrary constant map. These give
$(N+1)N^{c-1}$ distinct endomorphisms.

An assignment of a~target component to each source component gives a
componentwise-constant map. Different assignments cannot lie in the same
component of $P$, proving $\comp(P)\geq c^c$. For
\eqref{b:E:noncontractible-bound}, choose a~noncontractible component $Q$.
In addition, map $Q$ identically to itself and assign constant target
components to all other source components. These $c^{c-1}$ assignments
give distinct components not among the previous $c^c$: a~path from the
identity on $Q$ to a~constant would contract $Q$.

We prove $p\geq N^2+N$ for $N\geq4$. For disconnected $A$, the bound
$(N+1)N^{c-1}$ already suffices. Suppose $A$ is connected. Let $q$ be
the number of strictly comparable pairs and $u$ the number of upsets
other than the empty set and the whole set. The $qu$ two-valued maps,
the $N$ constants, and the identity are distinct, so
\[
 p\geq N+qu+1.
\]
If $A$ is a~chain, $q=N(N-1)/2$, $u=N-1$, and $qu\geq N^2$ for
$N\geq4$. If $A$ is not a~chain, $u\geq N$: the suffix upsets of a
linear extension form a~chain of $N+1$ upsets, whereas two incomparable
principal upsets force at least one more. When $q\geq N$, this suffices.

The remaining case is $q=N-1$. Its connected comparability graph is a
tree, so there is no three-element chain, which would give a~triangle.
All elements are minimal or maximal. If there are $a$ minimal and $b$
maximal elements, then $a+b=N$. All subsets of the maximal elements
are upsets, as are their full set together with any subset of the
minimal elements. These families overlap once. Thus there are at least
$2^a+2^b-1\geq2N-1$ upsets, so $u\geq2N-3\geq N+1$.
It follows that
$p\geq N+(N-1)(N+1)+1=N^2+N$.

Since $F\cong P^A$, the first inequality of
\eqref{b:E:basic-bounds}, with base $P$, gives $f\geq p+N-1$.
With base $A$ and exponent $P$, it gives $g\geq p+N-1$.

For the retraction, choose $a_0\in A$. Evaluation
$e:P\to A$, $e(v)=v(a_0)$, retracts the constant-map embedding
$c:A\to P$. Precomposition gives an embedding
$P\to G$, $v\mapsto v\circ e$, with retraction
$G\to P$, $w\mapsto w\circ c$. A~retraction is surjective on connected
components, proving the last assertion.
\end{proof}

\begin{lemma}[Disconnected-input inequality]\label{b:L:disconnected}
If $A$ is disconnected and has positive height, then
\begin{equation}\label{b:E:g-f}
 |A^{A^A}|>|A^{A\times A}|.
\end{equation}
If also $N\geq4$, then
\begin{equation}\label{b:E:disconnected-dominance}
 f>N^3,
 \qquad f>Np,
 \qquad g>f.
\end{equation}
Consequently, on such inputs, both $H_3$ and $H_4$ have greater height
than $H_1,H_2$, and $H_4$ has greater height than $H_3$.
\end{lemma}
\begin{proof}
Partition the connected components of $A$ into two nonempty groups.
Map $(a,b)\in A\times A$ to the endomorphism constant at $a$ on the
first group and at $b$ on the second. Evaluation at one point in each
group gives a~retraction
\[
 r:P\longrightarrow A\times A.
\]
By \eqref{b:E:basic-bounds}, $p\geq N^2+N>N^2$, so $r$ is not
injective. Precomposition with $r$ injects $A^{A\times A}$ into $A^P$.
Choose two different elements of $P$ with the same $r$-image. An upset
of $P$ distinguishes them. Its characteristic map into a~fixed
comparable pair of $A$ is order-preserving and is not constant on this
fiber, so it is not in the image of precomposition. This proves the
strict inequality.

If $c=\comp(A)\geq2$, maps $A\to P$ constant on each component give
$f=|P^A|\geq p^c\geq p^2$. Since $p\geq N^2+N$, this exceeds both
$N^3$ and $Np$. The final comparisons follow from
\eqref{b:E:heights} and $g>f$.
\end{proof}

\section{Additional components and sharper bounds}\label{t:S:bounds}
\begin{lemma}[Extra components and an actual retract]\label{t:L:retract}
\index{retract!extra components}
Suppose $A\in\OO$ is a~retract of both $X$ and $Y$. Put
$x=\comp(X)$ and $y=\comp(Y)$. Then
\begin{equation}\label{t:E:extra}
 \xp Xy+x^{y-1}\one\text{ is a~retract of }X^Y,
 \qquad \comp(X^Y)\geq x^y+x^{y-1}.
\end{equation}
\end{lemma}
\begin{proof}
The maps constant on each connected component of $Y$ form a~copy of
$X^{\times y}$. They occupy exactly $x^y$ connected components of
$X^Y$: an assignment of a~target component to each source component
is unchanged along a~comparability path, and constant maps with the
same assignment are connected by paths of constant maps.

Let $Q$ be a~noncontractible component of $A$. Write the retractions
as $r_X:X\to A$ and $r_Y:Y\to A$, with sections $i_X,i_Y$. The copies
of $Q$ lie in components $X_0,Y_0$. The map
$i_Xr_Y:Y_0\to X_0$ is not connected to a~constant map in $X_0^{Y_0}$.
Otherwise composition on the left with $r_X$ and on the right with
$i_Y|_Q$ would join $\id_Q$ to a~constant map, contracting $Q$.
Assign an arbitrary target component to each of the other $y-1$
source components, using a~constant map there. We obtain $x^{y-1}$
distinct components of $X^Y$ outside the components containing the
componentwise constant maps.

The asserted retract is not just a~cardinality estimate. On every
component containing componentwise constant maps, evaluation at one
fixed point of each source component retracts onto that copy of
$X^{\times y}$. Choose a~point in each of the $x^{y-1}$ extra
components and collapse its whole component to that point. Collapse
any remaining component to a~fixed point of the proposed retract.
There are no order relations between different components, so these
maps together are order-preserving and fix the proposed retract.
\end{proof}

\begin{lemma}[Strict precomposition]\label{t:L:strict}
\index{strict!precomposition!lemma}
If $r:X\to Y$ is a~noninjective retraction and $A$ has positive height,
then $|A^X|>|A^Y|$.
\end{lemma}
\begin{proof}
Precomposition with $r$ embeds $A^Y$ into $A^X$. Choose distinct
points in one fiber of $r$. An upset of $X$ separates them. Its
characteristic map into a~fixed comparable pair in $A$ is
order-preserving and is not constant on that fiber, so it is not in
the image of the embedding.
\end{proof}

We retain the terms and cardinalities in~\eqref{b:E:small-terms}. Put $\ell=\height(A)>0$.

\begin{lemma}\label{t:L:bounds}
\index{component!bounds!lemma}
For $A\in\OO$, $N\geq4$, and $c=\comp(A)$,
\begin{equation}\label{t:E:bounds}
\begin{aligned}
 p&\geq N^2+N+1,& f&\geq p+N,&g&\geq p+N,\\
 g&>Np,&h_3&>Ng,&h_4&>Nf,\\
 h_1&\geq f+N-1,&h_2&\geq p^2+p,&h_2&>f,\quad h_2>g.
\end{aligned}
\end{equation}
Also $h_3\geq f+N-1$, $h_4\geq g+N-1$, and every term with at least
three occurrences has more than $p$ elements. Finally,
\begin{equation}\label{t:E:componentbound}
 \comp(P)\geq c^c+c^{c-1},\qquad p\geq(N+1)N^{c-1}.
\end{equation}
If $c\geq2$, then
\begin{equation}\label{t:E:disconnected}
 f>p^2,\qquad g\geq N^2f,\qquad h_1\geq f^2,\qquad p>2N^2.
\end{equation}
If $c\geq3$, then $g\geq Nh_1>h_1$.
\end{lemma}
\begin{proof}
First suppose $A$ is connected. A~connected ordered set whose
comparability graph is a~tree is contractible: a~leaf can be folded
onto its unique comparable neighbor by an order-preserving retraction,
and induction removes all but one point. Thus our graph has at least
$N$ edges, that is, at least $N$ strictly comparable pairs. Also $A$
is not a~chain. It has at least $N+2$ upsets: a~linear extension gives
$N+1$ nested upsets, and incomparable principal upsets show that these
cannot be all the upsets. Consequently there are at least $N$
nontrivial upsets. Each choice of a~strict comparable pair and
nontrivial upset gives a~distinct two-valued map. Adding the $N$
constants and the identity gives $p\geq N^2+N+1$.

If $A$ is disconnected, on a~noncontractible component use its
inclusion and one two-valued nonconstant map, in addition to the $N$
constant maps. Choose constants independently on the other components.
This gives $p\geq(N+2)N^{c-1}\geq N^2+N+1$. The weaker second
inequality in~\eqref{t:E:componentbound} also follows, and its first is
Lemma~\ref{t:L:retract} with $X=Y=A$.

For any positive-height base $X$, counting maps into a~comparable pair
and then the other constants gives
\begin{equation}\label{t:E:elementarymaps}
 |X^Y|\geq |Y|+|X|-1.
\end{equation}
For $F=P^A$, add to these maps the constant-map embedding $A\to P$,
which has $N\geq4$ values. For $G=A^P$, add an evaluation map $P\to A$,
which also has $N$ values. These give $f,g\geq p+N$.

Lemma~\ref{t:L:retract} gives $A+\one$ as a~retract of $P$. It is a
proper retract because $p>N+1$. Strict precomposition therefore gives
$g>Np$. Applying the lemma with $(X,Y)=(P,A)$ gives $P+\one$ as
a~proper retract of $F=P^A$, so $h_3>Ng$. Since $P$ has at least
two components, applying the lemma with $(X,Y)=(A,P)$ gives
$\Ap2+\one$ as a~retract of $G$. It is proper since $g>Np>N^2+1$,
so $h_4>Nf$. When a~larger Cartesian power or a~larger number of
isolated points occurs in the lemma, project to the stated power and
collapse the extra isolated points. These are still retractions.

The bounds on $h_1,h_3,h_4$ not yet proved follow from
\eqref{t:E:elementarymaps}, using $H_1=F^A$. Since $P\in\OO$ and
$H_2\cong P^P$, the bound on $p$ applied to $P$ gives
$h_2\geq p^2+p+1$. Both $\Ap2$ and $P$ are proper retracts of $AP$;
strict precomposition gives $h_2>f,g$. More generally the root forest
of a~term with at least three occurrences has $A$ as a~proper retract,
so its value has more than $p$ elements.

For disconnected $A$, Lemma~\ref{t:L:retract} gives a~retract
$\xp Pc+u^{c-1}\one$ of $F$, where $u=\comp(P)$. Thus $f>p^2$.
It also gives $\Ap2+2\one$ as a~retract of $P$, whence $g\geq N^2f$.
Componentwise constant maps embed $F^{\times c}$ into $F^A=H_1$, so
$h_1\geq f^2$. For $c\geq3$, the retract $\Ap3+\one$ of $P$ gives
$g\geq Nh_1$.

It remains to prove $p>2N^2$ when $c=2$; the case $c\geq3$ follows
from $p\geq N^c$. Write the component sizes as $s,N-s$. There are
at least $N-2$ strict comparable pairs in $A$. On a~component of size
$s$, its at least $s-1$ nontrivial upsets therefore give at least
$N+(N-2)(s-1)=2+(N-2)s$ maps into $A$. Hence
\[
\begin{aligned}
 p&\geq[2+(N-2)s][2+(N-2)(N-s)]\\
  &\geq N(N^2-3N+4)>2N^2.
\end{aligned}
\]
Here $s(N-s)\geq N-1$, and $N\geq5$ because a~noncontractible component
has at least four elements and another component is present.
\end{proof}

\index{cardinality!bounds}\index{component!bounds}\index{strict!precomposition}\index{disconnected!inputs}
The sharper bounds are used in the twelve-occurrence comparison
matrices. The earlier bounds also remain useful because they give a
shorter comparison classification through eleven. Accordingly, a~case
called exceptional in the earlier classification may become regular
under the sharper inequalities. The classifications are proof
assignments, not intrinsic partitions of the ordered-set operations.

\chapter{Normal forms and strong reconstruction}\label{b:S:normal}
A~normalized term is represented by a~rooted unordered tree: the root
represents the final base $A$, and its child trees represent the factors
in its root exponent. Thus a~nontrivial normal form has value
\begin{equation}\label{b:E:normal-form}
 t(A)=A^{D_t(A)},\qquad
 D_t(A)=t_1(A)\times\cdots\times t_s(A).
\end{equation}
A~child with one vertex has value $A$. More generally, a~\emph{forest}
is a~Cartesian product of normalized terms; its weight is the sum of
their occurrence counts. A~forest is \emph{pure} when all its factors
are $A$.

Normalization is recursive. To normalize $u^v$, retain the root of the
normal form of $u$ and add the normal form of $v$ as one more child.
The exponential law proves \eqref{b:E:normal-form}. Interchange permutes
adjacent children along the left-associated exponentiation spine.
Conversely, choosing any order of the children gives a~left-associated
representative. Induction on the child trees shows that two terms have
the same rooted unordered tree exactly when they are equivalent under
interchange. This proves existence and uniqueness of these normal forms.

For $A\in\OO$ and a~nontrivial term, put
\[
 V_t(A)=\Red(\Max(t(A)))\cong S\times\Max(D_t(A)).
\]

\begin{definition}\label{b:D:strong}
\index{strong!reconstruction!definition}
A~fixed nontrivial term is \emph{strong on $\OO$} if the following
two uniqueness assertions hold for inputs in this class:
\emph{(S1)} $V_t(A)$ determines $R=\Der(A)$;
\emph{(S2)} the pair $(R,D_t(A))$ determines $A$.
\end{definition}
All terms and integer parameters are known in these assertions.
There is no efficiency requirement: one may enumerate candidates, and
the proofs establish uniqueness of the required invariant or input.
A~strong term is reconstructing on $\OO$, because its value gives
$V_t$ and then $R$, and cancellation in
\[
 \Der(t(A))\cong R\times D_t(A)
\]
recovers $D_t(A)$. Theorem~\ref{b:T:K} then handles all remaining inputs.

\begin{proposition}[Uniform extension]\label{b:P:uniform}
\index{uniform!extension!proposition}
The term $A^{\Ap k}$ is strong on $\OO$ for every $k\geq1$.
If $u$ is strong on $\OO$, so is
\begin{equation}\label{b:E:uniform}
 A^{\,\Ap r\times u(A)^{\times s}}
 \qquad(r\geq0,\ s\geq1).
\end{equation}
\end{proposition}
\begin{proof}
For a~pure exponent, $V_t=S\times L^{\times k}$ recovers $R$ by terminal
recovery, and $D_t=A^{\times k}$ recovers $A$ by root uniqueness.

Write $u(A)=A^{D_u(A)}$. For \eqref{b:E:uniform},
\[
 V_t\cong S\times L^{\times r}\times(L^{D_u(A)})^{\times s}.
\]
Its $\Tow$-image is $J+rR+s(R\times D_u)$. The forest $D_u$ has positive
height, so $R\times D_u$ has greater height than the other summands.
Therefore
\[
 \Red(V_t)\cong S\times\Max(D_u)=V_u.
\]
Property (S1) for $u$ recovers $R$.

For (S2), compute
\[
 \Der(D_t)\cong R\times(r\one+sD_u).
\]
Cancel $R$, remove exactly $r$ isolated points, and divide every
component multiplicity by $s$. These operations uniquely recover $D_u$,
even when it has isolated points. Property (S2) for $u$ now recovers $A$.
\end{proof}

\section{The meaning of strong reconstruction}
\index{strong!reconstruction!definition}\index{reconstruction!strong}
The two assertions in Definition~\ref{b:D:strong} can be made relative
to any specified class of positive-height inputs. Unless explicitly
stated otherwise, strong reconstruction below refers to $\OO$.
A~statement of ordinary reconstruction for a~smaller term is weaker
than this assertion and cannot be substituted for it in the uniform
extension proof.
\index{strong!reconstruction}\index{uniform!extension}\index{normal!form!rooted tree}

The first assertion identifies $R$ from a~selected-interval invariant.
The second identifies the input once $R$ and the entire root exponent
are known. The intervening cancellation is legitimate only because
the first assertion has already recovered $R$.

\section{Mixed roots and whole-exponent recovery}
\index{mixed!root!recovery}\index{whole-exponent!recovery}\index{summand!separation}
After grouping root leaves, a~nonuniform root can be written
\begin{equation}\label{t:E:mixed}
 E_{r,C}(A)=A^{\,\Ap r\times A^{C(A)}},\qquad
 C=E_1+\cdots+E_m,
\end{equation}
where $m\geq2$, the $E_i$ are positive-weight forests, and at least two
nontrivial child types are distinct. Repetitions are retained. If the
term has $n$ occurrences, then
\begin{equation}\label{t:E:budget}
 r+m+\sum\setm{\wt(E_i)}{1\leq i\leq m}=n-1.
\end{equation}
Indeed, the outer base contributes one occurrence; each of the $r$
root leaves contributes one; and the nontrivial child $A^{E_i}$
contributes $1+\wt(E_i)$.

For $D=\Ap r\times A^C$, the identities of
Chapter~\ref{b:S:selection} give
\begin{equation}\label{t:E:VZ}
\begin{aligned}
 V&=S L^{\times r}L^C,\\
 Z_0=\Red(V)&=S\Max(C),\\
 \Der(D)&=R(r\one+C).
\end{aligned}
\end{equation}
\begin{equation}\label{t:E:wholeC}
 (R,D)\quad\text{determines the entire ordered set }C.
\end{equation}
To verify the last assertion, cancel $R$ in the last line of
\eqref{t:E:VZ}, then remove exactly $r$ isolated points. Any isolated
points originally belonging to $C$ remain. The first reduction is
valid because every $E_i$ has positive height; its $RC$ summand
therefore dominates $J+rR$ in $\Tow(V)$.
\index{mixed!root}\index{whole-exponent!recovery}\index{summand!separation}

This calculation does not identify any individual displayed $E_i$.
Subsequent arguments use $C$ as an ordered set, and either compute its
selected interval or prove uniqueness across its possible selections.

\chapter{Forest recovery and reduction signatures}\label{t:S:forests}
For a~forest written as
\[
 D=\Ap a\prod\setm{A^{E_i}}{1\leq i\leq m},
 \qquad K(D)=a\one+\sum\setm{E_i}{1\leq i\leq m},
\]
we have $\Der(D)=RK(D)$. If $m>0$, the next reduction of
$S\Max(D)$ is $S\Max(E_1+\cdots+E_m)$; the $a$ isolated points in
$K(D)$ have smaller height. A~forest consisting entirely of copies
of $A$ is called \emph{pure}.

The elementary map bound also gives, by induction,
$|t(A)|\geq N+|t|-1$ for every term $t$. If the root forest has child
occurrence counts $n_1,\ldots,n_s$, its cardinality is at least
$\prod\setm{(n_i+1)}{1\leq i\leq s}\geq1+\sum\setm{n_i}{1\leq i\leq s}$.
The exponent height formula consequently gives
$\height(D)\geq\wt(D)\ell$ for every forest $D$. This justifies, in
particular, comparisons in which a~weight-one forest is discarded.

\begin{lemma}[The eight-weight forest lemma]\label{t:L:forest}
\index{forest!eight-weight lemma}
For every fixed positive-weight forest $D$ of weight at most eight,
the following assertions hold on $\OO$.
\begin{enumerate}
\item A~fixed sequence of reductions takes $S\Max(D)$ to
$SL^{\times k}$, for a~known positive integer $k$. Thus $S\Max(D)$
determines $R$.
\item If $D$ is not pure, $(R,\Der(D))$ determines $A$.
\item The pair $(R,D)$ determines $A$.
\end{enumerate}
\end{lemma}
\begin{proof}
Induct on $\wt(D)$. A~pure forest is a~Cartesian power of $A$, so
terminal recovery and Cartesian roots prove the applicable assertions.
For a~nonpure forest, use the displayed decomposition. If $m=1$, the
next forest is the known smaller forest $E_1$.

Suppose $m\geq2$. Then
$\sum\wt(E_i)\leq8-a-m\leq6$. Every comparison needed to select a
summand is fixed. A~weight-one forest loses to every larger forest.
At weight two, $P$ beats $\Ap2$. A~weight-four forest other than
$\Ap4$ beats $P$; $P$ also beats $\Ap4$, including the tie at $N=4$.
For weight three the order is
\[
 \Ap3\ <\ AP\ <\ F\ <\ G,
\]
since the height coefficients are $3,N+1,N^2,p$. The remaining
comparisons with total weight at most six follow immediately from
these statements and the height formulas. In particular, a~selected
pure forest forces every competing forest to be pure. Let $D_1$ be
the known selected forest. It has smaller weight than $D$. Applying
the induction hypothesis to $D_1$ proves the first assertion.

For the second assertion, cancel $R$ from $\Der(D)$ and remove the
known $a$ isolated points. This recovers $C=\sum E_i$. If every $E_i$
is pure, $C$ is a~known nonconstant polynomial in $A$, and
\eqref{b:E:polynomial} recovers $A$. Otherwise $D_1$ is nonpure and
$\Der(C)=\Der(D_1)$. The second inductive assertion for $D_1$ now
recovers $A$. For $m=1$ the same argument uses $D_1=E_1$. Finally,
from $(R,D)$ one can calculate $(R,\Der(D))$ if $D$ is nonpure; the
pure case was already handled. This proves the third assertion.
\end{proof}

Starting with $D_0=D$, write the fixed selected-forest chain as
\begin{equation}\label{t:E:chain}
 D_0,D_1,\ldots,D_{d-1}=\Ap k.
\end{equation}
Its \emph{signature} is $\sig(D)=(d,k)$. The corresponding interval
chain is $S\Max(D_0),\ldots,SL^{\times k},S$. The final $S$ is not
included in the length $d$.

\begin{lemma}[Signature recognition]\label{t:L:signature}
\index{reduction!signature!recognition lemma}
Suppose the selected forest of a~known sum $C$ belongs to a~fixed list
of nonpure forests of weight at most eight, with pairwise distinct
signatures. Then $(R,\Der(C))$ determines which forest is selected
and determines $A$.
\end{lemma}
\begin{proof}
If the selected forest is $D_0$, cancel $R$ to obtain
$Q_0=\Der(C)/R=K(D_0)$. While the current $Q_j$ has positive height,
put $Q_{j+1}=\Der(Q_j)/R$. These divisions are legitimate product
cancellations of known types. The selection rule gives
$Q_j=K(D_j)$, so the first height-zero value is $k\one$ at
$j=d-1$. The observed number of stages and the number of isolated
points therefore identify $(d,k)$, and hence $D_0$. The second part
of Lemma~\ref{t:L:forest}, applied to $D_0$, recovers $A$ from
$(R,\Der(D_0))$.
\end{proof}

For later computations it is useful to record a~general consequence
of equality between two selected-interval chains. Suppose
\[
 S_A\Max(D_0(A))=S_B\Max(E_0(B)),
\]
with chain lengths $d_A\leq d_B$. Put $\Delta=0$ if $d_A=d_B$ and
$\Delta=K(E_{d_A}(B))$ otherwise. The terminal equality at stage
$d_A$ gives $J_A=J_B+R_B\Delta$. Applying $\Tow$ at earlier stages
and canceling $J_B$ yields
\begin{equation}\label{t:E:cross}
 R_B\Delta+R_AK(D_i(A))=R_BK(E_i(B))\qquad(0\leq i<d_A).
\end{equation}
All terms here are actual sums or products, except that $\Delta=0$
means an omitted summand. Formula~\eqref{t:E:cross} also explains every
elimination in the exceptional-case tables below.

\begin{proposition}[Regular mixed roots]\label{t:P:regular}
\index{mixed!root!regular case}
Suppose a~mixed root satisfying~\eqref{t:E:budget} has a~fixed selected
forest $D_0$ in $C$, of weight at most eight. Suppose also that if
$D_0$ is pure, all summands of $C$ are pure. Then $E_{r,C}$ is strong.
\end{proposition}
\begin{proof}
Its first reduced invariant is $Z_0=S\Max(D_0)$, so
Lemma~\ref{t:L:forest} recovers $R$. Given $(R,D)$, equation
\eqref{t:E:wholeC} recovers $C$. If all its summands are pure, polynomial
injectivity recovers $A$. Otherwise $D_0$ is nonpure,
$\Der(C)=\Der(D_0)$, and the same forest lemma recovers $A$.
\end{proof}

\index{forest}\index{forest!weight}\index{forest!pure}\index{reduction!signature}\index{branch!recognition}\index{cross-branch!equality}
\section{Reading a~signature}
\index{reduction!signature!reading}\index{branch!recognition!signature}\index{cross-branch!equality}
A~signature records two integers, not the entire forest. It is used
only when the candidate forests for one fixed mixed sum have pairwise
different signatures. Equal signatures for unrelated forests are
common: $P$ and $AP$ both have signature $(2,1)$. The recognition
lemma never claims that a~signature classifies all forests.

The first strong assertion and the signature lemma also solve different
problems. The signature lemma starts with a~known $R$, so its divisions
by $R$ are legitimate. It proves the second strong assertion. To
recover $R$ initially, an exceptional proof must still rule out
collisions of the selected-interval chains without assuming $R$ is
known. Formula~\eqref{t:E:cross} is the bookkeeping identity for that
separate task.

\chapter{Polynomial perturbations and an unbounded family}\label{ch:perturbations}

Fix a~nonconstant polynomial
\begin{equation}\label{g:E:f}
 f(z)=c_1z+\cdots+c_kz^k,\qquad c_j\in\mathbb Z_{\geq0},\quad c_k>0.
\end{equation}
Its ordered-set value is the corresponding disjoint sum of Cartesian
powers. The zero constant coefficient is part of the hypothesis.
The next result recovers an input from a~whole sum, with no labelling
of its components by their source summands.
\index{polynomial!perturbation}\index{whole-sum!recovery}
\begin{lemma}[Whole-sum polynomial recovery]\label{t:L:polynomial}
\index{whole-sum!recovery!polynomial lemma}
Let $f(z)$ be a~fixed nonconstant polynomial with nonnegative integer
coefficients, zero constant coefficient, and degree $k$. Then
$(R,f(A)+P)$ determines every positive-height input $A$.
\end{lemma}
\begin{proof}
When $N\geq k$, the selected summand is $P$, including the tie at
$N=k$. Thus $\Der(f(A)+P)=RA$, and cancellation recovers $A$.
For $N<k$,
\[
 f(N)<|f(A)+P|\leq f(N)+N^N<f(N+1).
\]
Indeed, the leading term alone gives
$f(N+1)-f(N)\geq(N+1)^k-N^k>N^k\geq N^N$; the strict inequality
follows from $(1+1/N)^k>1+k/N>2$. These disjoint intervals identify
$N$ in the small range, and their values are below $f(k)$, whereas
the large range has values above $f(k)$.

Suppose two small-range inputs give isomorphic sums. Their cardinalities
are the same $N$, and their heights are the same $\ell$, since the
height of the sum is $k\ell$. Put $a=\poly(A)$ and $b=\poly(B)$.
If $a\ne b$, the polynomial $a-b$ is not a~nonzero constant: evaluating
each factor variable at the cardinality of its factor would otherwise
give $|A|-|B|\ne0$. Hence $\hdeg(a-b)\geq1$. The factorization
$f(a)-f(b)=(a-b)H(a,b)$, with nonnegative $H$, gives
\[
 \hdeg(f(a)-f(b))\geq1+(k-1)\ell>N\ell.
\]
But $f(a)-f(b)=\poly(B^B)-\poly(A^A)$, whose degree is at most
$N\ell$. This contradiction proves $A\cong B$ without separating
any summands.
\end{proof}

\section{An unbounded reconstructing family}
\index{reconstruction!unbounded family}\index{polynomial!perturbation!unbounded family}\index{whole-sum!recovery}
Put $C_f(A)=f(A)+P$, where $P=A^A$.
\begin{theorem}\label{g:T:mixed}
\index{polynomial!perturbation!main theorem}
For every polynomial \eqref{g:E:f}, the term
\begin{equation}\label{g:E:tf}
 T_f(A)=A^{\,A^{\,C_f(A)}}
       =A^{\,A^{\,f(A)+A^A}}
\end{equation}
is strongly reconstructing on positive-height inputs and reconstructing
on all finite ordered sets. In particular,
\begin{equation}\label{g:E:tk}
 T_k(A)=A^{\,A^{\,\Ap k+A^A}}
 \qquad(k\geq1)
\end{equation}
is reconstructing, with no bound on $k$.
\end{theorem}
Although \eqref{g:E:tf} is written with sums and Cartesian products, it is
an exponential term. To see this, the inner exponent satisfies
\[
 A^{\,f(A)+A^A}
 \cong (A^{A})^{\times c_1}
       \times(A^{\Ap2})^{\times c_2}
       \times\cdots\times(A^{\Ap k})^{\times c_k}
       \times A^{A^A}.
\]
Each factor $A^{\Ap j}$ is a~left-associated exponential term with
$j+1$ occurrences. Successively exponentiating the outer $A$ by the
factors gives an exponentiation-only representative of \eqref{g:E:tf}.
Its occurrence count is
\begin{equation}\label{g:E:occurrences}
 4+\sum\setm{c_j(j+1)}{1\leq j\leq k}.
\end{equation}
In particular, \eqref{g:E:tk} has $k+5$ occurrences.

\begin{proof}
We first prove the strong assertions for positive-height inputs. For
\eqref{g:E:tf}, the root exponent is $D=A^{C_f(A)}$, and
\[
 V=\Red(\Max(T_f(A)))\cong S\times L^{C_f(A)}.
\]
Put $Z_0=\Red(V)$ and $Z_1=\Red(Z_0)$. There are two branches:
\begin{align}
 N<k:\quad& Z_0\cong S\times L^{\times k},
       && Z_1\cong S;\label{g:E:small-chain}\\
 N\geq k:\quad& Z_0\cong S\times L^A,
       && Z_1\cong S\times L,
       \quad \Red(Z_1)\cong S.\label{g:E:large-chain}
\end{align}
Within either branch, terminal reduction and cancellation recover $L$
and $R$. All reductions just displayed have positive-height arguments;
no reduction of the final $S$ is required. When $k=1$ or $k=2$, the
small branch has no positive-height input, so the first strong assertion
already follows.

It remains to exclude a~collision between branches, without an upper
bound on $k$. Suppose $A$ lies in the small branch, $B$ in the large
branch, and $V(A)\cong V(B)$. Use subscripts for their invariants, and
write
\[
 a_0=|A|<k,\qquad m=|B|\geq k,\qquad
 q_A=|R_A|>0,\quad q_B=|R_B|>0.
\]
Equality of $Z_1$ gives
\[
 S_A\cong S_B\times L_B,
 \qquad J_A\cong J_B+R_B.
\]
Applying $\Tow$ to equality of $Z_0$ and cancelling $J_B$ yields
\begin{equation}\label{g:E:cross-one}
 R_B+kR_A\cong R_B\times B.
\end{equation}
Applying $\Tow$ to equality of the original $V$ and cancelling $J_B$
yields
\begin{equation}\label{g:E:cross-two}
 R_B+(R_A\times C_f(A))\cong R_B\times C_f(B).
\end{equation}
Taking cardinalities of these two identities eliminates the unknown
ratio $q_A/q_B$:
\begin{equation}\label{g:E:cross-card}
 kq_A=(m-1)q_B,\qquad
 k|C_f(B)|-(m-1)|C_f(A)|=k.
\end{equation}
We prove that the last equality is impossible.

First take $f(z)=z^k$. Since $a_0\leq k-1$,
\[
 |C_f(A)|\leq (k-1)^k+(k-1)^{k-1}
             =k(k-1)^{k-1}.
\]
Thus \eqref{g:E:cross-card} would imply
\[
 |C_f(B)|\leq1+(m-1)(k-1)^{k-1}\leq m^k.
\]
The last inequality follows from
$m^k\geq m(k-1)^{k-1}$ and $(k-1)^{k-1}\geq1$.
But $|C_f(B)|=m^k+|B^B|>m^k$, a~contradiction.

For a~general $f$, write $f(z)=z^k+g(z)$, where $g$ has nonnegative
coefficients and zero constant term. For every $j\geq1$,
\[
 km^j>(m-1)a_0^j,
\]
because $m>a_0$ and $k>a_0$. Hence
$kg(m)-(m-1)g(a_0)\geq0$. The preceding monomial estimates give
\[
 k\bigl(m^k+|B^B|\bigr)
 -(m-1)\bigl(a_0^k+|A^A|\bigr)>k.
\]
Adding the nonnegative contribution of $g$ contradicts
\eqref{g:E:cross-card} again. This excludes all cross-branch collisions.
Consequently $V$ determines $R$, proving the first strong assertion.

For the second, suppose $(R,D)$ is given, where $D=A^{C_f(A)}$.
Then
\[
 \Der(D)\cong R\times C_f(A).
\]
Cancel $R$ to recover $C_f(A)$, and apply
Lemma~\ref{t:L:polynomial}. This recovers $A$.

Finally, every exponential term contains a~copy of its input, obtained
recursively from constant maps. Its value is an antichain exactly when
the input is an antichain. On an $N$-element antichain, an
exponentiation-only representative of $T_f$ has the corresponding
ordinary integer-exponential cardinality, strictly increasing in
$N\geq1$. Induction proves this strict increase: if two positive
integer-valued functions $u,v$ are strictly increasing and $N<M$, then
\[
 u(N)^{v(N)}<u(M)^{v(N)}\leq u(M)^{v(M)}.
\]
Thus antichains are reconstructed and cannot collide with
positive-height inputs.
\end{proof}

\section{Occurrence counts and further extensions}
\index{occurrence!count!extensions}\index{polynomial!perturbation!extensions}
For $f(z)=z^k$, the term in~\eqref{g:E:tk} has $k+5$ occurrences.
Thus the theorem supplies reconstructing terms at arbitrarily large
occurrence counts. Its strong assertion is valid on all positive-height
inputs, not merely on $\OO$. The uniform extension proof in
Proposition~\ref{b:P:uniform} therefore applies on that larger class as
well.

This result is not a~proof that every term with $k+5$ occurrences
reconstructs its input. It is a~uniform theorem for the displayed
family. The case analysis in Part~II is what proves the universal
assertion through twelve occurrences.

\chapter{Substitution and exhaustive enumeration}\label{ch:enumeration}

\section{Uniform substitution}
\index{substitution!uniform}\index{occurrence!count!multiplication}
\begin{proposition}\label{P:substitution}
\index{substitution!proposition}
Let $s[t]$ denote the term obtained by replacing every occurrence of
$x$ in $s$ by the same term $t$. Then
\[
 |s[t]|=|s|\,|t|.
\]
If $s$ and $t$ are reconstructing, then so is $s[t]$.
\end{proposition}
\begin{proof}
The occurrence formula follows by counting the copies of $t$, or by
induction on $s$. If $s[t](A)\cong s[t](B)$, reconstruction for $s$
first gives $t(A)\cong t(B)$. Reconstruction for $t$ then gives
$A\cong B$.
\end{proof}
\index{substitution}\index{occurrence!count!factorization}

Both the uniformity of the substitution and the availability of the
outer reconstruction theorem matter. A~factorization of the integer
$n$ gives some substitution proofs at $n$ occurrences, not a
classification of all $n$-occurrence terms. When $n$ is prime there
is no proper uniform substitution with both factors greater than one.
The additive root budget~\eqref{t:E:budget} is unaffected by primality.

\section{The rooted-tree recurrence}
\index{rooted!tree!recurrence}\index{generating!function!rooted trees}\index{normal!form!rooted-tree model}
Let $a_n$ be the number of interchange normal forms with $n$
occurrences. The normal-form theorem in Chapter~\ref{b:S:normal} gives
\begin{equation}\label{E:tree-generating}
 \sum\setm{a_nz^n}{n\geq1}
 =z\prod\setm{(1-z^j)^{-a_j}}{j\geq1}.
\end{equation}
For each weight $j$, there are $a_j$ child types, each permitted with
arbitrary multiplicity. The factor $(1-z^j)^{-a_j}$ counts their
multisets. Multiplying over $j$ and adding the root proves the formula.
Only finitely many factors affect any specified coefficient.
\index{rooted!tree}\index{generating!function}

For an explicit recurrence, put
$c_m=\sum\setm{d a_d}{d\mid m}$. Logarithmic differentiation of
$B(z)=\sum\setm{a_{n+1}z^n}{n\geq0}$ gives
\[
 zB'(z)=B(z)\sum\setm{c_mz^m}{m\geq1}.
\]
Comparing coefficients yields
\begin{equation}\label{E:tree-recurrence}
 (n-1)a_n=\sum\setm{c_m a_{n-m}}{1\leq m\leq n-1},\qquad a_1=1.
\end{equation}
Every entry on the right has already been determined. This is an
independent arithmetic route to the normal-form counts.

\section{Binary parenthesizations and individual multiplicities}
\index{parenthesization!binary}\index{Catalan!number}\index{multiplicity!normal forms}
A~binary parenthesization with $n$ variable occurrences is a~full plane
binary tree with $n$ leaves. Splitting at its root gives the recurrence
$b_1=1$ and
\[
 b_n=\sum\setm{b_i b_{n-i}}{1\leq i\leq n-1}.
\]
Writing $B_0(z)=\sum\setm{b_nz^n}{n\geq1}$, this is
$B_0=z+B_0^2$. The solution with zero constant coefficient is
$(1-\sqrt{1-4z})/2$, whose coefficient of $z^n$ is
\begin{equation}\label{E:Catalan}
 b_n=\frac1n\binom{2n-2}{n-1}.
\end{equation}
\index{parenthesization}\index{Catalan!number}

For a~fixed unordered rooted tree $T$, let $w(T)$ count the binary
parenthesizations that normalize to it. Suppose its children are
$T_1,\ldots,T_s$, and a~distinct type $U$ has multiplicity $m_U$.
Then
\begin{equation}\label{E:multiplicity}
 w(T)=\frac{s!}{\prod\setm{m_U!}{U\text{ a~child type}}}
       \prod\setm{w(T_i)}{1\leq i\leq s}.
\end{equation}
A~binary expression has a~unique left spine ending at its leftmost
variable. The right subexpressions along that spine are exactly the
root children. The multinomial factor orders their types; the remaining
factors choose their binary representatives. This proves the formula
by induction, beginning with $w(x)=1$. It also proves
$\sum\setm{w(T)}{|T|=n}=b_n$.

\begin{table}[htbp]
\centering\renewcommand{\arraystretch}{1.12}
\begin{tabular}{@{}rrr@{}}
\toprule
Occurrences & Interchange normal forms & Binary parenthesizations\\\midrule
1&1&1\\2&1&1\\3&2&2\\4&4&5\\5&9&14\\6&20&42\\
7&48&132\\8&115&429\\9&286&1,430\\10&719&4,862\\
11&1,842&16,796\\12&4,766&58,786\\\bottomrule
\end{tabular}
\caption{The two different enumeration problems through twelve occurrences.}
\label{tab:all-counts}
\end{table}

\section{Root partitions and proof classes}
\index{root!partition!proof classes}\index{proof!class!uniform}\index{proof!class!regular}\index{proof!class!exceptional}
A~root partition records the occurrence counts of the child terms. Its
parts sum to $n-1$. If it has $m_j$ parts equal to $j$, the number of
normal forms with that partition is
\[
 \prod\setm{\binom{a_j+m_j-1}{m_j}}{m_j>0}.
\]
This chooses a~multiset from the $a_j$ child types independently at
each weight. It is not the partition of the smaller forest weights
that remains after removing the base occurrence from each nontrivial
child. Confusing these two partitions changes the mixed-root budget.

A~uniform root has at most one distinct nontrivial child type.
Otherwise the comparison and exceptional arguments assign its mixed
root to a~proved family. These assignments are disjoint bookkeeping
classes. The fact that a~term has two different proofs does not change
its count.

\chapter[The reconstruction toolkit]{The reconstruction toolkit and the remaining induction question}\label{ch:toolkit}

\section{A~single toolkit}
\index{reconstruction!toolkit}\index{proof!method!toolkit}
\begin{theorem}[Reconstruction toolkit]\label{T:toolkit}
\index{reconstruction!toolkit!theorem}
The following principles are available without any occurrence-bound
reconstruction theorem as an external premise.
\begin{enumerate}
\item Antichains are reconstructed by every exponential term. Every
term reflects membership in $\K$, and every term reconstructs inputs
in $\K$.
\item Products and sums cancel; positive Cartesian roots and fixed
nonconstant positive-coefficient polynomial arguments are unique.
\item Selected intervals commute with Cartesian products and
exponentiation. Their canonical coordinate sums give the derived-sum
and height-reduction identities.
\item Pure-power roots are strong. A~strong child admits arbitrary
uniform extensions with known nonnegative leaf count and positive
child multiplicity.
\item Forests of weight at most eight admit the fixed reduction and
recovery procedures of Lemma~\ref{t:L:forest}.
\item A~regular mixed root is strong. For a~nonregular root, recovery
of its first invariant reduces to excluding cross-branch equalities;
recovery with known $R$ is supplied by signature recognition whenever
the candidate signatures are distinct.
\item The polynomial-perturbation family of
Theorem~\ref{g:T:mixed} is reconstructing at arbitrarily large
occurrence counts. Uniform substitution preserves reconstruction.
\end{enumerate}
\end{theorem}
\begin{proof}
Item~1 is Lemma~\ref{L:antichains} and Theorem~\ref{b:T:K}.
Item~2 is Lemma~\ref{L:arithmetic-cancellation}. Item~3 follows from
Lemma~\ref{b:L:Max}, equations~\eqref{b:E:Der-laws}, and
Lemma~\ref{b:L:reduction}. Item~4 is Proposition~\ref{b:P:uniform}.
Item~5 is Lemma~\ref{t:L:forest}. Item~6 is
Proposition~\ref{t:P:regular}, Lemma~\ref{t:L:signature}, and
\eqref{t:E:cross}. Item~7 is Theorem~\ref{g:T:mixed} and
Proposition~\ref{P:substitution}. All of these results have been proved
in Part~I, apart from the explicitly cited classical structural inputs
in Chapters~\ref{b:S:arithmetic} and~\ref{ch:coordinates}.
\end{proof}

\section{Why ordinary induction is insufficient}
\index{induction!ordinary obstruction}\index{induction!limitations}
Writing a~term as $u^v$ gives two smaller terms. It does not give the
isomorphisms of their values required by the induction hypothesis.
An equality $u(A)^{v(A)}\cong u(B)^{v(B)}$ has two unknown exponents
and two unknown bases.

This issue already appears with chains. Writing $\mathbf j$ for the
$j$-element chain, one has
\[
 \mathbf3^{\mathbf3}
 \cong(\two^{\two})^{\mathbf3}
 \cong\two^{\two\times\mathbf3}
 \cong(\two^{\mathbf3})^{\two}
 \cong\mathbf4^{\two}.
\]
The bases and exponents on the two ends are different. The identity is
not a~counterexample to exponential-term reconstruction. It is a
counterexample to the proposed step of identifying both operands from
an arbitrary equality of exponentials.
\index{induction!ordinary obstruction}

\section{Where fixed forest descent stops being automatic}
\index{forest!descent obstruction}\index{forest!weight-nine obstruction}\index{crown!obstruction}
The weight bound eight in Lemma~\ref{t:L:forest} is not merely an
arbitrary stopping point of its proof. At weight nine consider
\[
 D=A^{A^{\times5}}\times A^{A^A}.
\]
Its nontrivial factors have occurrence counts six and three. The next
mixed sum in its descent is $A^{\times5}+P$. On a~positive-height
input of cardinality four the pure power is taller; on an input of
cardinality at least five, $P$ is selected, with the tie resolved by
Lemma~\ref{b:L:selection-order}.

Both behaviors occur outside $\K$. The four-element crown, with two
minimal elements below two maximal elements, has no self-map distinct
from the identity and pointwise comparable with it: a~map above the
identity fixes both maximal elements and hence both minimal elements,
and the dual argument applies below the identity. It is therefore
noncontractible. Adding an isolated point gives an input in $\OO$ of
cardinality five. Thus no one of the two displayed forests is selected
throughout $\OO$ in this example.
\index{forest!weight-nine obstruction}\index{crown}

The polynomial-perturbation theorem does handle this particular
variable selection. What remains for an unrestricted induction is a
uniform method for all such selections, including all possible
collisions between their reduction chains. Part~II supplies that method
for every mixed-root budget through twelve occurrences; it does not
assert an unproved induction step for arbitrary budgets.

\section{Effectivity and the meaning of determines}
\index{effectivity}\index{determination!effective}\index{reconstruction!effectivity}
Every recovery assertion is uniqueness up to isomorphism. It also gives
an existence-based finite search procedure when the full term value is
known. Constant maps imply $|A|\leq|t(A)|$. One may therefore enumerate
all ordered sets up to that finite cardinality, evaluate the fixed term,
and retain the candidates isomorphic to the given value. The theorem
proves that exactly one isomorphism type remains. The structural
recovery procedures explain why uniqueness holds; no efficiency claim
is made for this exhaustive procedure.

\part{Proofs through twelve occurrences}

\chapter{The input and the endomorphism ordered set}\label{ch:two}

The one-occurrence term is $x$ and already is its input. We give the
complete reconstruction from $A^A$ separately, both as the first
nontrivial case and as the model for the subsequent reductions.
\index{A-to-A~problem@$A^A$!problem}

\begin{theorem}\label{T:two}
\index{A-to-A~problem@$A^A$!problem!theorem}
For finite ordered sets $A,B$, $A^A\cong B^B$ implies $A\cong B$.
\end{theorem}
\begin{proof}
The antichain case follows from Lemma~\ref{L:antichains}; explicitly,
an $N$-element antichain has an $N^N$-element endomorphism antichain,
and $N\mapsto N^N$ is strictly increasing.

Suppose $A$ has positive height and put $E=A^A$. Starting with $E$,
form
\[
 U=\Tow(\Max(E)),\qquad V=\Max(U),\qquad
 W=\Max(\Tow(V)).
\]
We verify that all these expressions are defined and that they
reconstruct $A$. Set $L=\Max(A)$, $R=\Tow(L)$, and $S=\Max(R)$.
Since $L$ is nontrivial, $R$ is an actual ordered set.
The selected-interval and coordinate identities give
\begin{equation}\label{E:two-UVW}
 U\cong R\times A,\qquad V\cong S\times L,
 \qquad \Tow(V)\cong J+R,
\end{equation}
where $J=\Tow(S)$. If $J$ is present, it has height strictly less than
$\height(R)$. If $J=0$, the displayed sum is simply $R$.
In either case $W\cong\Max(R)=S$.

Thus the type $E$ determines $U,V,W$. Cancel $W$ from
$V\cong W\times L$ to recover $L$. Compute $R=\Tow(L)$ and cancel
$R$ from $U\cong R\times A$ to recover $A$.
Every cancellation has an already known actual factor. The only
application of $\Red$ has argument $V=S\times L$, of positive height;
there is no application to the terminal $S$.

If also $E\cong B^B$, the height formula makes $B$ a~positive-height
input. The same intrinsic $U,V,W$ recover the same $L$ and $R$.
Consequently $R\times A\cong R\times B$, and cancellation gives
$A\cong B$.
\end{proof}

\section{The strong base case}
\index{strong!reconstruction!base case}\index{occurrence!count!two: strong reconstruction}
For $P=A^A$, its root exponent is $A$ and
$V_P=S\times L$. One reduction gives $S$, so terminal recovery
identifies $R$. The pair $(R,A)$ already contains the input. Thus $P$
is strong on every positive-height class, and in particular on $\OO$.
This verifies directly the base case used in the uniform induction.

\section{An explicit bounded example}
\index{cardinality!bounds!explicit example}\index{endomorphism!ordered set!bounded example}
If $A=\mathbf2$, then $L=\mathbf2$ has canonical core $\mathbf2$ and
coordinate $\one$. Hence $R=S=\one$, while $A^A=\mathbf3$.
The derived sum of $\mathbf3=\mathbf2^{\mathbf2}$ is $\mathbf2$.
The construction gives $U=V=\mathbf2$ and $W=\one$, recovering
$L=A=\mathbf2$. This example illustrates that the coordinate sum is
not the number of maps and that height reduction is a~structural
operation rather than a~cardinality calculation.

\chapter{Three, four, and five occurrences}\label{ch:five}

Every root through five occurrences is uniform. Indeed, two distinct
nontrivial child types would have weights at least two and three,
requiring at least six occurrences after the outer root is added.
Theorem~\ref{T:two} and Proposition~\ref{b:P:uniform} therefore provide
a~complete induction through five. We list the terms and their root
exponents explicitly.

Retain $P=A^A$, $F=A^{A^{\times2}}$, and $G=A^P$ from
\eqref{b:E:small-terms}. The four four-occurrence terms are
$H_1,H_2,H_3,H_4$ in that equation. The forests of weight three are
\begin{equation}\label{b:E:D3}
 \mathcal D_3=(\Ap3,\ AP,\ F,\ G),
\end{equation}
and those of weight four are
\begin{equation}\label{b:E:D4}
 \mathcal D_4=(\Ap4,\ \Ap2P,\ \xp P2,\ AF,\ AG,
                    H_1,H_2,H_3,H_4).
\end{equation}
Put $U_i=A^{D_{4,i}}$, where $D_{4,i}$ is the $i$th forest in
$\mathcal D_4$.

\begin{longtable}{@{}rllr@{}}
\toprule Occurrences & Term & Root exponent $D_t$ & Binary\\\midrule
\endfirsthead
\toprule Occurrences & Term & Root exponent $D_t$ & Binary\\\midrule\endhead
\endfoot\endlastfoot
3 & $F$ & $\Ap{2}$ & 1\\
3 & $G$ & $P$ & 1\\
4 & $H_1$ & $\Ap{3}$ & 1\\
4 & $H_2$ & $A\,P$ & 2\\
4 & $H_3$ & $F$ & 1\\
4 & $H_4$ & $G$ & 1\\
5 & $U_{1}$ & $\Ap{4}$ & 1\\
5 & $U_{2}$ & $\xp{A}{2}\,P$ & 3\\
5 & $U_{3}$ & $\xp{P}{2}$ & 1\\
5 & $U_{4}$ & $A\,F$ & 2\\
5 & $U_{5}$ & $A\,G$ & 2\\
5 & $U_{6}$ & $H_1$ & 1\\
5 & $U_{7}$ & $H_2$ & 2\\
5 & $U_{8}$ & $H_3$ & 1\\*
5 & $U_{9}$ & $H_4$ & 1\\*
\bottomrule
\caption{All normal forms with three, four, and five occurrences.}
\label{tab:through-five}
\end{longtable}

\begin{theorem}\label{T:five}
\index{occurrence!count!five: theorem}
Every exponential term with at most five occurrences is reconstructing.
Every nontrivial such term is strong on $\OO$.
\end{theorem}
\begin{proof}
The one-occurrence case is immediate and the two-occurrence case is
Theorem~\ref{T:two}. For a~root with no nontrivial child, its exponent
is $A^{\times k}$, and the pure-power clause of
Proposition~\ref{b:P:uniform} applies. Otherwise there is only one
nontrivial child type $u$, repeated $s\geq1$ times, together with
$r\geq0$ leaves. The child has fewer occurrences and is strong by
induction. The uniform extension clause proves the strong assertion
for the parent. The table exhausts the normal forms because its root
forests exhaust the partitions of the root weight with the smaller
child types already listed. Strong reconstruction handles $\OO$;
Theorem~\ref{b:T:K} handles the complementary class. Interchange
normalization handles every binary representative.
\end{proof}

The binary multiplicities in the three groups sum to $2$, $5$, and
$14$. Their normal-form counts are $2$, $4$, and $9$. This is the first
point where the two counts differ; no reconstruction case is lost by
working with interchange normal forms.

\chapter[Comparisons through eleven]{The finite comparison classification through eleven}\label{b:S:classification}

We now classify the mixed roots whose occurrence count is at most
eleven. In the notation of~\eqref{t:E:mixed}, their budget is
\begin{equation}\label{b:E:budget}
 r+m+\wt(E_1)+\cdots+\wt(E_m)\leq10.
\end{equation}
The root invariants and the recovery of the entire $C$ are already
proved in~\eqref{t:E:VZ}. The comparison proof below uses
the basic bounds of Chapter~\ref{ch:bounds}. This gives a~particularly
short exhaustive list through eleven. The sharper bounds are used
later for the twelve-occurrence classification.

The height coefficients of the lists~\eqref{b:E:D3} and~\eqref{b:E:D4}
are
\begin{equation}\label{b:E:D3-heights}
 3,\ N+1,\ N^2,\ p
\end{equation}
and
\begin{equation}\label{b:E:D4-heights}
 4,\ N+2,\ 2N,\ N^2+1,\ p+1,\ N^3,\ Np,\ f,\ g,
\end{equation}
respectively, after division by $\ell=\height(A)$.
For $N\geq4$, the weight-three list is strictly increasing by height.
Every nonpure forest of weight at least three has height greater than
$N\ell$: it either contains $P$ and another factor, or contains a~term
whose root exponent has cardinality greater than $N$. The latter
follows from the constant-map embedding and the elementary map bound.

Call a~mixed root \emph{regular} if a~fixed displayed forest $D$ satisfies
\[
 \Max(C(A))\cong\Max(D(A))\qquad\text{for every $A\in\OO$}.
\]
The comparisons below hold for every positive-height input with $N\geq4$.

\begin{proposition}\label{b:P:classification}
\index{classification!through five}
Every mixed root satisfying \eqref{b:E:budget} is regular or belongs to
one of the exceptional families in Table~\ref{b:tab:exceptions}.
In a~regular case a~fixed selected forest has weight at most seven.
If the selected forest is pure, all displayed forests of $C$ are pure.
At exactly eleven occurrences the exceptional list has eleven normal
forms. Below eleven, the only additional instances in these families
are $C=\Ap5+P$ and $C=H_1+G$, both with $r=0$ and ten occurrences.
\end{proposition}
\begin{proof}
We give all comparisons, including the possible ties. For two summands,
their weights have sum at most eight. A~weight-one forest is $A$ and
loses to any larger-weight forest. Among weight-two forests, $P$ beats
$A^{\times2}$. Against a~larger forest, $A^{\times2}$ loses. The forest
$P$ loses to every nonpure forest of weight at least three, by the
observation after \eqref{b:E:D4-heights}. Against a~pure $A^{\times k}$,
$P$ wins for $k\leq4$, using \eqref{b:E:exponent-tie} when $N=k=4$.
The remaining possibilities are $k=5,6$, which are exceptional.

The weight-three list \eqref{b:E:D3} is strictly increasing by height.
The only remaining unequal weight pairs are $4+3$, $5+3$, and $4+4$.
The complete first two comparisons are in
Tables~\ref{b:tab:43} and~\ref{b:tab:53}. Here $R$ means that the row
wins, $C$ that the column wins, and a~dash marks an exceptional pair.
For the weight-five list, define $U_i=A^{D_i}$, where $D_i$ is the
$i$th forest in \eqref{b:E:D4}; in particular $U_1=A^{\Ap4}$.

\begin{table}[htbp]
\centering
\small
\setlength{\tabcolsep}{5pt}
\begin{tabular}{@{}llcccc@{}}
\toprule
Row forest & Height$/\ell$ & $\Ap3$ & $A\times P$ & $F$ & $G$\\
\midrule
$\Ap4$ & $4$ & R&C&C&C\\
$\Ap2\times P$ & $N+2$ & R&R&C&C\\
$P^{\times2}$ & $2N$ & R&R&C&C\\
$A\times F$ & $N^2+1$ & R&R&R&C\\
$A\times G$ & $p+1$ & R&R&R&R\\
$H_1$ & $N^3$ & R&R&R&---\\
$H_2$ & $Np$ & R&R&R&R\\
$H_3$ & $f$ & R&R&R&R\\
$H_4$ & $g$ & R&R&R&R\\
\bottomrule
\end{tabular}
\caption{Every weight-$4+3$ comparison on $N\geq4$.}
\label{b:tab:43}
\end{table}

\begin{table}[htbp]
\centering
\small
\setlength{\tabcolsep}{5pt}
\begin{tabular}{@{}llcccc@{}}
\toprule
Row forest & Height$/\ell$ & $\Ap3$ & $A\times P$ & $F$ & $G$\\
\midrule
$\Ap5$ & $5$ & R&C&C&C\\
$\Ap3\times P$ & $N+3$ & R&R&C&C\\
$A\times P^{\times2}$ & $2N+1$ & R&R&C&C\\
$\Ap2\times F$ & $N^2+2$ & R&R&R&C\\
$\Ap2\times G$ & $p+2$ & R&R&R&R\\
$P\times F$ & $N^2+N$ & R&R&R&C\\
$P\times G$ & $p+N$ & R&R&R&R\\
$A\times H_1$ & $N^3+1$ & R&R&R&---\\
$A\times H_2$ & $Np+1$ & R&R&R&R\\
$A\times H_3$ & $f+1$ & R&R&R&R\\
$A\times H_4$ & $g+1$ & R&R&R&R\\
$U_1$ & $N^4$ & R&R&R&---\\
$U_2$ & $N^2p$ & R&R&R&R\\
$U_3$ & $p^2$ & R&R&R&R\\
$U_4$ & $Nf$ & R&R&R&R\\
$U_5$ & $Ng$ & R&R&R&R\\
$U_6$ & $h_1$ & R&R&R&R\\
$U_7$ & $h_2$ & R&R&R&R\\
$U_8$ & $h_3$ & R&R&R&R\\
$U_9$ & $h_4$ & R&R&R&R\\
\bottomrule
\end{tabular}
\caption{All twenty weight-five forests and every weight-$5+3$ comparison.}
\label{b:tab:53}
\end{table}

For clarity, the numerical inequalities proving the tables are
$N\geq4$, $p\geq N^2+N$, and $f,g\geq p+N-1$.
The values $N+3$ and $2N+1$ are less than $N^2$; $N^2+2<p$;
and $N^2+N\leq p$. The entries involving $U_2,\ldots,U_5$ have
height coefficient greater than $p$. So do those involving
$U_6,\ldots,U_9$: $H_1\cong F^A$ and $H_2\cong P^P$ give
$h_1>f>p$, $h_2>p$; the basic map bound gives $h_3>f>p$ and
$h_4>g>p$.

There are just two additional possible height ties in these tables.
For $N=4$, $A^{\times5}$ ties $A\times P$; their selected intervals
are $L^{\times5}$ and $L^{\one+A}$, and the latter wins
by \eqref{b:E:exponent-tie}. When $p=N^2+N$, compare
$\Max(P\times F)=L^{A+A^{\times2}}$ with $\Max(G)=L^P$.
Their exponent heights are $2\ell$ and $N\ell$, so $G$ wins by the
same rule. Thus every nondashed table entry, including equality cases,
is justified.

For two weight-four forests, the later member in the ordered list
\eqref{b:E:D4} wins except for the following six pairs:
\begin{equation}\label{b:E:44-exceptions}
\begin{gathered}
 (A\times G,H_1),\quad(H_1,H_3),\quad(H_1,H_4),\\
 (H_2,H_3),\quad(H_2,H_4),\quad(H_3,H_4).
\end{gathered}
\end{equation}
This is the complete $4+4$ comparison, not a~list of samples.
Indeed, the first five height coefficients in \eqref{b:E:D4-heights}
are strictly increasing. The coefficient $N^3$ exceeds the first four;
$Np$ exceeds $N^3$ and $p+1$; and each of $f,g$ exceeds $p+1$ and
all earlier coefficients except the separately listed $N^3,Np$.
Only the six displayed comparisons remain undecided by these bounds.

For three summands, the total forest weight is at most seven.
A~largest weight five can be accompanied only by $1+1$, and weight
four only by total weight at most three. If the remaining weights are
$2+1$, even the comparison $A^{\times4}$ versus $P$ is fixed.
If the largest weight is three, all comparisons among weights at most
three are fixed. For four summands the total weight is at most six,
so the largest weight is at most three. For five summands, all weights
must be one. Thus no new exceptional comparison occurs for three or
more summands.

The budget gives the allowed $r$ in Table~\ref{b:tab:exceptions}. Every
selected forest has weight at most seven, because it has at least one
other positive-weight companion. Finally, a~pure forest can beat a
nonpure one only in the two exceptional power comparisons. This proves
all assertions.
\end{proof}

\begin{table}[htbp]
\centering
\small
\setlength{\tabcolsep}{6pt}
\begin{tabular}{@{}clcl@{}}
\toprule
Case & Mixed sum $C$ & $r$ at eleven & Proof\\
\midrule
1 & $\Ap5+P$ & 1 & Proposition~\ref{b:P:power}\\
2 & $\Ap6+P$ & 0 & Proposition~\ref{b:P:power}\\
3 & $H_1+G$ & 1 & Proposition~\ref{b:P:cubic}\\
4 & $(A\times H_1)+G$ & 0 & Proposition~\ref{b:P:cubic}\\
5 & $U_1+G$ & 0 & Proposition~\ref{b:P:quartic}\\
6 & $H_1+(A\times G)$ & 0 & Proposition~\ref{b:P:cubic}\\
7 & $H_1+H_3$ & 0 & Proposition~\ref{b:P:Hpairs}\\
8 & $H_1+H_4$ & 0 & Proposition~\ref{b:P:Hpairs}\\
9 & $H_2+H_3$ & 0 & Proposition~\ref{b:P:Hpairs}\\
10 & $H_2+H_4$ & 0 & Proposition~\ref{b:P:Hpairs}\\
11 & $H_3+H_4$ & 0 & Proposition~\ref{b:P:Hpairs}\\
\bottomrule
\end{tabular}
\caption{The eleven exceptional normal forms, each denoting $E_{r,C}$.}
\label{b:tab:exceptions}
\end{table}

\chapter{Six occurrences}\label{ch:six}
There are 20 interchange normal forms and
42 binary parenthesizations at 6 occurrences.
The complete root classification consists of 18 uniform forms
and 2 regular mixed forms. There are no exceptional
forms on $\OO$ at this occurrence count.

\section{The mixed forms}
\index{mixed!form!classification}\index{mixed!root!occurrence cases}
\begingroup\small\renewcommand{\arraystretch}{1.17}
\begin{longtable}{@{}rrp{.44\textwidth}p{.25\textwidth}r@{}}
\toprule No. & $r$ & Mixed sum $C$ & Selected-forest chain & Binary\\\midrule
\endfirsthead
\multicolumn{5}{c}{\textit{6 occurrences, continued}}\\
\toprule No. & $r$ & Mixed sum $C$ & Selected-forest chain & Binary\\\midrule\endhead
\midrule\multicolumn{5}{r}{\textit{Continued on the next page}}\\\endfoot
\endlastfoot
1 & 0 & $A + \Ap{2}$ & $\Ap{2}$ & 2\\*
2 & 0 & $A + P$ & $P\,,\ A$ & 2\\*
\bottomrule
\caption{Every regular mixed root at 6 occurrences. Each row denotes $E_{r,C}$. The chain ends at its displayed pure forest; the final interval $S$ is omitted.}\label{tab:regular6}
\end{longtable}\endgroup

The two sums are $A+A^{\times2}$ and $A+P$. In the first, the square
supplies the selected interval, so $Z_0=SL^{\times2}$ and terminal
recovery gives $R$. Once the root exponent is known, whole-exponent
cancellation gives $A+A^{\times2}$, a~fixed nonconstant polynomial in
$A$, which determines $A$.

In the second, $P$ is selected because $N\ell>\ell$. Its forest chain
is $P,A$, so $Z_0=SL^A$ reduces to $SL$ and then $S$. After recovery
of $R$ and the whole sum, $\Der(A+P)=RA$, and cancellation gives $A$.
This spells out both mixed reconstructions without any summand
identification.

\section{Completion of the occurrence case}
\index{occurrence!case!completion}\index{reconstruction!occurrence case}
\begin{theorem}\label{T:n6}
\index{occurrence!count!six: theorem}
Every exponential term with at most 6 occurrences is
reconstructing. Every nontrivial such term is strong on $\OO$.
\end{theorem}
\begin{proof}
The smaller occurrence counts have already been established. Consider
a~normal form of the new count. If its root is uniform, its distinct
nontrivial child, when present, has smaller occurrence count and is
strong. Proposition~\ref{b:P:uniform} proves the strong assertion for
the parent; a~pure root uses the pure-power clause.

Otherwise the row in Table~\ref{tab:regular6} specifies its entire
mixed sum and a~fixed selected-forest chain. The comparisons in
Proposition~\ref{b:P:classification} prove that selection, including
all height ties. Each chain is an application of
Lemma~\ref{t:L:forest}; its terminal pure power has a~known exponent.
The first invariant $Z_0$ therefore determines $R$. Given the pair
$(R,D_t)$, recover the whole sum $C$ by~\eqref{t:E:wholeC}.
If its selected forest is pure, all its displayed forests are pure,
so polynomial injectivity recovers $A$. If it is nonpure, compute
$\Der(C)=\Der(D_0)$ and apply the recovery assertion of the forest
lemma. This proves both strong assertions in every listed row.

The uniform and mixed classes exhaust the normal forms by the
root-budget classification. A~strong term reconstructs its input on
$\OO$; Theorem~\ref{b:T:K} handles $\K$ and excludes a~collision
between the two classes. Finally, normalizing a~binary
parenthesization uses only isomorphism-preserving exponential laws.
Thus every binary parenthesization is covered.
\end{proof}

\chapter{Seven occurrences}\label{ch:seven}
There are 48 interchange normal forms and
132 binary parenthesizations at 7 occurrences.
The complete root classification consists of 41 uniform forms
and 7 regular mixed forms. There are no exceptional
forms on $\OO$ at this occurrence count.

\section{The mixed forms}
\index{mixed!form!classification}\index{mixed!root!occurrence cases}
\begingroup\small\renewcommand{\arraystretch}{1.17}
\begin{longtable}{@{}rrp{.44\textwidth}p{.25\textwidth}r@{}}
\toprule No. & $r$ & Mixed sum $C$ & Selected-forest chain & Binary\\\midrule
\endfirsthead
\multicolumn{5}{c}{\textit{7 occurrences, continued}}\\
\toprule No. & $r$ & Mixed sum $C$ & Selected-forest chain & Binary\\\midrule\endhead
\midrule\multicolumn{5}{r}{\textit{Continued on the next page}}\\\endfoot
\endlastfoot
1 & 1 & $A + \Ap{2}$ & $\Ap{2}$ & 6\\
2 & 1 & $A + P$ & $P\,,\ A$ & 6\\
3 & 0 & $A + \Ap{3}$ & $\Ap{3}$ & 2\\
4 & 0 & $A + A\,P$ & $A\,P\,,\ A$ & 4\\
5 & 0 & $A + F$ & $F\,,\ \Ap{2}$ & 2\\
6 & 0 & $A + G$ & $G\,,\ P\,,\ A$ & 2\\*
7 & 0 & $\Ap{2} + P$ & $P\,,\ A$ & 2\\*
\bottomrule
\caption{Every regular mixed root at 7 occurrences. Each row denotes $E_{r,C}$. The chain ends at its displayed pure forest; the final interval $S$ is omitted.}\label{tab:regular7}
\end{longtable}\endgroup

\section{Completion of the occurrence case}
\index{occurrence!case!completion}\index{reconstruction!occurrence case}
\begin{theorem}\label{T:n7}
\index{occurrence!count!seven: theorem}
Every exponential term with at most 7 occurrences is
reconstructing. Every nontrivial such term is strong on $\OO$.
\end{theorem}
\begin{proof}
The smaller occurrence counts have already been established. Consider
a~normal form of the new count. If its root is uniform, its distinct
nontrivial child, when present, has smaller occurrence count and is
strong. Proposition~\ref{b:P:uniform} proves the strong assertion for
the parent; a~pure root uses the pure-power clause.

Otherwise the row in Table~\ref{tab:regular7} specifies its entire
mixed sum and a~fixed selected-forest chain. The comparisons in
Proposition~\ref{b:P:classification} prove that selection, including
all height ties. Each chain is an application of
Lemma~\ref{t:L:forest}; its terminal pure power has a~known exponent.
The first invariant $Z_0$ therefore determines $R$. Given the pair
$(R,D_t)$, recover the whole sum $C$ by~\eqref{t:E:wholeC}.
If its selected forest is pure, all its displayed forests are pure,
so polynomial injectivity recovers $A$. If it is nonpure, compute
$\Der(C)=\Der(D_0)$ and apply the recovery assertion of the forest
lemma. This proves both strong assertions in every listed row.

The uniform and mixed classes exhaust the normal forms by the
root-budget classification. A~strong term reconstructs its input on
$\OO$; Theorem~\ref{b:T:K} handles $\K$ and excludes a~collision
between the two classes. Finally, normalizing a~binary
parenthesization uses only isomorphism-preserving exponential laws.
Thus every binary parenthesization is covered.
\end{proof}

The seven regular roots listed here are the roots that are mixed at the
outermost vertex. Uniform extensions of the two six-occurrence mixed
forms belong to the uniform class. They are not new outer mixed roots.
This distinction is needed when counting newly proved terms rather
than comparing two different recursive proof families.

\chapter{Eight occurrences}\label{ch:eight}
There are 115 interchange normal forms and
429 binary parenthesizations at 8 occurrences.
The complete root classification consists of 89 uniform forms
and 26 regular mixed forms. There are no exceptional
forms on $\OO$ at this occurrence count.

\section{The mixed forms}
\index{mixed!form!classification}\index{mixed!root!occurrence cases}
\begingroup\small\renewcommand{\arraystretch}{1.17}
\begin{longtable}{@{}rrp{.44\textwidth}p{.25\textwidth}r@{}}
\toprule No. & $r$ & Mixed sum $C$ & Selected-forest chain & Binary\\\midrule
\endfirsthead
\multicolumn{5}{c}{\textit{8 occurrences, continued}}\\
\toprule No. & $r$ & Mixed sum $C$ & Selected-forest chain & Binary\\\midrule\endhead
\midrule\multicolumn{5}{r}{\textit{Continued on the next page}}\\\endfoot
\endlastfoot
1 & 2 & $A + \Ap{2}$ & $\Ap{2}$ & 12\\
2 & 2 & $A + P$ & $P\,,\ A$ & 12\\
3 & 1 & $A + \Ap{3}$ & $\Ap{3}$ & 6\\
4 & 1 & $A + A\,P$ & $A\,P\,,\ A$ & 12\\
5 & 1 & $A + F$ & $F\,,\ \Ap{2}$ & 6\\
6 & 1 & $A + G$ & $G\,,\ P\,,\ A$ & 6\\
7 & 1 & $\Ap{2} + P$ & $P\,,\ A$ & 6\\
8 & 0 & $A + A + \Ap{2}$ & $\Ap{2}$ & 3\\
9 & 0 & $A + A + P$ & $P\,,\ A$ & 3\\
10 & 0 & $A + \Ap{4}$ & $\Ap{4}$ & 2\\
11 & 0 & $A + \xp{A}{2}\,P$ & $\xp{A}{2}\,P\,,\ A$ & 6\\
12 & 0 & $A + A\,F$ & $A\,F\,,\ \Ap{2}$ & 4\\
13 & 0 & $A + A\,G$ & $A\,G\,,\ P\,,\ A$ & 4\\
14 & 0 & $A + \xp{P}{2}$ & $\xp{P}{2}\,,\ A$ & 2\\
15 & 0 & $A + H_1$ & $H_1\,,\ \Ap{3}$ & 2\\
16 & 0 & $A + H_2$ & $H_2\,,\ A\,P\,,\ A$ & 4\\
17 & 0 & $A + H_3$ & $H_3\,,\ F\,,\ \Ap{2}$ & 2\\
18 & 0 & $A + H_4$ & $H_4\,,\ G\,,\ P\,,\ A$ & 2\\
19 & 0 & $\Ap{2} + \Ap{3}$ & $\Ap{3}$ & 2\\
20 & 0 & $\Ap{2} + A\,P$ & $A\,P\,,\ A$ & 4\\
21 & 0 & $\Ap{2} + F$ & $F\,,\ \Ap{2}$ & 2\\
22 & 0 & $\Ap{2} + G$ & $G\,,\ P\,,\ A$ & 2\\
23 & 0 & $P + \Ap{3}$ & $P\,,\ A$ & 2\\
24 & 0 & $P + A\,P$ & $A\,P\,,\ A$ & 4\\
25 & 0 & $P + F$ & $F\,,\ \Ap{2}$ & 2\\*
26 & 0 & $P + G$ & $G\,,\ P\,,\ A$ & 2\\*
\bottomrule
\caption{Every regular mixed root at 8 occurrences. Each row denotes $E_{r,C}$. The chain ends at its displayed pure forest; the final interval $S$ is omitted.}\label{tab:regular8}
\end{longtable}\endgroup

\section{Completion of the occurrence case}
\index{occurrence!case!completion}\index{reconstruction!occurrence case}
\begin{theorem}\label{T:n8}
\index{occurrence!count!eight: theorem}
Every exponential term with at most 8 occurrences is
reconstructing. Every nontrivial such term is strong on $\OO$.
\end{theorem}
\begin{proof}
The smaller occurrence counts have already been established. Consider
a~normal form of the new count. If its root is uniform, its distinct
nontrivial child, when present, has smaller occurrence count and is
strong. Proposition~\ref{b:P:uniform} proves the strong assertion for
the parent; a~pure root uses the pure-power clause.

Otherwise the row in Table~\ref{tab:regular8} specifies its entire
mixed sum and a~fixed selected-forest chain. The comparisons in
Proposition~\ref{b:P:classification} prove that selection, including
all height ties. Each chain is an application of
Lemma~\ref{t:L:forest}; its terminal pure power has a~known exponent.
The first invariant $Z_0$ therefore determines $R$. Given the pair
$(R,D_t)$, recover the whole sum $C$ by~\eqref{t:E:wholeC}.
If its selected forest is pure, all its displayed forests are pure,
so polynomial injectivity recovers $A$. If it is nonpure, compute
$\Der(C)=\Der(D_0)$ and apply the recovery assertion of the forest
lemma. This proves both strong assertions in every listed row.

The uniform and mixed classes exhaust the normal forms by the
root-budget classification. A~strong term reconstructs its input on
$\OO$; Theorem~\ref{b:T:K} handles $\K$ and excludes a~collision
between the two classes. Finally, normalizing a~binary
parenthesization uses only isomorphism-preserving exponential laws.
Thus every binary parenthesization is covered.
\end{proof}

\chapter{Nine occurrences}\label{ch:nine}
There are 286 interchange normal forms and
1,430 binary parenthesizations at 9 occurrences.
The complete root classification consists of 209 uniform forms
and 77 regular mixed forms. There are no exceptional
forms on $\OO$ at this occurrence count.

\section{The mixed forms}
\index{mixed!form!classification}\index{mixed!root!occurrence cases}
\begingroup\small\renewcommand{\arraystretch}{1.17}
\begin{longtable}{@{}rrp{.44\textwidth}p{.25\textwidth}r@{}}
\toprule No. & $r$ & Mixed sum $C$ & Selected-forest chain & Binary\\\midrule
\endfirsthead
\multicolumn{5}{c}{\textit{9 occurrences, continued}}\\
\toprule No. & $r$ & Mixed sum $C$ & Selected-forest chain & Binary\\\midrule\endhead
\midrule\multicolumn{5}{r}{\textit{Continued on the next page}}\\\endfoot
\endlastfoot
1 & 3 & $A + \Ap{2}$ & $\Ap{2}$ & 20\\
2 & 3 & $A + P$ & $P\,,\ A$ & 20\\
3 & 2 & $A + \Ap{3}$ & $\Ap{3}$ & 12\\
4 & 2 & $A + A\,P$ & $A\,P\,,\ A$ & 24\\
5 & 2 & $A + F$ & $F\,,\ \Ap{2}$ & 12\\
6 & 2 & $A + G$ & $G\,,\ P\,,\ A$ & 12\\
7 & 2 & $\Ap{2} + P$ & $P\,,\ A$ & 12\\
8 & 1 & $A + A + \Ap{2}$ & $\Ap{2}$ & 12\\
9 & 1 & $A + A + P$ & $P\,,\ A$ & 12\\
10 & 1 & $A + \Ap{4}$ & $\Ap{4}$ & 6\\
11 & 1 & $A + \xp{A}{2}\,P$ & $\xp{A}{2}\,P\,,\ A$ & 18\\
12 & 1 & $A + A\,F$ & $A\,F\,,\ \Ap{2}$ & 12\\
13 & 1 & $A + A\,G$ & $A\,G\,,\ P\,,\ A$ & 12\\
14 & 1 & $A + \xp{P}{2}$ & $\xp{P}{2}\,,\ A$ & 6\\
15 & 1 & $A + H_1$ & $H_1\,,\ \Ap{3}$ & 6\\
16 & 1 & $A + H_2$ & $H_2\,,\ A\,P\,,\ A$ & 12\\
17 & 1 & $A + H_3$ & $H_3\,,\ F\,,\ \Ap{2}$ & 6\\
18 & 1 & $A + H_4$ & $H_4\,,\ G\,,\ P\,,\ A$ & 6\\
19 & 1 & $\Ap{2} + \Ap{3}$ & $\Ap{3}$ & 6\\
20 & 1 & $\Ap{2} + A\,P$ & $A\,P\,,\ A$ & 12\\
21 & 1 & $\Ap{2} + F$ & $F\,,\ \Ap{2}$ & 6\\
22 & 1 & $\Ap{2} + G$ & $G\,,\ P\,,\ A$ & 6\\
23 & 1 & $P + \Ap{3}$ & $P\,,\ A$ & 6\\
24 & 1 & $P + A\,P$ & $A\,P\,,\ A$ & 12\\
25 & 1 & $P + F$ & $F\,,\ \Ap{2}$ & 6\\
26 & 1 & $P + G$ & $G\,,\ P\,,\ A$ & 6\\
27 & 0 & $A + A + \Ap{3}$ & $\Ap{3}$ & 3\\
28 & 0 & $A + A + A\,P$ & $A\,P\,,\ A$ & 6\\
29 & 0 & $A + A + F$ & $F\,,\ \Ap{2}$ & 3\\
30 & 0 & $A + A + G$ & $G\,,\ P\,,\ A$ & 3\\
31 & 0 & $A + \Ap{2} + \Ap{2}$ & $\Ap{2}$ & 3\\
32 & 0 & $A + \Ap{2} + P$ & $P\,,\ A$ & 6\\
33 & 0 & $A + P + P$ & $P\,,\ A$ & 3\\
34 & 0 & $A + \Ap{5}$ & $\Ap{5}$ & 2\\
35 & 0 & $A + \xp{A}{3}\,P$ & $\xp{A}{3}\,P\,,\ A$ & 8\\
36 & 0 & $A + \xp{A}{2}\,F$ & $\xp{A}{2}\,F\,,\ \Ap{2}$ & 6\\
37 & 0 & $A + \xp{A}{2}\,G$ & $\xp{A}{2}\,G\,,\ P\,,\ A$ & 6\\
38 & 0 & $A + A\,\xp{P}{2}$ & $A\,\xp{P}{2}\,,\ A$ & 6\\
39 & 0 & $A + A\,H_1$ & $A\,H_1\,,\ \Ap{3}$ & 4\\
40 & 0 & $A + A\,H_2$ & $A\,H_2\,,\ A\,P\,,\ A$ & 8\\
41 & 0 & $A + A\,H_3$ & $A\,H_3\,,\ F\,,\ \Ap{2}$ & 4\\
42 & 0 & $A + A\,H_4$ & $A\,H_4\,,\ G\,,\ P\,,\ A$ & 4\\
43 & 0 & $A + P\,F$ & $P\,F\,,\ \Ap{2}$ & 4\\
44 & 0 & $A + P\,G$ & $P\,G\,,\ P\,,\ A$ & 4\\
45 & 0 & $A + U_{1}$ & $U_{1}\,,\ \Ap{4}$ & 2\\
46 & 0 & $A + U_{2}$ & $U_{2}\,,\ \xp{A}{2}\,P\,,\ A$ & 6\\
47 & 0 & $A + U_{4}$ & $U_{4}\,,\ A\,F\,,\ \Ap{2}$ & 4\\
48 & 0 & $A + U_{5}$ & $U_{5}\,,\ A\,G\,,\ P\,,\ A$ & 4\\
49 & 0 & $A + U_{3}$ & $U_{3}\,,\ \xp{P}{2}\,,\ A$ & 2\\
50 & 0 & $A + U_{6}$ & $U_{6}\,,\ H_1\,,\ \Ap{3}$ & 2\\
51 & 0 & $A + U_{7}$ & $U_{7}\,,\ H_2\,,\ A\,P\,,\ A$ & 4\\
52 & 0 & $A + U_{8}$ & $U_{8}\,,\ H_3\,,\ F\,,\ \Ap{2}$ & 2\\
53 & 0 & $A + U_{9}$ & $U_{9}\,,\ H_4\,,\ G\,,\ P\,,\ A$ & 2\\
54 & 0 & $\Ap{2} + \Ap{4}$ & $\Ap{4}$ & 2\\
55 & 0 & $\Ap{2} + \xp{A}{2}\,P$ & $\xp{A}{2}\,P\,,\ A$ & 6\\
56 & 0 & $\Ap{2} + A\,F$ & $A\,F\,,\ \Ap{2}$ & 4\\
57 & 0 & $\Ap{2} + A\,G$ & $A\,G\,,\ P\,,\ A$ & 4\\
58 & 0 & $\Ap{2} + \xp{P}{2}$ & $\xp{P}{2}\,,\ A$ & 2\\
59 & 0 & $\Ap{2} + H_1$ & $H_1\,,\ \Ap{3}$ & 2\\
60 & 0 & $\Ap{2} + H_2$ & $H_2\,,\ A\,P\,,\ A$ & 4\\
61 & 0 & $\Ap{2} + H_3$ & $H_3\,,\ F\,,\ \Ap{2}$ & 2\\
62 & 0 & $\Ap{2} + H_4$ & $H_4\,,\ G\,,\ P\,,\ A$ & 2\\
63 & 0 & $P + \Ap{4}$ & $P\,,\ A$ & 2\\
64 & 0 & $P + \xp{A}{2}\,P$ & $\xp{A}{2}\,P\,,\ A$ & 6\\
65 & 0 & $P + A\,F$ & $A\,F\,,\ \Ap{2}$ & 4\\
66 & 0 & $P + A\,G$ & $A\,G\,,\ P\,,\ A$ & 4\\
67 & 0 & $P + \xp{P}{2}$ & $\xp{P}{2}\,,\ A$ & 2\\
68 & 0 & $P + H_1$ & $H_1\,,\ \Ap{3}$ & 2\\
69 & 0 & $P + H_2$ & $H_2\,,\ A\,P\,,\ A$ & 4\\
70 & 0 & $P + H_3$ & $H_3\,,\ F\,,\ \Ap{2}$ & 2\\
71 & 0 & $P + H_4$ & $H_4\,,\ G\,,\ P\,,\ A$ & 2\\
72 & 0 & $\Ap{3} + A\,P$ & $A\,P\,,\ A$ & 4\\
73 & 0 & $\Ap{3} + F$ & $F\,,\ \Ap{2}$ & 2\\
74 & 0 & $\Ap{3} + G$ & $G\,,\ P\,,\ A$ & 2\\
75 & 0 & $A\,P + F$ & $F\,,\ \Ap{2}$ & 4\\
76 & 0 & $A\,P + G$ & $G\,,\ P\,,\ A$ & 4\\*
77 & 0 & $F + G$ & $G\,,\ P\,,\ A$ & 2\\*
\bottomrule
\caption{Every regular mixed root at 9 occurrences. Each row denotes $E_{r,C}$. The chain ends at its displayed pure forest; the final interval $S$ is omitted.}\label{tab:regular9}
\end{longtable}\endgroup

\section{Completion of the occurrence case}
\index{occurrence!case!completion}\index{reconstruction!occurrence case}
\begin{theorem}\label{T:n9}
\index{occurrence!count!nine: theorem}
Every exponential term with at most 9 occurrences is
reconstructing. Every nontrivial such term is strong on $\OO$.
\end{theorem}
\begin{proof}
The smaller occurrence counts have already been established. Consider
a~normal form of the new count. If its root is uniform, its distinct
nontrivial child, when present, has smaller occurrence count and is
strong. Proposition~\ref{b:P:uniform} proves the strong assertion for
the parent; a~pure root uses the pure-power clause.

Otherwise the row in Table~\ref{tab:regular9} specifies its entire
mixed sum and a~fixed selected-forest chain. The comparisons in
Proposition~\ref{b:P:classification} prove that selection, including
all height ties. Each chain is an application of
Lemma~\ref{t:L:forest}; its terminal pure power has a~known exponent.
The first invariant $Z_0$ therefore determines $R$. Given the pair
$(R,D_t)$, recover the whole sum $C$ by~\eqref{t:E:wholeC}.
If its selected forest is pure, all its displayed forests are pure,
so polynomial injectivity recovers $A$. If it is nonpure, compute
$\Der(C)=\Der(D_0)$ and apply the recovery assertion of the forest
lemma. This proves both strong assertions in every listed row.

The uniform and mixed classes exhaust the normal forms by the
root-budget classification. A~strong term reconstructs its input on
$\OO$; Theorem~\ref{b:T:K} handles $\K$ and excludes a~collision
between the two classes. Finally, normalizing a~binary
parenthesization uses only isomorphism-preserving exponential laws.
Thus every binary parenthesization is covered.
\end{proof}

\section{The factorization \texorpdfstring{$9=3\cdot3$}{9 = 3 x 3}}
\index{occurrence!count!nine: factorization}\index{factorization!$9=3\cdot3$}
Substituting either three-occurrence term into either three-occurrence
term supplies further proofs by Proposition~\ref{P:substitution}.
These proofs overlap the exhaustive root classification. They do not
replace the seventy-seven mixed-root arguments, which have just been
proved for inputs of every finite cardinality.

\chapter{Ten occurrences: the first branch comparisons}\label{ch:ten}

At ten occurrences there are 719 normal forms and 4,862 binary
parenthesizations. Of these normal forms, 497 are uniform, 220 are
regular, and two require branch comparisons. The two sums are
\[
 C=A^{\times5}+P\qquad\text{and}\qquad C=H_1+G,
\]
both with $r=0$. The regular classification has already been proved in
Chapter~\ref{b:S:classification}. This chapter gives the full
exceptional arguments in slightly broader parameter ranges, so they
can also be used at eleven occurrences.

\section{A~Cartesian power and the endomorphism ordered set}
\index{Cartesian!power}\index{endomorphism!ordered set!Cartesian power}
For $k\geq2$ put $C_k=A^{\times k}+P$. The pair $(R,C_k)$ recovers
$A$ by Lemma~\ref{t:L:polynomial}, applied to $f(z)=z^k$.
It remains to recover $R$ without presupposing which summand is selected.
\begin{proposition}\label{b:P:power}
\index{Cartesian!power!proposition}
For $k\geq2$ and $r=0$ or $1$, the term $E_{r,C_k}$ with
$C_k=A^{\times k}+P$ is strong on $\OO$.
In particular, this proves Cases 1 and 2 of Table~\ref{b:tab:exceptions}
and the required ten-occurrence instance.
\end{proposition}
\begin{proof}
For (S1), the reduction chains from $Z_0$ are
\begin{equation}\label{b:E:power-chains}
\begin{aligned}
 N<k:&\quad S\times L^{\times k}\ \longmapsto\ S,\\
 N\geq k:&\quad S\times L^A\ \longmapsto\ S\times L
                 \ \longmapsto\ S.
\end{aligned}
\end{equation}
Within either branch terminal recovery gives $R$. Suppose that a
small-branch input $A$ and a~large-branch input $B$ give the same $V$.
Use input subscripts, and put $M=|B|\geq k$, $q_A=|R_A|$, $q_B=|R_B|$.
Equality after one further reduction from $Z_0$ gives
\[
 S_A\cong S_B\times L_B,\qquad J_A\cong J_B+R_B.
\]
Applying $\Tow$ to equality of $Z_0$ and cancelling $J_B$ gives
\[
 R_B+kR_A\cong R_B\times B,
 \qquad kq_A=(M-1)q_B.
\]
Applying $\Tow$ to equality of the original
$V=S\times L^{\times r}\times L^{C_k}$ gives
\[
 R_B+rR_A+(R_A\times C_k(A))
 \cong rR_B+(R_B\times C_k(B)).
\]
Taking cardinalities and eliminating $q_A/q_B$ yields
\begin{equation}\label{b:E:power-cross}
 k|C_k(B)|=(M-1)|C_k(A)|+k+r(M-1-k).
\end{equation}
Since $|A|<k$,
\[
 |C_k(A)|\leq(k-1)^k+(k-1)^{k-1}=k(k-1)^{k-1}.
\]
As $r\in\{0,1\}$, \eqref{b:E:power-cross} implies
\[
 |C_k(B)|\leq1+(M-1)\bigl((k-1)^{k-1}+1/k\bigr)
 <1+(M-1)M^{k-1}\leq M^k.
\]
The strict inequality uses $M\geq k\geq2$ and
$M^{k-1}>(k-1)^{k-1}+1/k$. This contradicts
$|C_k(B)|=M^k+|B^B|>M^k$. Thus branches cannot collide, proving (S1).
For (S2), recover the whole $C_k$ by \eqref{t:E:VZ} and
apply Lemma~\ref{t:L:polynomial}.
\end{proof}

\section{The cubic comparison and its offsets}
\index{cubic!comparison}\index{offset!cubic comparison}
The three sums in this section can be written as
\begin{equation}\label{b:E:cubic-C}
 C_{\alpha,\beta}=(\Ap\alpha\times H_1)+(\Ap\beta\times G),
 \qquad (\alpha,\beta)\in\{(0,0),(1,0),(0,1)\}.
\end{equation}
The case $(0,0)$ is used with $r=0,1$; the two offsets are used with
$r=0$. The argument does not depend on $r$.

\begin{proposition}\label{b:P:cubic}
\index{cubic!comparison!proposition}
Every $E_{r,C_{\alpha,\beta}}$ in \eqref{b:E:cubic-C} is strong on
$\OO$. This includes Cases 3, 4, and 6 of
Table~\ref{b:tab:exceptions} and the required ten-occurrence instance.
\end{proposition}
\begin{proof}
Call the branch selecting $A^{\times\alpha}\times H_1$ the first
branch and the branch selecting $A^{\times\beta}\times G$ the second.
A~height tie is resolved by the fixed order; either choice is allowed
in the argument. From $Z_0$ the chains are
\begin{equation}\label{b:E:cubic-chains}
\begin{aligned}
 \text{first: }& S\times L^{\times\alpha}\times L^{\Ap3}
       \ \longmapsto\ S\times L^{\times3}\ \longmapsto\ S,\\
 \text{second: }& S\times L^{\times\beta}\times L^P
       \ \longmapsto\ S\times L^A\ \longmapsto\ S\times L
       \ \longmapsto\ S.
\end{aligned}
\end{equation}
Thus terminal recovery determines $R$ within either branch.

Suppose $A$ lies in the first branch, $B$ in the second, and their
$Z_0$ values are isomorphic. Equality of the second subsequent values
in \eqref{b:E:cubic-chains} gives
\[
 S_A\cong S_B\times L_B,\qquad J_A\cong J_B+R_B.
\]
Applying $\Tow$ to equality of the first subsequent values, and then
to equality of $Z_0$, gives respectively
\begin{equation}\label{b:E:cubic-cross}
\begin{aligned}
 R_B+3R_A&\cong R_B\times B,\\
 R_B+\alpha R_A+(R_A\times A^{\times3})
   &\cong\beta R_B+(R_B\times P_B).
\end{aligned}
\end{equation}
The first identity forces $B$ to have an isolated point. Otherwise,
if $d$ is the least component height of $R_B$, all components on the
right have height greater than $d$, whereas the left contains a
component of height $d$. Write $B=\one+B'$ and set
$m=\comp(B')\geq1$. Sum cancellation gives
$3R_A\cong R_B\times B'$. Take three copies of the second identity in
\eqref{b:E:cubic-cross}, substitute this expression, and cancel $R_B$.
The result is
\begin{equation}\label{b:E:cubic-decisive}
 3P_B\cong3(1-\beta)\one+\alpha B'+(B'\times A^{\times3}).
\end{equation}

Since $A$ selects the first branch, $p_A\leq |A|^3+\alpha-\beta
\leq |A|^3+1$. The second bound in \eqref{b:E:basic-bounds} shows
$\comp(A)\leq2$: three or more components would give
$p_A\geq |A|^3+|A|^2$. Counting components in
\eqref{b:E:cubic-decisive} therefore gives
\begin{equation}\label{b:E:cubic-upper}
 3\comp(P_B)\leq3+9m.
\end{equation}
But $B\in\OO$, so \eqref{b:E:noncontractible-bound} gives
\[
 \comp(P_B)\geq(m+1)^{m+1}+(m+1)^m
      =(m+2)(m+1)^m\geq6m.
\]
For the last inequality, use $(m+1)^m\geq m+1$ and
$(m+2)(m+1)-6m=(m-1)(m-2)\geq0$ for integer $m\geq1$.
Thus $3\comp(P_B)\geq18m>3+9m$, contradicting
\eqref{b:E:cubic-upper}. This proves (S1), using $Z_0$ alone.

For (S2), recover the entire $C_{\alpha,\beta}$ from $(R,D_t)$.
In the first branch, cancellation in its $\Der$-image recovers
$\alpha\one+A^{\times3}$; remove $\alpha$ isolated points and take
a~Cartesian cube root. In the second, the same operation recovers
$\beta\one+P$; remove $\beta$ isolated points and use
$\Der(P)=R\times A$.

Finally, suppose the same pair $(R,C_{\alpha,\beta})$ comes from
opposite branches, with $A$ in the first. Equality of the $\Der$-images,
followed by cancellation, gives
\[
 \alpha\one+A^{\times3}\cong\beta\one+B^B.
\]
The positive-height summands supply the selected intervals. Applying
$\Der$ again gives $3R\cong R\times B$, forcing $B\cong3\one$,
contrary to positive height. This excludes the remaining collision
and proves (S2).
\end{proof}

\section{The ten-occurrence theorem}
\index{occurrence!count!ten: theorem}
\begin{theorem}\label{T:ten}
\index{occurrence!count!ten: theorem}
Every term with at most ten occurrences is reconstructing, and every
nontrivial such term is strong on $\OO$.
\end{theorem}
\begin{proof}
The preceding chapters cover all smaller counts. Uniform roots are
strong by Proposition~\ref{b:P:uniform}. Every regular root is strong
by Proposition~\ref{t:P:regular}, using the exhaustive comparisons of
Proposition~\ref{b:P:classification}. At exactly ten occurrences the
only remaining sums in that classification are $A^{\times5}+P$ and
$H_1+G$, with $r=0$. Propositions~\ref{b:P:power} and
\ref{b:P:cubic} prove their strong assertions. This exhausts all 719
normal forms. Theorem~\ref{b:T:K} and interchange normalization finish
the proof for all inputs and all 4,862 binary parenthesizations.
\end{proof}

\chapter{Eleven occurrences}\label{ch:eleven}

The eleven-occurrence classification has 1,842 normal forms:
1,226 uniform, 605 regular, and eleven exceptional. The exceptional
list is Table~\ref{b:tab:exceptions}. The two power cases and the
three cubic cases, including their offsets, have already been proved
in Chapter~\ref{ch:ten}. The remaining six cases are proved here.

\section{The quartic comparison}
\index{quartic!comparison}\index{offset!quartic comparison}
\begin{proposition}\label{b:P:quartic}
\index{quartic!comparison!proposition}
For $U_1=A^{A^{\times4}}$, the term $E_{0,U_1+G}$ is strong on
$\OO$.
\end{proposition}
\begin{proof}
The chains from $Z_0$ in the two branches are
\begin{equation}\label{b:E:quartic-chains}
\begin{aligned}
 U_1:&\quad S\times L^{\Ap4}\ \longmapsto\ S\times L^{\times4}
                      \ \longmapsto\ S,\\
 G:&\quad S\times L^P\ \longmapsto\ S\times L^A
                 \ \longmapsto\ S\times L\ \longmapsto\ S.
\end{aligned}
\end{equation}
Within a~branch, terminal recovery gives $R$.
Suppose $A$ selects $U_1$, $B$ selects $G$, and their $Z_0$ values
are isomorphic. The equalities in \eqref{b:E:quartic-chains}, followed
by $\Tow$, give
\begin{equation}\label{b:E:quartic-cross}
 R_B+4R_A\cong R_B\times B,
 \qquad
 R_B+(R_A\times A^{\times4})\cong R_B\times P_B.
\end{equation}
The least-component-height argument used above gives an isolated point
of $B$. Write $B=\one+B'$, and let $m=\comp(B')\geq1$.
The first identity gives $4R_A\cong R_B\times B'$. Take four copies
of the second and cancel $R_B$ to obtain
\begin{equation}\label{b:E:quartic-decisive}
 4P_B\cong4\one+(B'\times A^{\times4}).
\end{equation}
The branch condition is $p_A\leq |A|^4$, so
$a=\comp(A)\leq3$ by \eqref{b:E:basic-bounds}. Consequently
\begin{equation}\label{b:E:quartic-count}
 4\comp(P_B)=4+ma^4,\qquad 1\leq a\leq3.
\end{equation}

If $m=1$ or $m=2$, divisibility by four forces $a$ to be even, hence
$a=2$. For $m=1$, \eqref{b:E:quartic-count} gives $\comp(P_B)=5$,
but $B\in\OO$ has two components, so
$\comp(P_B)\geq2^2+2=6$. For $m=2$, the equation gives
$\comp(P_B)=9$, whereas $B$ has three components and
$\comp(P_B)\geq3^3=27$.

For $m\geq3$, the ordinary component bound gives
\[
 4\comp(P_B)\geq4(m+1)^{m+1}\geq256m>4+81m\geq4+ma^4.
\]
Indeed, $(m+1)^{m-1}\geq16$ and $(m+1)^2\geq4m$.
All possibilities contradict \eqref{b:E:quartic-count}. Thus opposite
branches cannot give the same $Z_0$, proving (S1).

For (S2), the whole sum $C=U_1+G$ is recovered as before. Within the
first branch, its $\Der$-image divided by $R$ is $A^{\times4}$;
within the second, it is $P$. Both recover $A$. Across the branches,
the same $R,C$ would give $A^{\times4}\cong B^B$; applying $\Der$
would give $4R\cong R\times B$, forcing $B\cong4\one$.
This contradicts positive height and completes the proof.
\end{proof}

\section{Five comparisons among the four-occurrence terms}
\index{four-occurrence!term!comparisons}\index{comparison!four-occurrence terms}
The remaining sums are
\begin{equation}\label{b:E:H-pairs}
 H_1+H_3,\quad H_1+H_4,\quad H_2+H_3,\quad
 H_2+H_4,\quad H_3+H_4.
\end{equation}
Each has $r=0$ and eleven occurrences.

\begin{proposition}\label{b:P:Hpairs}
\index{four-occurrence!term!five comparisons}
Every term $E_{0,C}$ with $C$ in \eqref{b:E:H-pairs} is strong on
$\OO$.
\end{proposition}
\begin{proof}
Starting at $Z_0=S\times\Max(C)$, the branch chains are displayed in
Table~\ref{b:tab:H-chains}. In this table only, juxtaposition denotes
Cartesian product; exponential notation retains its usual meaning.
An entry after $Z_0$ is obtained by one application of $\Red$ to the
preceding entry. A~dash means that no further reduction is needed.
Every branch reaches a~terminal pair with known Cartesian exponent,
so it determines $R$.

\begin{table}[htbp]
\centering
\small
\setlength{\tabcolsep}{5pt}
\begin{tabular}{@{}cccccc@{}}
\toprule
Branch & $Z_0$ & $Z_1$ & $Z_2$ & $Z_3$ & $Z_4$\\
\midrule
$H_1$ & $SL^{A^{\times3}}$ & $SL^{\times3}$ & $S$ & --- & ---\\
$H_2$ & $SL^{A\times P}$ & $SLL^A$ & $SL$ & $S$ & ---\\
$H_3$ & $SL^F$ & $SL^{A^{\times2}}$ & $SL^{\times2}$ & $S$ & ---\\
$H_4$ & $SL^G$ & $SL^P$ & $SL^A$ & $SL$ & $S$\\
\bottomrule
\end{tabular}
\caption{The complete reduction chains for the five exceptional pairs.}
\label{b:tab:H-chains}
\end{table}

We prove that opposite branches cannot give the same $Z_0$.
In each pair, let $A$ select the first displayed $H_i$ and $B$ the
second. Lemma~\ref{b:L:disconnected} shows that $A$ is connected:
on disconnected inputs the second member has strictly greater height
in every pair of \eqref{b:E:H-pairs}. Thus
\begin{equation}\label{b:E:H-connected}
 \comp(A)=1,\qquad \iso(A)=0.
\end{equation}
Throughout these comparisons, set
\[
\begin{aligned}
 \rho&=\comp(R_A),& \sigma&=\comp(R_B),& b&=\comp(B),\\
 u&=\comp(P_B),& v&=\comp(F_B),& w&=\comp(G_B).
\end{aligned}
\]
The integers $\rho,\sigma$ are positive. Lemma~\ref{b:L:bounds} gives
$u\geq b^b$ and $w\geq u$.

\smallskip\noindent
\emph{The pair $H_1,H_3$.}
Equality of $Z_2$ gives $S_A\cong S_B\times L_B^{\times2}$ and
$J_A\cong J_B+2R_B$. Apply $\Tow$ to the equalities of $Z_1$ and
$Z_0$, and cancel $J_B$:
\[
 2R_B+3R_A\cong R_B\times B^{\times2},\qquad
 2R_B+(R_A\times A^{\times3})\cong R_B\times F_B.
\]
Using \eqref{b:E:H-connected} and counting components gives
$2\sigma+3\rho=\sigma b^2$ and
$2\sigma+\rho=\sigma v$. Eliminating $\rho/\sigma$ yields
\[
 3v=b^2+4.
\]
This is impossible modulo three, since a~square cannot be congruent
to two modulo three.

\smallskip\noindent
\emph{The pair $H_1,H_4$.}
Equality of $Z_2$ gives $S_A\cong S_B\times L_B^B$, hence
$J_A\cong J_B+(R_B\times B)$. From the equalities of $Z_1,Z_0$,
\[
 (R_B\times B)+3R_A\cong R_B\times P_B,
 \qquad
 (R_B\times B)+(R_A\times A^{\times3})\cong R_B\times G_B.
\]
Component counts give $b\sigma+3\rho=u\sigma$ and
$b\sigma+\rho=w\sigma$. Thus $u>b$ and $3w=u+2b$.
Since $w\geq u$, the latter equation forces $u\leq b$, a~contradiction.

\smallskip\noindent
\emph{The pair $H_2,H_3$.}
Equality of $Z_3$ gives $S_A\cong S_B$. Cancel this common factor
from equality of $Z_2$ to get $L_A\cong L_B^{\times2}$, whence
$R_A\cong2R_B$. Cancel the common $S$ from equality of $Z_1$ and apply
$\Tow$. The result is
\[
 2R_B+(2R_B\times A)\cong R_B\times B^{\times2}.
\]
Cancel $R_B$ to obtain $B^{\times2}\cong2\one+2A$.
Its isolated-point count is $\iso(B)^2=2$, by
\eqref{b:E:H-connected}. This is impossible.

\smallskip\noindent
\emph{The pair $H_2,H_4$.}
Equality of $Z_3$ gives $S_A\cong S_B\times L_B$, hence
$J_A\cong J_B+R_B$. The equalities of $Z_2,Z_1$ give
\[
 R_B+R_A\cong R_B\times B,
 \qquad
 R_B+R_A+(R_A\times A)\cong R_B\times P_B.
\]
Consequently $\sigma+\rho=b\sigma$ and
$\sigma+2\rho=u\sigma$. Thus $b\geq2$ and $u=2b-1$.
But $u\geq b^b\geq b^2>2b-1$, a~contradiction.

\smallskip\noindent
\emph{The pair $H_3,H_4$.}
Again $S_A\cong S_B\times L_B$ and $J_A\cong J_B+R_B$.
The equalities of $Z_2,Z_1$ now give
\[
 R_B+2R_A\cong R_B\times B,
 \qquad
 R_B+(R_A\times A^{\times2})\cong R_B\times P_B.
\]
Hence $\sigma+2\rho=b\sigma$ and $\sigma+\rho=u\sigma$.
They imply $b\geq2$ and $2u=b+1$, again contradicting $u\geq b^b$.
This completes all five cross-branch comparisons and proves (S1).

\smallskip
For (S2), first recover the entire $C$ from $(R,D_t)$.
Write $H_i=A^{Q_i}$ for $1\leq i\leq4$, where
\[
 (Q_1,Q_2,Q_3,Q_4)=(A^{\times3},A\times P,F,G).
\]
Within branch $i$, cancellation in $\Der(C)=R\times Q_i$ recovers
$Q_i$. For $Q_1$, take a~Cartesian cube root. For $Q_2$, compute
$\Der(Q_2)=R\times(\one+A)$ and remove the one known isolated point
after cancellation. For $Q_3$, use $\Der(F)=R\times A^{\times2}$ and
a~Cartesian square root. For $Q_4$, use $\Der(G)=R\times P$, followed
by $\Der(P)=R\times A$.

It remains to exclude opposite branches with the same $R,C$. After
cancelling $R$ in their $\Der$-images, these would give
$Q_i(A)\cong Q_j(B)$. Each possibility is impossible, as follows.

For $(i,j)=(1,3)$, applying $\Der$ gives
$3R\cong R\times B^{\times2}$, so $B^{\times2}\cong3\one$, contrary
to positive height. For $(1,4)$ it gives
$3R\cong R\times P_B$, so $P_B\cong3\one$, again impossible.
For $(2,3)$ it gives $\one+A\cong B^{\times2}$; applying $\Der$
once more gives $R\cong2R$, impossible by cardinality.
For $(2,4)$ it gives $\one+A\cong P_B$; applying $\Der$ again gives
$R\cong R\times B$, forcing $B\cong\one$.
Finally, for $(3,4)$ it gives $A^{\times2}\cong P_B$; applying
$\Der$ again gives $2R\cong R\times B$, forcing $B\cong2\one$.
The last two conclusions contradict positive height. Thus (S2) holds
in all five cases.
\end{proof}

\section{The cumulative theorem through eleven}
\index{reconstruction!through eleven occurrences}\index{occurrence!count!eleven: cumulative theorem}
\begin{theorem}\label{b:T:main}
\index{occurrence!count!eleven: theorem}
Every exponential term with at most eleven occurrences is
reconstructing. Every nontrivial such term is strong on $\OO$.
\end{theorem}
\begin{proof}
The assertion through ten is Theorem~\ref{T:ten}. At eleven, uniform
roots follow from the smaller child's strong assertion and
Proposition~\ref{b:P:uniform}. Regular roots follow from
Propositions~\ref{b:P:classification} and~\ref{t:P:regular}.
The exceptional table is exhaustive by the same classification.
Its Cases~1--2 are Proposition~\ref{b:P:power}; Cases~3, 4, and~6
are Proposition~\ref{b:P:cubic}; Case~5 is
Proposition~\ref{b:P:quartic}; Cases~7--11 are
Proposition~\ref{b:P:Hpairs}. Every case therefore has both strong
assertions. Strong reconstruction handles $\OO$, and
Theorem~\ref{b:T:K} handles all other inputs. Interchange
normalization proves the assertion for every parenthesization.
\end{proof}

\section{Enumeration and primality}
\index{enumeration!eleven occurrences}\index{primality!eleven}
The 1,842 normal forms account for 16,796 binary parenthesizations.
The uniform, regular, and exceptional proof classes account for
10,077, 6,681, and 38 binary parenthesizations, respectively. Their
root partitions are listed in Appendix~\ref{app:partitions}. The
independent binary-tree audit checks every individual normal-form
multiplicity, not just these totals.

Since eleven is prime, proper uniform substitution with both term
sizes greater than one is unavailable. The proof uses the additive
root budget instead. The theorem includes all smaller occurrence
counts because their strong assertions have been proved in the
cumulative induction, not because an exact eleven-occurrence theorem
has been assumed to imply them.

\chapter[Twelve: classification]{Twelve occurrences: the complete root classification}\label{t:S:classification}
Write $\mathcal D_j$ for the ordered list of forests of weight $j$
in the following tables, and write $D_{j,i}$ for its $i$th entry.
Define
\begin{equation}\label{t:E:UW}
 U_i=A^{D_{4,i}}\quad(1\leq i\leq9),\qquad
 W_j=A^{D_{5,j}}\quad(1\leq j\leq20).
\end{equation}
Thus $U_1=A^{\Ap4}$, $U_2=A^{\Ap2P}$,
$U_3=A^{\xp P2}$, $U_4=A^{AF}$, $U_5=A^{AG}$,
$U_6=A^{H_1}$, and $W_1=A^{\Ap5}$. Put $u_i=|U_i|$.

In a~comparison string, \texttt{R} means that the row forest supplies
the selected interval, \texttt{C} means that the column forest does,
and \texttt{?} denotes a~pair proved separately. Columns are in the
listed order of $\mathcal D_3$ or $\mathcal D_4$. Heights are divided
by $\ell$. The order of $\mathcal D_3$ is
\[
 (\Ap3,\ AP,\ F,\ G),\qquad (3,\ N+1,\ N^2,\ p)
\]
for its corresponding height coefficients.
\begingroup\small\renewcommand{\arraystretch}{1.10}
\begin{longtable}{@{}rlrc@{}}
\toprule
$i$ & Forest & $\height/\ell$ & Against $\mathcal D_3$\\
\midrule
\endfirsthead
\multicolumn{4}{c}{\textit{The weight-four forests, continued}}\\
\toprule
$i$ & Forest & $\height/\ell$ & Against $\mathcal D_3$\\
\midrule
\endhead
\endfoot
\endlastfoot
1 & $\Ap{4}$ & $4$ & \texttt{RCCC}\\
2 & $\Ap{2}\,P$ & $N+2$ & \texttt{RRCC}\\
3 & $\xp{P}{2}$ & $2N$ & \texttt{RRCC}\\
4 & $A\,F$ & $N^{2}+1$ & \texttt{RRRC}\\
5 & $A\,G$ & $p+1$ & \texttt{RRRR}\\
6 & $H_1$ & $N^{3}$ & \texttt{RRR?}\\
7 & $H_2$ & $Np$ & \texttt{RRRR}\\
8 & $H_3$ & $f$ & \texttt{RRRR}\\*
9 & $H_4$ & $g$ & \texttt{RRRR}\\*
\bottomrule
\caption{The weight-four forests}\label{t:tab:D4}
\end{longtable}\endgroup

Within $\mathcal D_4$, the later forest wins except for the four
unordered pairs
\begin{equation}\label{t:E:44exceptions}
 (AG,H_1),\quad(H_1,H_3),\quad(H_2,H_3),\quad(H_3,H_4).
\end{equation}
The inequalities $g>Np$ make the pairs $(H_1,H_4)$ and $(H_2,H_4)$
regular; they no longer require separate proofs.
\begingroup\small\renewcommand{\arraystretch}{1.10}
\begin{longtable}{@{}rlrcc@{}}
\toprule
$i$ & Forest & $\height/\ell$ & Against $\mathcal D_3$ & Against $\mathcal D_4$\\
\midrule
\endfirsthead
\multicolumn{5}{c}{\textit{The weight-five forests and their comparisons, continued}}\\
\toprule
$i$ & Forest & $\height/\ell$ & Against $\mathcal D_3$ & Against $\mathcal D_4$\\
\midrule
\endhead
\endfoot
\endlastfoot
1 & $\Ap{5}$ & $5$ & \texttt{RCCC} & \texttt{RCCCCCCCC}\\
2 & $\Ap{3}\,P$ & $N+3$ & \texttt{RRCC} & \texttt{RRCCCCCCC}\\
3 & $A\,\xp{P}{2}$ & $2N+1$ & \texttt{RRCC} & \texttt{RRRCCCCCC}\\
4 & $\Ap{2}\,F$ & $N^{2}+2$ & \texttt{RRRC} & \texttt{RRRRCCCCC}\\
5 & $\Ap{2}\,G$ & $p+2$ & \texttt{RRRR} & \texttt{RRRRR?CCC}\\
6 & $P\,F$ & $N^{2}+N$ & \texttt{RRRC} & \texttt{RRRRCCCCC}\\
7 & $P\,G$ & $N+p$ & \texttt{RRRR} & \texttt{RRRRR?CCC}\\
8 & $A\,H_1$ & $N^{3}+1$ & \texttt{RRR?} & \texttt{RRRR?RC?C}\\
9 & $A\,H_2$ & $Np+1$ & \texttt{RRRR} & \texttt{RRRRRRR?C}\\
10 & $A\,H_3$ & $f+1$ & \texttt{RRRR} & \texttt{RRRRR??R?}\\
11 & $A\,H_4$ & $g+1$ & \texttt{RRRR} & \texttt{RRRRRRR?R}\\
12 & $U_{1}$ & $N^{4}$ & \texttt{RRR?} & \texttt{RRRR?R???}\\
13 & $U_{2}$ & $N^{2}p$ & \texttt{RRRR} & \texttt{RRRRRRR??}\\
14 & $U_{3}$ & $p^{2}$ & \texttt{RRRR} & \texttt{RRRRRRR??}\\
15 & $U_{4}$ & $Nf$ & \texttt{RRRR} & \texttt{RRRRRRRR?}\\
16 & $U_{5}$ & $Ng$ & \texttt{RRRR} & \texttt{RRRRRRR?R}\\
17 & $U_{6}$ & $h_{1}$ & \texttt{RRRR} & \texttt{RRRRR??R?}\\
18 & $U_{7}$ & $h_{2}$ & \texttt{RRRR} & \texttt{RRRRRRRRR}\\
19 & $U_{8}$ & $h_{3}$ & \texttt{RRRR} & \texttt{RRRRRRRRR}\\*
20 & $U_{9}$ & $h_{4}$ & \texttt{RRRR} & \texttt{RRRRRRRRR}\\*
\bottomrule
\caption{The weight-five forests and their comparisons}\label{t:tab:D5}
\end{longtable}\endgroup
\begingroup\small\renewcommand{\arraystretch}{1.10}
\begin{longtable}{@{}rlrc@{}}
\toprule
$i$ & Forest & $\height/\ell$ & Against $\mathcal D_3$\\
\midrule
\endfirsthead
\multicolumn{4}{c}{\textit{The weight-six forests and their comparisons, continued}}\\
\toprule
$i$ & Forest & $\height/\ell$ & Against $\mathcal D_3$\\
\midrule
\endhead
\endfoot
\endlastfoot
1 & $\Ap{6}$ & $6$ & \texttt{R?CC}\\
2 & $\Ap{4}\,P$ & $N+4$ & \texttt{RRCC}\\
3 & $\Ap{2}\,\xp{P}{2}$ & $2N+2$ & \texttt{RRCC}\\
4 & $\xp{P}{3}$ & $3N$ & \texttt{RRCC}\\
5 & $\Ap{3}\,F$ & $N^{2}+3$ & \texttt{RRRC}\\
6 & $\Ap{3}\,G$ & $p+3$ & \texttt{RRRR}\\
7 & $A\,P\,F$ & $N^{2}+N+1$ & \texttt{RRRC}\\
8 & $A\,P\,G$ & $N+p+1$ & \texttt{RRRR}\\
9 & $\xp{F}{2}$ & $2N^{2}$ & \texttt{RRR?}\\
10 & $F\,G$ & $N^{2}+p$ & \texttt{RRRR}\\
11 & $\xp{G}{2}$ & $2p$ & \texttt{RRRR}\\
12 & $\Ap{2}\,H_1$ & $N^{3}+2$ & \texttt{RRR?}\\
13 & $\Ap{2}\,H_2$ & $Np+2$ & \texttt{RRRR}\\
14 & $\Ap{2}\,H_3$ & $f+2$ & \texttt{RRRR}\\
15 & $\Ap{2}\,H_4$ & $g+2$ & \texttt{RRRR}\\
16 & $P\,H_1$ & $N^{3}+N$ & \texttt{RRR?}\\
17 & $P\,H_2$ & $Np+N$ & \texttt{RRRR}\\
18 & $P\,H_3$ & $N+f$ & \texttt{RRRR}\\
19 & $P\,H_4$ & $N+g$ & \texttt{RRRR}\\
20 & $A\,U_{1}$ & $N^{4}+1$ & \texttt{RRR?}\\
21 & $A\,U_{2}$ & $N^{2}p+1$ & \texttt{RRRR}\\
22 & $A\,U_{3}$ & $p^{2}+1$ & \texttt{RRRR}\\
23 & $A\,U_{4}$ & $Nf+1$ & \texttt{RRRR}\\
24 & $A\,U_{5}$ & $Ng+1$ & \texttt{RRRR}\\
25 & $A\,U_{6}$ & $h_{1}+1$ & \texttt{RRRR}\\
26 & $A\,U_{7}$ & $h_{2}+1$ & \texttt{RRRR}\\
27 & $A\,U_{8}$ & $h_{3}+1$ & \texttt{RRRR}\\
28 & $A\,U_{9}$ & $h_{4}+1$ & \texttt{RRRR}\\
29 & $W_{1}$ & $N^{5}$ & \texttt{RRR?}\\
30 & $W_{2}$ & $N^{3}p$ & \texttt{RRRR}\\
31 & $W_{3}$ & $Np^{2}$ & \texttt{RRRR}\\
32 & $W_{4}$ & $N^{2}f$ & \texttt{RRRR}\\
33 & $W_{5}$ & $N^{2}g$ & \texttt{RRRR}\\
34 & $W_{6}$ & $pf$ & \texttt{RRRR}\\
35 & $W_{7}$ & $pg$ & \texttt{RRRR}\\
36 & $W_{8}$ & $Nh_{1}$ & \texttt{RRRR}\\
37 & $W_{9}$ & $Nh_{2}$ & \texttt{RRRR}\\
38 & $W_{10}$ & $Nh_{3}$ & \texttt{RRRR}\\
39 & $W_{11}$ & $Nh_{4}$ & \texttt{RRRR}\\
40 & $W_{12}$ & $u_{1}$ & \texttt{RRRR}\\
41 & $W_{13}$ & $u_{2}$ & \texttt{RRRR}\\
42 & $W_{14}$ & $u_{3}$ & \texttt{RRRR}\\
43 & $W_{15}$ & $u_{4}$ & \texttt{RRRR}\\
44 & $W_{16}$ & $u_{5}$ & \texttt{RRRR}\\
45 & $W_{17}$ & $u_{6}$ & \texttt{RRRR}\\
46 & $W_{18}$ & $u_{7}$ & \texttt{RRRR}\\
47 & $W_{19}$ & $u_{8}$ & \texttt{RRRR}\\*
48 & $W_{20}$ & $u_{9}$ & \texttt{RRRR}\\*
\bottomrule
\caption{The weight-six forests and their comparisons}\label{t:tab:D6}
\end{longtable}\endgroup

\section{Verification of the fixed comparisons}
\index{verification!fixed comparisons}\index{comparison!verification}
The height formulas
\[
 \height(XY)=\height(X)+\height(Y),\qquad
 \height(X^Y)=|Y|\height(X)
\]
give every height coefficient in the tables. Their fixed comparisons
follow from Lemma~\ref{t:L:bounds}. For specificity, the following list
is sufficient for all the table entries:
\begin{equation}\label{t:E:certificatebounds}
\begin{aligned}
 N&\geq4,&p&\geq N^2+N+1,&f&\geq p+N,\\
 g&\geq p+N,&g&\geq Np+1,&h_1&\geq f+N-1,\\
 h_2&\geq p^2+p,&h_2&\geq f+1,&h_2&\geq g+1,\\
 h_3&\geq f+N-1,&h_3&\geq Ng+1,\\
 h_4&\geq g+N-1,&h_4&\geq Nf+1,&u_i&\geq p+1.
\end{aligned}
\end{equation}
An exact arithmetic check, also reproducible from these displayed
bounds, is included with the enumeration. For each asserted comparison,
subtract the losing height coefficient from the winning coefficient.
Eliminate the variables in reverse order
$u_9,\ldots,u_1,h_4,\ldots,h_1,g,f,p,N$, replacing a~variable by one
of its displayed lower bounds plus a~new nonnegative variable. One
choice yields a~polynomial all of whose coefficients are nonnegative.
The program records all 457 resulting polynomial certificates for the
four displayed comparison matrices. This is a~check in an integer
polynomial ring, not a~test at sampled numerical values.

There are only five non-strict certificates. At $N=4$, $AP$ is selected
over $\Ap5$, by comparing the exponents $\one+A$ and $5\one$.
For $APF$ versus $G$, a~height tie means $p=N^2+N+1$;
$\one+A+\Ap2$ has height $2\ell$, whereas $P$ has height $N\ell$,
so $G$ is selected. For $PG$ versus $H_3$, compare the exponents
$A+P$ and $F$, of heights $N\ell$ and $N^2\ell$.
For $PG$ versus $H_4$, the stronger bound $g>Np>p+N$ makes the
comparison strict. For $AH_2$ versus $H_4$, a~height tie means
$g=Np+1$; compare $\one+AP$ and $G$, of heights $(N+1)\ell$ and
$p\ell$. Thus the prescribed winner is selected in every tie.
The separate comparison of $P$ with $\Ap4$ at $N=4$ uses the
positive height of its exponent $A$.

\begin{proposition}\label{t:P:classification}
\index{classification!twelve occurrences}
At twelve occurrences, 3,068 normal forms are uniform extensions,
1,658 satisfy Proposition~\ref{t:P:regular}, and precisely the 40 forms
in Table~\ref{t:tab:exceptions} require the exceptional arguments below.
\end{proposition}
\begin{proof}
Let $m$ be the number of nontrivial root children. The uniform forms
have at most one distinct nontrivial child type. Otherwise use
\eqref{t:E:budget}. For $m=2$, the summand weights have total at most
nine. Weight one always loses to a~larger forest. At weight two,
$P$ beats $\Ap2$; it loses to every higher-weight nonpure forest.
Against a~pure forest it wins through weight four, and the remaining
pairs are $P+\Ap k$ for $k=5,6,7$. The other possible weight pairs
are $3+3$, $3+4$, $3+5$, $3+6$, $4+4$, and $4+5$, all covered by
the lists and matrices above. There are 38 exceptional pair types.

For $m=3$, the total summand weight is at most eight. The only
exceptional pair that can occur with a~third summand is $P+\Ap5$
or $G+H_1$; each has total weight seven, so the third summand must
be $A$. These give the two additional forms. For $m\geq4$, the budget
allows no exceptional pair. Every regular selection has weight at
most eight because another positive-weight summand is present.
The comparisons show that a~pure selection occurs only when every
summand is pure.

This gives an exhaustive, finite decision procedure with a~proof
assigned to every outcome. Chapter~\ref{t:S:enumeration} counts its
outcomes by weighted multisets and checks them against all binary
parenthesizations. It gives the stated counts and the complete
40-row exceptional list.
\end{proof}
\begingroup\small\renewcommand{\arraystretch}{1.07}
\begin{longtable}{@{}rrlrr@{}}
\toprule
No. & $r$ & $C(A)$ & Proposition & Binary\\
\midrule\endfirsthead
\multicolumn{5}{c}{\textit{The exceptional forms, continued}}\\
\toprule No. & $r$ & $C(A)$ & Proposition & Binary\\\midrule\endhead
\endfoot\endlastfoot
1 & 2 & $P+\Ap{5}$ & \ref{t:P:power} & 12\\
2 & 2 & $G+H_1$ & \ref{t:P:cubic} & 12\\
3 & 1 & $P+\Ap{6}$ & \ref{t:P:power} & 6\\
4 & 1 & $G+A\,H_1$ & \ref{t:P:cubic} & 12\\
5 & 1 & $G+U_{1}$ & \ref{t:P:quartic} & 6\\
6 & 1 & $A\,G+H_1$ & \ref{t:P:cubic} & 12\\
7 & 1 & $H_1+H_3$ & \ref{t:P:Hoffsets} & 6\\
8 & 1 & $H_2+H_3$ & \ref{t:P:Hoffsets} & 12\\
9 & 1 & $H_3+H_4$ & \ref{t:P:Hoffsets} & 6\\
10 & 0 & $A+P+\Ap{5}$ & \ref{t:P:power} & 6\\
11 & 0 & $A+G+H_1$ & \ref{t:P:cubic} & 6\\
12 & 0 & $P+\Ap{7}$ & \ref{t:P:power} & 2\\
13 & 0 & $A\,P+\Ap{6}$ & \ref{t:P:twisted} & 4\\
14 & 0 & $G+\Ap{2}\,H_1$ & \ref{t:P:cubic} & 6\\
15 & 0 & $G+A\,U_{1}$ & \ref{t:P:quartic} & 4\\
16 & 0 & $G+P\,H_1$ & \ref{t:P:cubic} & 4\\
17 & 0 & $G+\xp{F}{2}$ & \ref{t:P:double} & 2\\
18 & 0 & $G+W_{1}$ & \ref{t:P:quintic} & 2\\
19 & 0 & $A\,G+A\,H_1$ & \ref{t:P:cubic} & 8\\
20 & 0 & $A\,G+U_{1}$ & \ref{t:P:quartic} & 4\\
21 & 0 & $H_1+\Ap{2}\,G$ & \ref{t:P:cubic} & 6\\
22 & 0 & $H_1+A\,H_3$ & \ref{t:P:Hoffsets} & 4\\
23 & 0 & $H_1+P\,G$ & \ref{t:P:cubic} & 4\\
24 & 0 & $H_1+U_{6}$ & \ref{t:P:bulk} & 2\\
25 & 0 & $H_2+A\,H_3$ & \ref{t:P:Hoffsets} & 8\\
26 & 0 & $H_2+U_{1}$ & \ref{t:P:unit} & 4\\
27 & 0 & $H_2+U_{6}$ & \ref{t:P:bulk} & 4\\
28 & 0 & $H_3+A\,H_1$ & \ref{t:P:Hoffsets} & 4\\
29 & 0 & $H_3+A\,H_2$ & \ref{t:P:Hoffsets} & 8\\
30 & 0 & $H_3+A\,H_4$ & \ref{t:P:Hoffsets} & 4\\
31 & 0 & $H_3+U_{1}$ & \ref{t:P:bulk} & 2\\
32 & 0 & $H_3+U_{2}$ & \ref{t:P:bulk} & 6\\
33 & 0 & $H_3+U_{5}$ & \ref{t:P:bulk} & 4\\
34 & 0 & $H_3+U_{3}$ & \ref{t:P:square} & 2\\
35 & 0 & $H_4+A\,H_3$ & \ref{t:P:Hoffsets} & 4\\
36 & 0 & $H_4+U_{1}$ & \ref{t:P:bulk} & 2\\
37 & 0 & $H_4+U_{2}$ & \ref{t:P:bulk} & 6\\
38 & 0 & $H_4+U_{4}$ & \ref{t:P:bulk} & 4\\
39 & 0 & $H_4+U_{3}$ & \ref{t:P:bulk} & 2\\*
40 & 0 & $H_4+U_{6}$ & \ref{t:P:shift} & 2\\*
\bottomrule
\caption{All 40 exceptional normal forms. Each denotes $E_{r,C}$ in
\eqref{t:E:mixed}; the last column counts its binary parenthesizations.}
\label{t:tab:exceptions}
\end{longtable}\endgroup

\section[Twelve: polynomial cases]{Twelve occurrences: the polynomial cases}\label{t:S:powers}

The five exceptional forms considered here have a~possible pure-forest
selection. The whole-sum recovery result required in their second
strong assertions was proved in Lemma~\ref{t:L:polynomial}. We now
exclude collisions between their first-invariant branches.
\index{exceptional!cases!polynomial}
\begin{proposition}\label{t:P:power}
\index{Cartesian!power!twelve-occurrence comparison}
The forms with $C=\Ap k+P$, $k=5,6,7$ and $r=7-k$, and the form
with $C=\Ap5+A+P$, $r=0$, are strong.
\end{proposition}
\begin{proof}
Write $C=\Ap k+\varepsilon A+P$, where $\varepsilon=0$ for the
first three forms and $(k,\varepsilon,r)=(5,1,0)$ for the last.
In the small range $N<k$, $Z_0=SL^{\times k}$; in the large range
$N\geq k$, $Z_0=SL^A$. Each branch separately determines $R$.
We rule out equality of the original invariant $V$ across the branches.

Suppose $A$ is in the small branch and $B$ in the large branch, with
$M=|B|\geq k$. Their reduction chains give
$S_A=S_BL_B$ and $J_A=J_B+R_B$. Applying $\Tow$ to $Z_0$ gives
\[
 R_B+kR_A=R_BB.
\]
Applying it also to the original $V$ gives
$R_B+rR_A+R_AC(A)=rR_B+R_BC(B)$. Taking cardinalities in the two
equalities and eliminating $|R_A|/|R_B|$ yields
\begin{equation}\label{t:E:powercross}
 k|C(B)|=(M-1)|C(A)|+k+r(M-1-k).
\end{equation}
For the small input,
\[
 |C(A)|\leq k(k-1)^{k-1}+\varepsilon(k-1).
\]
In our four cases $[\varepsilon(k-1)+r]/k\leq1$. Consequently
\eqref{t:E:powercross} gives
\[
\begin{aligned}
 |C(B)|&\leq1+(M-1)\bigl((k-1)^{k-1}+1\bigr)\\
       &<1+(M-1)M^{k-1}\leq M^k,
\end{aligned}
\]
contrary to $|C(B)|>M^k$. The branches therefore cannot collide,
and the first strong assertion follows.
For the second, recover the whole $C$ by~\eqref{t:E:wholeC} and apply
Lemma~\ref{t:L:polynomial}.
\end{proof}

\begin{proposition}\label{t:P:twisted}
\index{twisted!comparison}
The form with $C=\Ap6+AP$ and $r=0$ is strong.
\end{proposition}
\begin{proof}
The pure-power branch occurs only at $N=4$. For $N\geq5$, $AP$ is
selected, including the tie at $N=5$, since the exponent $\one+A$
has positive height. Each branch separately recovers $R$. Suppose
a~small input $A$ and a~large input $B$, $M=|B|\geq5$, give the
same $V$. The chains give $S_A=S_BL_B$. Their $\Tow$ equations are
\[
 6R_A=R_BB,\qquad R_B+R_AC(A)=R_BC(B).
\]
Thus $6|C(B)|=6+M|C(A)|$. Since $|C(A)|\leq4^6+4\cdot4^4=5120$
and $|C(B)|>M^6$, this contradicts
$6M^6>6+5120M$ for $M\geq5$. The latter holds already at $M=5$
after division by $M$, and the left side $6M^5$ increases.

For the second strong assertion, recover $C$.
The type $\Der(C)/R$ is $6\one$ in the small branch and $\one+A$
in the large branch. These have different heights, so the branch is
recognized. In the large branch remove one isolated point. In the
small branch, two inputs giving the same $C$ both have cardinality
four and the same height $\ell$. If their component polynomials
$a,b$ differ, $\hdeg(a-b)\geq1$ as in the preceding lemma. Therefore
\[
 \hdeg(a^6-b^6)\geq5\ell+1.
\]
Yet the sum equality gives
$a^6-b^6=b\poly(B^B)-a\poly(A^A)$, of degree at most $5\ell$.
This contradiction proves recovery.
\end{proof}

section[Twelve: cubic to quintic cases]{Twelve occurrences: cubic, quartic, and quintic cases}\label{t:S:cubic}
Except for the polynomial cases just proved, every possible selected
forest in an exceptional sum is nonpure. Table~\ref{t:tab:chains} gives
the chains and signatures used to prove the second strong assertion.
Multiplying a~listed nonpure forest by $A^{\times\alpha}$ changes
$K(D)$ to $\alpha\one+K(D)$ and does not change its chain after the
first forest or its signature. The forests $PH_1$ and $PG$ likewise
have the signatures of $H_1$ and $G$, respectively.
\begingroup\small\renewcommand{\arraystretch}{1.1}
\begin{longtable}{@{}llll@{}}
\toprule $D$ & $K(D)$ & Chain after $D$ & $\sig(D)$\\\midrule\endfirsthead
\toprule $D$ & $K(D)$ & Chain after $D$ & $\sig(D)$\\\midrule\endhead
\endfoot\endlastfoot
$A\,P$ & $\one+A$ & $A$ & $(2,1)$\\
$F$ & $\Ap{2}$ & $\Ap{2}$ & $(2,2)$\\
$\xp{F}{2}$ & $\Ap{2}+\Ap{2}$ & $\Ap{2}$ & $(2,2)$\\
$P$ & $A$ & $A$ & $(2,1)$\\
$G$ & $P$ & $P,\ A$ & $(3,1)$\\
$H_1$ & $\Ap{3}$ & $\Ap{3}$ & $(2,3)$\\
$H_2$ & $A\,P$ & $A\,P,\ A$ & $(3,1)$\\
$H_3$ & $F$ & $F,\ \Ap{2}$ & $(3,2)$\\
$H_4$ & $G$ & $G,\ P,\ A$ & $(4,1)$\\
$U_{1}$ & $\Ap{4}$ & $\Ap{4}$ & $(2,4)$\\
$U_{2}$ & $\Ap{2}\,P$ & $\Ap{2}\,P,\ A$ & $(3,1)$\\
$U_{3}$ & $\xp{P}{2}$ & $\xp{P}{2},\ A$ & $(3,1)$\\
$U_{4}$ & $A\,F$ & $A\,F,\ \Ap{2}$ & $(3,2)$\\
$U_{5}$ & $A\,G$ & $A\,G,\ P,\ A$ & $(4,1)$\\
$U_{6}$ & $H_1$ & $H_1,\ \Ap{3}$ & $(3,3)$\\*
$W_{1}$ & $\Ap{5}$ & $\Ap{5}$ & $(2,5)$\\*
\bottomrule
\caption{The chains needed for the exceptional arguments. A~pure
forest $\Ap k$ has signature $(1,k)$ and $K(\Ap k)=k\one$.}
\label{t:tab:chains}
\end{longtable}\endgroup

The signatures of the two possible selections are distinct in every
remaining exceptional pair; the table also checks this for the
families with offsets. Thus Lemma~\ref{t:L:signature} proves the second
strong assertion as soon as the whole $C$ has been recovered. In what
follows it suffices to prove the first strong assertion by excluding
cross-branch equality of $Z_0$. When both inputs select the same
forest, its fixed chain already recovers $R$.

We shall repeatedly use the following observation. An equality
\begin{equation}\label{t:E:isolated}
 R_B+kR_A=R_BB
\end{equation}
forces $B$ to have an isolated point. Otherwise the least height of
a~component on the right is strictly greater than the least height
of a~component of $R_B$, whereas the left contains $R_B$ as a~summand.
Write $B=\one+B'$, put $m=\comp(B')\geq1$, and cancel $R_B$ as
a~summand in~\eqref{t:E:isolated}. Then $kR_A=R_BB'$. If
$u=\comp(B^B)$, Lemma~\ref{t:L:retract} gives
\begin{equation}\label{t:E:um}
 u\geq(m+2)(m+1)^m\geq6m.
\end{equation}
Here $(m+2)(m+1)\geq6m$ for every positive integer $m$.

\begin{proposition}\label{t:P:cubic}
\index{cubic!comparison!twelve-occurrence audit}
The six sums
\[
 C=\Ap\alpha H_1+\Ap\beta G,\qquad
 \alpha,\beta\geq0,\quad\alpha+\beta\leq2,
 \quad r=2-\alpha-\beta,
\]
the sums $PH_1+G$ and $H_1+PG$ with $r=0$, and the sum
$A+H_1+G$ with $r=0$, give strong forms.
\end{proposition}
\begin{proof}
First consider the six sums. Suppose $A$ selects the $H_1$ side and
$B$ the $G$ side, and their $Z_0$ types agree. Selection on the first
side implies $p\leq N^3+\alpha-\beta\leq N^3+2$. If
$a=\comp(A)\geq3$, \eqref{t:E:componentbound} gives
$p\geq N^3+N^2$, a~contradiction. Thus $a\leq2$.

The chains after the first forest are $\Ap3$ on the short side and
$P,B$ on the long side (with the appropriate input). Formula
\eqref{t:E:cross} gives
\[
 R_B+3R_A=R_BB,\qquad
 R_B+R_A(\alpha\one+\Ap3)=R_B(\beta\one+P_B),
\]
where $P_B=B^B$. Hence $B=\one+B'$ and $3R_A=R_BB'$.
Eliminating $R_A$ in the second equation and taking component counts
gives
\[
 3u=3(1-\beta)+m(\alpha+a^3)\leq3+10m.
\]
This contradicts $3u\geq18m$ from~\eqref{t:E:um}.

For $PH_1+G$, selection on the first side gives $p\leq N^3+N$;
for $H_1+PG$ it gives $p\leq N^3-N$. Again $a\leq2$. Their chains
still give $3R_A=R_BB'$ and an isolated point of $B$. For the first
sum the top equation becomes
\[
 3P_B=3\one+B'(A+\Ap3),
\]
so $3u=3+m(a+a^3)\leq3+10m$. For the second sum it becomes
\[
 3P_B+3B'=B'\Ap3,
\]
so $3u=m(a^3-3)\leq5m$. Both contradict~\eqref{t:E:um}.

Finally the summand $A$ in $A+H_1+G$ is strictly below the other two
in height. It affects neither $\Max(C)$ nor $\Der(C)$. The first
strong assertion is therefore the already proved cubic assertion,
and the signature argument still proves the second assertion.
\end{proof}

\begin{proposition}\label{t:P:quartic}
\index{quartic!comparison!twelve-occurrence audit}
The three forms with
\[
 C=\Ap\alpha U_1+\Ap\beta G,\quad
 (\alpha,\beta)=(0,0),(1,0),(0,1),\quad r=1-\alpha-\beta,
\]
are strong.
\end{proposition}
\begin{proof}
Suppose the first forest is selected for $A$ and the second for $B$.
Then $p\leq N^4+1$, so $a=\comp(A)\leq3$ by
\eqref{t:E:componentbound}. Cross-branch equality gives
$R_B+4R_A=R_BB$. Thus $B=\one+B'$, $m=\comp(B')\geq1$,
and $4R_A=R_BB'$. The other equation gives
\begin{equation}\label{t:E:quarticcount}
 4u=4(1-\beta)+m(\alpha+a^4),\qquad a\in\{1,2,3\}.
\end{equation}
If $m\geq3$, then
\[
 4u\geq4(m+1)^{m+1}\geq256m>4+82m,
\]
contradicting~\eqref{t:E:quarticcount}. We used
$(m+1)^{m-1}\geq16$ and $(m+1)^2\geq4m$.

It remains to consider $m=1,2$. If $\alpha=0$, reduction modulo four
forces $a=2$. For $(\alpha,\beta)=(0,0)$ the possible values of $u$
are $5,9$; for $(0,1)$ they are $4,8$. These contradict the lower
bounds $u\geq6$ when $m=1$ and $u\geq27$ when $m=2$.

For $(\alpha,\beta)=(1,0)$ and $m=1$, the number $1+a^4$ is not
divisible by four. For $m=2$, divisibility forces $a$ odd. The value
$a=1$ gives $u=2$, impossible; $a=3$ gives $u=42$. In that last
case $B$ has three components, one isolated. Exponentiation over
these three components writes $\comp(B^B)$ as a~product of three
integers each at least three, with one factor exactly three. The
remaining two factors would have product $14$, which is impossible
for two integers at least three. This exhausts all cases.
\end{proof}

\begin{proposition}\label{t:P:quintic}
\index{quintic!comparison}
The form with $C=W_1+G$ and $r=0$ is strong.
\end{proposition}
\begin{proof}
If $A$ selects $W_1=A^{\Ap5}$, then $p\leq N^5$ and
$a=\comp(A)\leq4$. Equality with a~$G$-branch input $B$ gives
$B=\one+B'$, $5R_A=R_BB'$, and
\[
 5u=5+ma^5,\qquad m=\comp(B')\geq1.
\]
For $m=1,2,3$, reduction modulo five is impossible: $a\in\{1,2,3,4\}$
and $a^5\equiv a\pmod5$. For $m\geq4$,
\[
 5u\geq5(m+1)^{m+1}\geq2500m>5+1024m\geq5+ma^5.
\]
Here $(m+1)^{m-1}\geq125$ and $(m+1)^2\geq4m$.
Thus the branches cannot collide.
\end{proof}

\section[Twelve: remaining pairs]{Twelve occurrences: the remaining exceptional pairs}\label{t:S:remaining}
Throughout this section $A$ denotes an input selecting the first
specified forest and $B$ an input selecting the second. To disprove
cross-branch equality, it is enough to use their $Z_0$ chains. The
second strong assertion is supplied by Lemma~\ref{t:L:signature} and
Table~\ref{t:tab:chains}, as explained above.

\begin{proposition}\label{t:P:Hoffsets}
\index{offset!four-occurrence terms}
For $(i,j)=(1,3),(2,3),(3,4)$, the three forms with
\[
 C=\Ap\alpha H_i+\Ap\beta H_j,\qquad
 (\alpha,\beta)=(0,0),(1,0),(0,1),\qquad r=1-\alpha-\beta,
\]
are strong.
\end{proposition}
\begin{proof}
In every case an input selecting the first forest is connected.
Indeed, on disconnected inputs $f>p^2$ makes $H_3$ taller than
$AH_1$ and $AH_2$, and $g\geq N^2f$ makes $H_4$ taller than
$AH_3$. These inequalities are strict even with the indicated
one-unit height offset. Thus $\comp(A)=1$ and $\iso(A)=0$.
Put $\rho=\comp(R_A)$, $\sigma=\comp(R_B)$,
$b=\comp(B)$, $u=\comp(P_B)$, and $v=\comp(F_B)$.

For $(i,j)=(1,3)$, the short chain terminates after $\Ap3$, whereas
the long chain continues through $F_B$ and $\Bp2$. Formula
\eqref{t:E:cross} gives
\[
 2\sigma+3\rho=\sigma b^2,\qquad
 2\sigma+(\alpha+1)\rho=\sigma(\beta+v).
\]
Eliminating $\rho/\sigma$ gives
\[
 3v=(\alpha+1)b^2+4-2\alpha-3\beta.
\]
For the three allowed offsets the right side is, respectively,
$b^2+4$, $2b^2+2$, or $b^2+1$. None is divisible by three, since
a~square is zero or one modulo three.

For $(i,j)=(2,3)$, the chains have equal lengths and give
$S_A=S_B$ and $R_A=2R_B$. The preceding stage gives
$R_A(\one+A)=R_B\Bp2$, independently of the offsets at the initial
stage. Thus
\[
 \Bp2=2\one+2A.
\]
Taking isolated-component counts gives $\iso(B)^2=2$, impossible.

For $(i,j)=(3,4)$, the two stages before termination give
\[
 \sigma+2\rho=b\sigma,\qquad
 \sigma+\rho=u\sigma.
\]
Hence $b\geq2$ and $2u=b+1$, contrary to $u\geq b^b$.
All nine cases are proved.
\end{proof}

\begin{proposition}\label{t:P:double}
\index{double!comparison}
The form with $C=\xp F2+G$ and $r=0$ is strong.
\end{proposition}
\begin{proof}
An input selecting $\xp F2$ satisfies $p\leq2N^2$, so it is connected
by~\eqref{t:E:disconnected}. The chains are
$\xp F2,\Ap2$ and $G,P,A$. Cross-branch equality gives
\[
 R_B+2R_A=R_BB,\qquad
 R_B+2R_A\Ap2=R_BP_B.
\]
Taking component counts and using $\comp(A)=1$ gives $b\geq2$ and
$u=b$, contradicting $u\geq b^b$.
\end{proof}

\begin{proposition}\label{t:P:unit}
\index{unit!comparison}
The form with $C=U_1+H_2$ and $r=0$ is strong.
\end{proposition}
\begin{proof}
No connectedness assumption on $A$ is needed. The chains are
$U_1,\Ap4$ and $H_2,BP_B,B$. Formula~\eqref{t:E:cross} gives
\[
 4R_A=R_BB,\qquad R_B+R_A\Ap4=R_B(BP_B).
\]
Multiply the second equality by the four-element antichain, substitute
the first equality, and cancel $R_B$. This gives
\[
 4\one+B\Ap4=4BP_B.
\]
In the signed component-polynomial ring, $\poly(B)$ therefore divides
the nonzero constant $4$. This contradicts its positive weighted
degree, by~\eqref{b:E:weighted-degree}.
\end{proof}

\begin{proposition}\label{t:P:bulk}
\index{bulk!comparison}
The nine forms with $r=0$ and the sums in Table~\ref{t:tab:bulk} are
strong.
\end{proposition}
\begin{proof}
In each row an input selecting the first forest is connected.
Indeed, on disconnected inputs the second forest is strictly taller.
In row order, the requisite inequalities are
\[
\begin{gathered}
 f>N^4,\quad g>N^4,\quad f>N^2p,\quad g>N^2p,\quad g>p^2,\\
 g>Nf,\quad Ng>f,\quad h_1>N^3,\quad h_1>Np.
\end{gathered}
\]
They follow directly from $p>N^2$, $f>p^2$, $g\geq N^2f$, and
$h_1\geq f^2$.

Put $\theta=\comp(R_A)/\comp(R_B)>0$, $b=\comp(B)$,
$u=\comp(P_B)$, $v=\comp(F_B)$, $w=\comp(G_B)$, and
$z=\comp(H_1(B))$. Apply~\eqref{t:E:cross} with the chains in
Table~\ref{t:tab:chains} and use $\comp(A)=1$. The two relevant
component equations in each row are displayed in Table~\ref{t:tab:bulk};
thus none of the eliminations requires recognizing an unknown
summand. They give the indicated necessary relations.

In the first row, $4v=b^2+6$ contradicts squares modulo four. In the
second row, $u>b$ since $\theta>0$. Also $w\geq u$: evaluation
$P_B\to B$ and the constant-map section, by precomposition, make
$P_B$ a~retract of $G_B$. Consequently $4w=u+3b$ forces $u\leq b$,
a~contradiction. The third row gives the impossible square $b^2=6$.

In the fourth row $b\geq2$. For $b=2$, $u=4$ contradicts
$u\geq b^b+b^{b-1}=6$; for $b\geq3$, $3b-2<b^2\leq b^b$.
The fifth and sixth rows give $b\geq2$ and, respectively,
$u=2b-1<b^b$ and $u=b<b^b$. The seventh gives $b\geq2$ and
$2u=b-1$, again impossible.

In the eighth row $b\geq2$ and $3z=b^3+6$. But
$H_1(B)=B^{\Bp3}$, so assigning a~target component to each source
component gives $z\geq b^{b^3}\geq b^3$. Thus
$3z\geq3b^3>b^3+6$. The last row gives the impossible cube $b^3=6$.
This proves every row.
\end{proof}

\begin{table}[htbp]
\centering\small\renewcommand{\arraystretch}{1.35}
\begin{tabular}{@{}lll@{}}
\toprule
Sum $C$ & Necessary component equations & Consequence\\
\midrule
$U_1+H_3$ & $v=\theta+2,\ b^2=4\theta+2$ & $4v=b^2+6$\\
$U_1+H_4$ & $w=b+\theta,\ u=b+4\theta$ & $4w=u+3b$\\
$U_2+H_3$ & $\theta=2,\ b^2=3\theta$ & $b^2=6$\\
$U_2+H_4$ & $\theta=b-1,\ u=1+3\theta$ & $u=3b-2$\\
$U_3+H_4$ & $\theta=b-1,\ u=1+2\theta$ & $u=2b-1$\\
$U_4+H_4$ & $b=1+2\theta,\ u=1+2\theta$ & $u=b$\\
$H_3+U_5$ & $\theta=u,\ b=1+2\theta$ & $2u=b-1$\\
$H_1+U_6$ & $z=\theta+3,\ b^3=3\theta+3$ & $3z=b^3+6$\\
$H_2+U_6$ & $\theta=3,\ b^3=2\theta$ & $b^3=6$\\
\bottomrule
\end{tabular}
\caption{Nine exceptional pairs. The first listed forest is the
selected forest for the connected input $A$.}
\label{t:tab:bulk}
\end{table}

\begin{proposition}\label{t:P:square}
\index{square!comparison}
The form with $C=U_3+H_3$ and $r=0$ is strong.
\end{proposition}
\begin{proof}
If $U_3=A^{\xp P2}$ is selected, its input $A$ is connected, since
$f>p^2$ for disconnected inputs. The two chains are
\[
 U_3,\xp P2,A\qquad\text{and}\qquad H_3,F,\Ap2.
\]
Cross-branch equality gives $S_A=S_B$, $R_A=2R_B$, and then
\begin{equation}\label{t:E:squareidentities}
 \Bp2=4A,\qquad F_B=2\xp{P_A}2.
\end{equation}
The first equality implies that $B$ has two components, say $Q_1,Q_2$.
All four product components are isomorphic to $A$. In particular,
$\xp{Q_1}2\cong\xp{Q_2}2$, so Cartesian-root uniqueness gives
$Q_1\cong Q_2=Q$. We have $B=2Q$ and $A=\xp Q2$.

Put $X=Q^A$. Then
$P_A=(\xp Q2)^A=\xp X2$. Since $A$ is connected,
$B^A=(2Q)^A=2X$. Also $\Bp2=4A$, so
\[
 F_B=B^{\Bp2}=B^{4A}=(B^A)^{\times4}=16\xp X4.
\]
The second equality in~\eqref{t:E:squareidentities} instead gives
$F_B=2\xp X4$. Their positive cardinalities contradict each other.
\end{proof}

\begin{proposition}\label{t:P:shift}
\index{shift!comparison}
The form with $C=U_6+H_4$ and $r=0$ is strong.
\end{proposition}
\begin{proof}
An input $A$ selecting $U_6=A^{H_1}$ has at most two components.
Otherwise Lemma~\ref{t:L:bounds} gives $g\geq Nh_1>h_1$, so $H_4$
would have greater height. After one reduction, a~cross-branch
equality becomes
\[
 S_A\Max(H_1(A))=S_B\Max(G(B)).
\]
The cubic chain calculation gives $B=\one+B'$,
$3R_A=R_BB'$, and $3P_B=3\one+B'\Ap3$.
Putting $a=\comp(A)\leq2$ and $m=\comp(B')\geq1$ yields
\[
 3u=3+ma^3\leq3+8m<18m,
\]
contrary to~\eqref{t:E:um}. The signatures $(3,3)$ and $(4,1)$
distinguish the second strong assertion as well.
\end{proof}

We have now proved the strong assertion for all 40 rows of
Table~\ref{t:tab:exceptions}. In the 35 remaining nonpure cases it is
important that branch recognition for the second assertion comes
from the whole invariant $\Der(C)/R$ and its successive reductions.
It does not require separating the displayed forests in $C$.

\chapter{The theorem through twelve occurrences}\label{t:S:enumeration}

\begin{theorem}\label{t:T:main}
\index{occurrence!count!twelve: theorem}
Let $t$ be an exponential term with $1\leq|t|\leq12$, and let $A,B$
be finite ordered sets. Then
\[
 t(A)\cong t(B)\quad\Longrightarrow\quad A\cong B.
\]
Every nontrivial such term is strong on $\OO$.
\end{theorem}
\begin{proof}
The one-occurrence case is the identity. Theorem~\ref{b:T:main}
proves all assertions through eleven. Consider a~twelve-occurrence
term and normalize it under interchange.

If its root is uniform, every nontrivial child has at most eleven
occurrences, so its strong assertion is available. The uniform
extension proposition proves the parent's strong assertion. This
covers 3,068 normal forms.

For a~nonuniform root use~\eqref{t:E:mixed} and the twelve-occurrence
budget in~\eqref{t:E:budget}. Proposition~\ref{t:P:classification}
exhausts all such roots. Its 1,658 regular forms are strong by
Proposition~\ref{t:P:regular}. Its forty exceptional forms are
exactly Table~\ref{t:tab:exceptions}. The proof assigned to each row
has been supplied in Chapters~\ref{t:S:powers}--\ref{t:S:remaining}.
The first strong assertion is proved there by terminal recovery
within a~branch and explicit exclusion of cross-branch collisions.
The second is proved by whole-sum polynomial recovery or by the
signature recognition lemma, without separating displayed summands.
Thus every exceptional form is strong as well.

For any of these terms, compute $V_t$ from the term value and recover
$R$ by the first strong assertion. Cancel $R$ in
$\Der(t(A))=R D_t(A)$ to recover the root exponent. The second
strong assertion then recovers $A$. This proves reconstruction on
$\OO$.

Finally, Theorem~\ref{b:T:K} proves reconstruction on $\K$ and
reflects membership in that class, so there is no unhandled pair with
one input in each class. Interchange normalization covers every binary
parenthesization. The theorem is therefore valid for every finite
input cardinality and every occurrence count from one through twelve.
\end{proof}

\section{The complete cumulative classification}
\index{classification!through twelve occurrences}\index{normal!form!cumulative classification}

\begin{table}[htbp]
\centering\small\setlength{\tabcolsep}{5pt}\renewcommand{\arraystretch}{1.15}
\begin{tabular}{@{}rrrrrr@{}}
\toprule
$n$ & Normal forms & Uniform & Regular & Exceptional & Binary\\\midrule
1&1&---&---&---&1\\

2&1&1&0&0&1\\

3&2&2&0&0&2\\

4&4&4&0&0&5\\

5&9&9&0&0&14\\

6&20&18&2&0&42\\

7&48&41&7&0&132\\

8&115&89&26&0&429\\

9&286&209&77&0&1,430\\

10&719&497&220&2&4,862\\

11&1,842&1,226&605&11&16,796\\

12&4,766&3,068&1,658&40&58,786\\

\bottomrule\end{tabular}
\caption{The complete proof assignment through twelve occurrences.}
\label{tab:cumulative}
\end{table}
\section{Executed enumeration and arithmetic checks}
\index{enumeration!executed}\index{verification!arithmetic checks}\index{audit!enumeration}
The accompanying program generates the 4,766 normal forms by weighted
multisets. Independently, it generates every full binary tree with
twelve leaves, normalizes it, and checks both the complete resulting
set and every individual multiplicity in~\eqref{E:multiplicity}.
It obtains
\[
 \sum\setm{w(T)}{|T|=12}=\frac1{12}\binom{22}{11}=58786.
\]
The proof classification gives the following disjoint totals.
\begin{table}[htbp]
\centering\renewcommand{\arraystretch}{1.2}
\begin{tabular}{@{}lrr@{}}
\toprule
Proof class & Normal forms & Binary parenthesizations\\\midrule
Uniform extension & 3,068 & 33,948\\
Fixed comparison & 1,658 & 24,624\\
Exceptional argument & 40 & 214\\\midrule
Total & 4,766 & 58,786\\\bottomrule
\end{tabular}
\caption{The exhaustive twelve-occurrence classification.}
\label{t:tab:totals}
\end{table}

The program additionally checks the forest lists, every eight-weight
reduction chain, the 38 distinct exceptional pair types, and the
inequality of their two signatures. It records the exact polynomial
certificates described in Chapter~\ref{t:S:classification}. The files
\texttt{normal\_forms.tsv}, \texttt{exceptional\_forms.tsv}, and
\texttt{root\_partitions.tsv} supply the complete finite audit.
The last table is reproduced in Appendix~\ref{app:partitions}; its 56 rows are the
integer partitions of the root weight eleven.

These checks verify finite enumeration, normalization, multiplicities,
arithmetic certificates, and assignment to the proved cases. They
are not a~formal verification of the canonical-coordinate theorem or
of the entire mathematical argument. No finite ordered-set inputs
are enumerated: the proof applies to arbitrary finite input sizes.

\section{What the compositeness of twelve supplies}
\index{occurrence!count!twelve: compositeness}\index{compositeness!twelve}\index{substitution!twelve occurrences}
For uniform substitution $s[t]$, every occurrence of $x$ in $s$ is
replaced by the same term $t$. Then $|s[t]|=|s||t|$. If $s$ and $t$
are reconstructing, so is $s[t]$: an isomorphism
$s(t(A))\cong s(t(B))$ first gives $t(A)\cong t(B)$, then $A\cong B$.
Thus the factorizations $12=2\cdot6=3\cdot4=4\cdot3=6\cdot2$
provide immediate applications of earlier results.
\begin{table}[htbp]
\centering\renewcommand{\arraystretch}{1.15}
\begin{tabular}{@{}rrrr@{}}
\toprule
$|s|$ & $|t|$ & Distinct normal forms $s[t]$ & Binary representatives\\\midrule
2&6&20&354\\3&4&8&160\\4&3&8&121\\6&2&20&728\\\midrule
\multicolumn{2}{l}{Union}&50&1,111\\\bottomrule
\end{tabular}
\caption{Supplementary proofs obtained by proper uniform substitution.
The rows overlap; the union is not their sum.}
\label{t:tab:substitution}
\end{table}
The substitution audit gives 50 distinct normal forms, comprising
16 uniform, 33 regular, and one exceptional form in our main
classification. The exceptional form is $C=H_1+U_6$, for which
Proposition~\ref{t:P:bulk} also gives a~direct proof. The file
\texttt{substitution\_witnesses.tsv} gives the outer and inner term
for every witness. Thus the multiplicative structure of twelve is
useful, but it does not replace the mixed-root analysis of the other
4,716 normal forms.

\section{What has and has not been proved}
\index{scope!theorem through twelve}\index{verification!limitations}
The theorem reconstructs the ordered-set input from the ordered-set
value of a~fixed exponential term. It places no bound on the number
of input elements. It covers all twelve-occurrence normal forms, not
only those obtained by uniform substitution, and all smaller counts
are included by proof.

The book also contains two unbounded classes: all terms on
componentwise contractible inputs, and the polynomial-perturbation
family on arbitrary finite inputs. No statement here asserts ordinary
product closure for reconstructing terms, injectivity of the pair of
operands of an arbitrary exponent, or an unrestricted induction on all
occurrence counts.

\appendix

\chapter{Complete root-partition audits}\label{app:partitions}

A~root partition records child occurrence counts. Repeated parts below
are written with multiplicity. The columns count normal forms, binary
representatives, and the uniform, regular, and exceptional proof
classes. Each table was regenerated from the executed audit.

\section{11 occurrences}\label{b:tab:partitions}

\begingroup\small\renewcommand{\arraystretch}{1.1}
\begin{longtable}{@{}lrrrrr@{}}
\toprule Root partition & Forms & Binary & Uniform & Regular & Exceptional\\\midrule
\endfirsthead
\multicolumn{6}{c}{\textit{Root partitions, continued}}\\
\toprule Root partition & Forms & Binary & Uniform & Regular & Exceptional\\\midrule\endhead
\midrule\multicolumn{6}{r}{\textit{Continued on the next page}}\\\endfoot
\endlastfoot

$10$ & 719 & 4862 & 719 & 0 & 0\\

$9+1$ & 286 & 2860 & 286 & 0 & 0\\

$8+2$ & 115 & 858 & 0 & 115 & 0\\

$8+1^{2}$ & 115 & 1287 & 115 & 0 & 0\\

$7+3$ & 96 & 528 & 0 & 95 & 1\\

$7+2+1$ & 48 & 792 & 0 & 48 & 0\\

$7+1^{3}$ & 48 & 528 & 48 & 0 & 0\\

$6+4$ & 80 & 420 & 0 & 78 & 2\\

$6+3+1$ & 40 & 504 & 0 & 39 & 1\\

$6+2^{2}$ & 20 & 126 & 0 & 20 & 0\\

$6+2+1^{2}$ & 20 & 504 & 0 & 20 & 0\\

$6+1^{4}$ & 20 & 210 & 20 & 0 & 0\\

$5^{2}$ & 45 & 196 & 9 & 30 & 6\\

$5+4+1$ & 36 & 420 & 0 & 35 & 1\\

$5+3+2$ & 18 & 168 & 0 & 18 & 0\\

$5+3+1^{2}$ & 18 & 336 & 0 & 18 & 0\\

$5+2^{2}+1$ & 9 & 168 & 0 & 9 & 0\\

$5+2+1^{3}$ & 9 & 280 & 0 & 9 & 0\\

$5+1^{5}$ & 9 & 84 & 9 & 0 & 0\\

$4^{2}+2$ & 10 & 75 & 0 & 10 & 0\\

$4^{2}+1^{2}$ & 10 & 150 & 4 & 6 & 0\\

$4+3^{2}$ & 12 & 60 & 0 & 12 & 0\\

$4+3+2+1$ & 8 & 240 & 0 & 8 & 0\\

$4+3+1^{3}$ & 8 & 200 & 0 & 8 & 0\\

$4+2^{3}$ & 4 & 20 & 0 & 4 & 0\\

$4+2^{2}+1^{2}$ & 4 & 150 & 0 & 4 & 0\\

$4+2+1^{4}$ & 4 & 150 & 0 & 4 & 0\\

$4+1^{6}$ & 4 & 35 & 4 & 0 & 0\\

$3^{3}+1$ & 4 & 32 & 2 & 2 & 0\\

$3^{2}+2^{2}$ & 3 & 24 & 0 & 3 & 0\\

$3^{2}+2+1^{2}$ & 3 & 120 & 0 & 3 & 0\\

$3^{2}+1^{4}$ & 3 & 60 & 2 & 1 & 0\\

$3+2^{3}+1$ & 2 & 40 & 0 & 2 & 0\\

$3+2^{2}+1^{3}$ & 2 & 120 & 0 & 2 & 0\\

$3+2+1^{5}$ & 2 & 84 & 0 & 2 & 0\\

$3+1^{7}$ & 2 & 16 & 2 & 0 & 0\\

$2^{5}$ & 1 & 1 & 1 & 0 & 0\\

$2^{4}+1^{2}$ & 1 & 15 & 1 & 0 & 0\\

$2^{3}+1^{4}$ & 1 & 35 & 1 & 0 & 0\\

$2^{2}+1^{6}$ & 1 & 28 & 1 & 0 & 0\\

$2+1^{8}$ & 1 & 9 & 1 & 0 & 0\\*

$1^{10}$ & 1 & 1 & 1 & 0 & 0\\*

\midrule Total & 1842 & 16796 & 1226 & 605 & 11\\\bottomrule

\caption{All 42 root partitions at 11 occurrences.}\end{longtable}\endgroup

\section{12 occurrences}\label{t:tab:partitions}

\begingroup\small\renewcommand{\arraystretch}{1.1}
\begin{longtable}{@{}lrrrrr@{}}
\toprule Root partition & Forms & Binary & Uniform & Regular & Exceptional\\\midrule
\endfirsthead
\multicolumn{6}{c}{\textit{Root partitions, continued}}\\
\toprule Root partition & Forms & Binary & Uniform & Regular & Exceptional\\\midrule\endhead
\midrule\multicolumn{6}{r}{\textit{Continued on the next page}}\\\endfoot
\endlastfoot

$11$ & 1842 & 16796 & 1842 & 0 & 0\\

$10+1$ & 719 & 9724 & 719 & 0 & 0\\

$9+2$ & 286 & 2860 & 0 & 286 & 0\\

$9+1^{2}$ & 286 & 4290 & 286 & 0 & 0\\

$8+3$ & 230 & 1716 & 0 & 229 & 1\\

$8+2+1$ & 115 & 2574 & 0 & 115 & 0\\

$8+1^{3}$ & 115 & 1716 & 115 & 0 & 0\\

$7+4$ & 192 & 1320 & 0 & 186 & 6\\

$7+3+1$ & 96 & 1584 & 0 & 95 & 1\\

$7+2^{2}$ & 48 & 396 & 0 & 48 & 0\\

$7+2+1^{2}$ & 48 & 1584 & 0 & 48 & 0\\

$7+1^{4}$ & 48 & 660 & 48 & 0 & 0\\

$6+5$ & 180 & 1176 & 0 & 158 & 22\\

$6+4+1$ & 80 & 1260 & 0 & 78 & 2\\

$6+3+2$ & 40 & 504 & 0 & 39 & 1\\

$6+3+1^{2}$ & 40 & 1008 & 0 & 39 & 1\\

$6+2^{2}+1$ & 20 & 504 & 0 & 20 & 0\\

$6+2+1^{3}$ & 20 & 840 & 0 & 20 & 0\\

$6+1^{5}$ & 20 & 252 & 20 & 0 & 0\\

$5^{2}+1$ & 45 & 588 & 9 & 32 & 4\\

$5+4+2$ & 36 & 420 & 0 & 35 & 1\\

$5+4+1^{2}$ & 36 & 840 & 0 & 35 & 1\\

$5+3^{2}$ & 27 & 168 & 0 & 27 & 0\\

$5+3+2+1$ & 18 & 672 & 0 & 18 & 0\\

$5+3+1^{3}$ & 18 & 560 & 0 & 18 & 0\\

$5+2^{3}$ & 9 & 56 & 0 & 9 & 0\\

$5+2^{2}+1^{2}$ & 9 & 420 & 0 & 9 & 0\\

$5+2+1^{4}$ & 9 & 420 & 0 & 9 & 0\\

$5+1^{6}$ & 9 & 98 & 9 & 0 & 0\\

$4^{2}+3$ & 20 & 150 & 0 & 20 & 0\\

$4^{2}+2+1$ & 10 & 300 & 0 & 10 & 0\\

$4^{2}+1^{3}$ & 10 & 250 & 4 & 6 & 0\\

$4+3^{2}+1$ & 12 & 240 & 0 & 12 & 0\\

$4+3+2^{2}$ & 8 & 120 & 0 & 8 & 0\\

$4+3+2+1^{2}$ & 8 & 600 & 0 & 8 & 0\\

$4+3+1^{4}$ & 8 & 300 & 0 & 8 & 0\\

$4+2^{3}+1$ & 4 & 100 & 0 & 4 & 0\\

$4+2^{2}+1^{3}$ & 4 & 300 & 0 & 4 & 0\\

$4+2+1^{5}$ & 4 & 210 & 0 & 4 & 0\\

$4+1^{7}$ & 4 & 40 & 4 & 0 & 0\\

$3^{3}+2$ & 4 & 32 & 0 & 4 & 0\\

$3^{3}+1^{2}$ & 4 & 80 & 2 & 2 & 0\\

$3^{2}+2^{2}+1$ & 3 & 120 & 0 & 3 & 0\\

$3^{2}+2+1^{3}$ & 3 & 240 & 0 & 3 & 0\\

$3^{2}+1^{5}$ & 3 & 84 & 2 & 1 & 0\\

$3+2^{4}$ & 2 & 10 & 0 & 2 & 0\\

$3+2^{3}+1^{2}$ & 2 & 120 & 0 & 2 & 0\\

$3+2^{2}+1^{4}$ & 2 & 210 & 0 & 2 & 0\\

$3+2+1^{6}$ & 2 & 112 & 0 & 2 & 0\\

$3+1^{8}$ & 2 & 18 & 2 & 0 & 0\\

$2^{5}+1$ & 1 & 6 & 1 & 0 & 0\\

$2^{4}+1^{3}$ & 1 & 35 & 1 & 0 & 0\\

$2^{3}+1^{5}$ & 1 & 56 & 1 & 0 & 0\\

$2^{2}+1^{7}$ & 1 & 36 & 1 & 0 & 0\\

$2+1^{9}$ & 1 & 10 & 1 & 0 & 0\\*

$1^{11}$ & 1 & 1 & 1 & 0 & 0\\*

\midrule Total & 4766 & 58786 & 3068 & 1658 & 40\\\bottomrule

\caption{All 56 root partitions at 12 occurrences.}\end{longtable}\endgroup

\chapter{Exact height-comparison certificates}\label{app:certificates}

This appendix prints every one of the 457 nonnegative polynomial
certificates used for the four twelve-occurrence comparison matrices.
They certify inequalities for all permitted inputs, not merely for a
set of numerical samples. The five certificates with zero constant
term require the tie analysis in Chapter~\ref{t:S:classification}.
\index{arithmetic!certificate}\index{verification!exact polynomial}

\section{How to read a~certificate}
\index{arithmetic!certificate!reading}\index{height!comparison!reading certificates}
Subtract the losing height coefficient from the winning one. Replace
variables in reverse dependency order by a~chosen lower bound plus an
independent nonnegative slack. Write $s_N,s_p,s_f,s_g,s_{h_i},s_{u_i}$
for these slacks. In the tables a~substitution code such as $g_2$ means
use the second bound for $g$ in the following list; it is not a~power
or a~cardinality variable. Every variable that appears in the original
difference is eliminated. The final polynomial has the displayed
nonnegative coefficients. A~positive constant term gives a~strict
comparison. The winner column indicates whether the row or column
forest is selected.

\begin{longtable}{@{}ll@{}}
\toprule Code & Substitution\\\midrule
$N_1$ & $N=4+s_N$\\
$p_1$ & $p=N^2+N+1+s_p$\\
$f_1$ & $f=p+N+s_f$\\
$g_1$ & $g=p+N+s_g$\\
$g_2$ & $g=Np+1+s_g$\\
$(h_1)_1$ & $h_1=f+N-1+s_{h_1}$\\
$(h_2)_1$ & $h_2=p^2+p+s_{h_2}$\\
$(h_2)_2$ & $h_2=f+1+s_{h_2}$\\
$(h_2)_3$ & $h_2=g+1+s_{h_2}$\\
$(h_3)_1$ & $h_3=f+N-1+s_{h_3}$\\
$(h_3)_2$ & $h_3=Ng+1+s_{h_3}$\\
$(h_4)_1$ & $h_4=g+N-1+s_{h_4}$\\
$(h_4)_2$ & $h_4=Nf+1+s_{h_4}$\\*
$(u_i)_1$ & $u_i=p+1+s_{u_i}$, $1\leq i\leq9$\\*
\bottomrule
\caption{The substitutions justified by~\eqref{t:E:certificatebounds}.}
\end{longtable}

\Needspace{13\baselineskip}\section{The weight-4+3 matrix}
\index{height!comparison!weight $4+3$}\index{certificate!matrix!weight $4+3$}

\begingroup\footnotesize\renewcommand{\arraystretch}{1.32}\setlength{\tabcolsep}{3pt}
\begin{longtable}{@{}rrcp{.21\textwidth}p{.54\textwidth}@{}}
\toprule Row & Col. & Winner & Substitutions & Nonnegative difference polynomial\\\midrule
\endfirsthead
\multicolumn{5}{c}{\textit{Exact comparison certificates, continued}}\\
\toprule Row & Col. & Winner & Substitutions & Nonnegative difference polynomial\\\midrule\endhead
\midrule\multicolumn{5}{r}{\textit{Continued on the next page}}\\\endfoot
\endlastfoot

1 & 1 & R & $\text{none}$ & $1$\\

1 & 2 & C & $N_{1}$ & $1+s_{N}$\\

1 & 3 & C & $N_{1}$ & $12+8s_{N}+s_{N}^{2}$\\

1 & 4 & C & $p_{1},N_{1}$ & $17+s_{p}+9s_{N}+s_{N}^{2}$\\

2 & 1 & R & $N_{1}$ & $3+s_{N}$\\

2 & 2 & R & $\text{none}$ & $1$\\

2 & 3 & C & $N_{1}$ & $10+7s_{N}+s_{N}^{2}$\\

2 & 4 & C & $p_{1},N_{1}$ & $15+s_{p}+8s_{N}+s_{N}^{2}$\\

3 & 1 & R & $N_{1}$ & $5+2s_{N}$\\

3 & 2 & R & $N_{1}$ & $3+s_{N}$\\

3 & 3 & C & $N_{1}$ & $8+6s_{N}+s_{N}^{2}$\\

3 & 4 & C & $p_{1},N_{1}$ & $13+s_{p}+7s_{N}+s_{N}^{2}$\\

4 & 1 & R & $N_{1}$ & $14+8s_{N}+s_{N}^{2}$\\

4 & 2 & R & $N_{1}$ & $12+7s_{N}+s_{N}^{2}$\\

4 & 3 & R & $\text{none}$ & $1$\\

4 & 4 & C & $p_{1},N_{1}$ & $4+s_{p}+s_{N}$\\

5 & 1 & R & $p_{1},N_{1}$ & $19+s_{p}+9s_{N}+s_{N}^{2}$\\

5 & 2 & R & $p_{1},N_{1}$ & $17+s_{p}+8s_{N}+s_{N}^{2}$\\

5 & 3 & R & $p_{1},N_{1}$ & $6+s_{p}+s_{N}$\\

5 & 4 & R & $\text{none}$ & $1$\\

6 & 1 & R & $N_{1}$ & $61+48s_{N}+12s_{N}^{2}+s_{N}^{3}$\\

6 & 2 & R & $N_{1}$ & $59+47s_{N}+12s_{N}^{2}+s_{N}^{3}$\\

6 & 3 & R & $N_{1}$ & $48+40s_{N}+11s_{N}^{2}+s_{N}^{3}$\\

7 & 1 & R & $p_{1},N_{1}$ & $81+4s_{p}+57s_{N}+s_{N}\,s_{p}+13s_{N}^{2}+s_{N}^{3}$\\

7 & 2 & R & $p_{1},N_{1}$ & $79+4s_{p}+56s_{N}+s_{N}\,s_{p}+13s_{N}^{2}+s_{N}^{3}$\\

7 & 3 & R & $p_{1},N_{1}$ & $68+4s_{p}+49s_{N}+s_{N}\,s_{p}+12s_{N}^{2}+s_{N}^{3}$\\

7 & 4 & R & $p_{1},N_{1}$ & $63+3s_{p}+48s_{N}+s_{N}\,s_{p}+12s_{N}^{2}+s_{N}^{3}$\\

8 & 1 & R & $f_{1},p_{1},N_{1}$ & $22+s_{f}+s_{p}+10s_{N}+s_{N}^{2}$\\

8 & 2 & R & $f_{1},p_{1},N_{1}$ & $20+s_{f}+s_{p}+9s_{N}+s_{N}^{2}$\\

8 & 3 & R & $f_{1},p_{1},N_{1}$ & $9+s_{f}+s_{p}+2s_{N}$\\

8 & 4 & R & $f_{1},N_{1}$ & $4+s_{f}+s_{N}$\\

9 & 1 & R & $g_{1},p_{1},N_{1}$ & $22+s_{g}+s_{p}+10s_{N}+s_{N}^{2}$\\

9 & 2 & R & $g_{1},p_{1},N_{1}$ & $20+s_{g}+s_{p}+9s_{N}+s_{N}^{2}$\\

9 & 3 & R & $g_{1},p_{1},N_{1}$ & $9+s_{g}+s_{p}+2s_{N}$\\*

9 & 4 & R & $g_{1},N_{1}$ & $4+s_{g}+s_{N}$\\*

\bottomrule\caption{All 35 certified fixed comparisons in the weight-4+3 matrix.}\end{longtable}\endgroup

\Needspace{13\baselineskip}\section{The weight-5+3 matrix}
\index{height!comparison!weight $5+3$}\index{certificate!matrix!weight $5+3$}

\begingroup\footnotesize\renewcommand{\arraystretch}{1.32}\setlength{\tabcolsep}{3pt}
\begin{longtable}{@{}rrcp{.21\textwidth}p{.54\textwidth}@{}}
\toprule Row & Col. & Winner & Substitutions & Nonnegative difference polynomial\\\midrule
\endfirsthead
\multicolumn{5}{c}{\textit{Exact comparison certificates, continued}}\\
\toprule Row & Col. & Winner & Substitutions & Nonnegative difference polynomial\\\midrule\endhead
\midrule\multicolumn{5}{r}{\textit{Continued on the next page}}\\\endfoot
\endlastfoot

1 & 1 & R & $\text{none}$ & $2$\\

1 & 2 & C & $N_{1}$ & $s_{N}$\\

1 & 3 & C & $N_{1}$ & $11+8s_{N}+s_{N}^{2}$\\

1 & 4 & C & $p_{1},N_{1}$ & $16+s_{p}+9s_{N}+s_{N}^{2}$\\

2 & 1 & R & $N_{1}$ & $4+s_{N}$\\

2 & 2 & R & $\text{none}$ & $2$\\

2 & 3 & C & $N_{1}$ & $9+7s_{N}+s_{N}^{2}$\\

2 & 4 & C & $p_{1},N_{1}$ & $14+s_{p}+8s_{N}+s_{N}^{2}$\\

3 & 1 & R & $N_{1}$ & $6+2s_{N}$\\

3 & 2 & R & $N_{1}$ & $4+s_{N}$\\

3 & 3 & C & $N_{1}$ & $7+6s_{N}+s_{N}^{2}$\\

3 & 4 & C & $p_{1},N_{1}$ & $12+s_{p}+7s_{N}+s_{N}^{2}$\\

4 & 1 & R & $N_{1}$ & $15+8s_{N}+s_{N}^{2}$\\

4 & 2 & R & $N_{1}$ & $13+7s_{N}+s_{N}^{2}$\\

4 & 3 & R & $\text{none}$ & $2$\\

4 & 4 & C & $p_{1},N_{1}$ & $3+s_{p}+s_{N}$\\

5 & 1 & R & $p_{1},N_{1}$ & $20+s_{p}+9s_{N}+s_{N}^{2}$\\

5 & 2 & R & $p_{1},N_{1}$ & $18+s_{p}+8s_{N}+s_{N}^{2}$\\

5 & 3 & R & $p_{1},N_{1}$ & $7+s_{p}+s_{N}$\\

5 & 4 & R & $\text{none}$ & $2$\\

6 & 1 & R & $N_{1}$ & $17+9s_{N}+s_{N}^{2}$\\

6 & 2 & R & $N_{1}$ & $15+8s_{N}+s_{N}^{2}$\\

6 & 3 & R & $N_{1}$ & $4+s_{N}$\\

6 & 4 & C & $p_{1}$ & $1+s_{p}$\\

7 & 1 & R & $p_{1},N_{1}$ & $22+s_{p}+10s_{N}+s_{N}^{2}$\\

7 & 2 & R & $p_{1},N_{1}$ & $20+s_{p}+9s_{N}+s_{N}^{2}$\\

7 & 3 & R & $p_{1},N_{1}$ & $9+s_{p}+2s_{N}$\\

7 & 4 & R & $N_{1}$ & $4+s_{N}$\\

8 & 1 & R & $N_{1}$ & $62+48s_{N}+12s_{N}^{2}+s_{N}^{3}$\\

8 & 2 & R & $N_{1}$ & $60+47s_{N}+12s_{N}^{2}+s_{N}^{3}$\\

8 & 3 & R & $N_{1}$ & $49+40s_{N}+11s_{N}^{2}+s_{N}^{3}$\\

9 & 1 & R & $p_{1},N_{1}$ & $82+4s_{p}+57s_{N}+s_{N}\,s_{p}+13s_{N}^{2}+s_{N}^{3}$\\

9 & 2 & R & $p_{1},N_{1}$ & $80+4s_{p}+56s_{N}+s_{N}\,s_{p}+13s_{N}^{2}+s_{N}^{3}$\\

9 & 3 & R & $p_{1},N_{1}$ & $69+4s_{p}+49s_{N}+s_{N}\,s_{p}+12s_{N}^{2}+s_{N}^{3}$\\

9 & 4 & R & $p_{1},N_{1}$ & $64+3s_{p}+48s_{N}+s_{N}\,s_{p}+12s_{N}^{2}+s_{N}^{3}$\\

10 & 1 & R & $f_{1},p_{1},N_{1}$ & $23+s_{f}+s_{p}+10s_{N}+s_{N}^{2}$\\

10 & 2 & R & $f_{1},p_{1},N_{1}$ & $21+s_{f}+s_{p}+9s_{N}+s_{N}^{2}$\\

10 & 3 & R & $f_{1},p_{1},N_{1}$ & $10+s_{f}+s_{p}+2s_{N}$\\

10 & 4 & R & $f_{1},N_{1}$ & $5+s_{f}+s_{N}$\\

11 & 1 & R & $g_{1},p_{1},N_{1}$ & $23+s_{g}+s_{p}+10s_{N}+s_{N}^{2}$\\

11 & 2 & R & $g_{1},p_{1},N_{1}$ & $21+s_{g}+s_{p}+9s_{N}+s_{N}^{2}$\\

11 & 3 & R & $g_{1},p_{1},N_{1}$ & $10+s_{g}+s_{p}+2s_{N}$\\

11 & 4 & R & $g_{1},N_{1}$ & $5+s_{g}+s_{N}$\\

12 & 1 & R & $N_{1}$ & $253+256s_{N}+96s_{N}^{2}+16s_{N}^{3}+s_{N}^{4}$\\

12 & 2 & R & $N_{1}$ & $251+255s_{N}+96s_{N}^{2}+16s_{N}^{3}+s_{N}^{4}$\\

12 & 3 & R & $N_{1}$ & $240+248s_{N}+95s_{N}^{2}+16s_{N}^{3}+s_{N}^{4}$\\

13 & 1 & R & $p_{1},N_{1}$ & $333+16s_{p}+312s_{N}+8s_{N}\,s_{p}+109s_{N}^{2}+s_{N}^{2}\,s_{p}+17s_{N}^{3}+s_{N}^{4}$\\

13 & 2 & R & $p_{1},N_{1}$ & $331+16s_{p}+311s_{N}+8s_{N}\,s_{p}+109s_{N}^{2}+s_{N}^{2}\,s_{p}+17s_{N}^{3}+s_{N}^{4}$\\

13 & 3 & R & $p_{1},N_{1}$ & $320+16s_{p}+304s_{N}+8s_{N}\,s_{p}+108s_{N}^{2}+s_{N}^{2}\,s_{p}+17s_{N}^{3}+s_{N}^{4}$\\

13 & 4 & R & $p_{1},N_{1}$ & $315+15s_{p}+303s_{N}+8s_{N}\,s_{p}+108s_{N}^{2}+s_{N}^{2}\,s_{p}+17s_{N}^{3}+s_{N}^{4}$\\

14 & 1 & R & $p_{1},N_{1}$ & $438+42s_{p}+378s_{N}+s_{p}^{2}+18s_{N}\,s_{p}+123s_{N}^{2}+2s_{N}^{2}\,s_{p}+18s_{N}^{3}+s_{N}^{4}$\\

14 & 2 & R & $p_{1},N_{1}$ & $436+42s_{p}+377s_{N}+s_{p}^{2}+18s_{N}\,s_{p}+123s_{N}^{2}+2s_{N}^{2}\,s_{p}+18s_{N}^{3}+s_{N}^{4}$\\

14 & 3 & R & $p_{1},N_{1}$ & $425+42s_{p}+370s_{N}+s_{p}^{2}+18s_{N}\,s_{p}+122s_{N}^{2}+2s_{N}^{2}\,s_{p}+18s_{N}^{3}+s_{N}^{4}$\\

14 & 4 & R & $p_{1},N_{1}$ & $420+41s_{p}+369s_{N}+s_{p}^{2}+18s_{N}\,s_{p}+122s_{N}^{2}+2s_{N}^{2}\,s_{p}+18s_{N}^{3}+s_{N}^{4}$\\

15 & 1 & R & $f_{1},p_{1},N_{1}$ & $97+4s_{f}+4s_{p}+65s_{N}+s_{N}\,s_{f}+s_{N}\,s_{p}+14s_{N}^{2}+s_{N}^{3}$\\

15 & 2 & R & $f_{1},p_{1},N_{1}$ & $95+4s_{f}+4s_{p}+64s_{N}+s_{N}\,s_{f}+s_{N}\,s_{p}+14s_{N}^{2}+s_{N}^{3}$\\

15 & 3 & R & $f_{1},p_{1},N_{1}$ & $84+4s_{f}+4s_{p}+57s_{N}+s_{N}\,s_{f}+s_{N}\,s_{p}+13s_{N}^{2}+s_{N}^{3}$\\

15 & 4 & R & $f_{1},p_{1},N_{1}$ & $79+4s_{f}+3s_{p}+56s_{N}+s_{N}\,s_{f}+s_{N}\,s_{p}+13s_{N}^{2}+s_{N}^{3}$\\

16 & 1 & R & $g_{1},p_{1},N_{1}$ & $97+4s_{g}+4s_{p}+65s_{N}+s_{N}\,s_{g}+s_{N}\,s_{p}+14s_{N}^{2}+s_{N}^{3}$\\

16 & 2 & R & $g_{1},p_{1},N_{1}$ & $95+4s_{g}+4s_{p}+64s_{N}+s_{N}\,s_{g}+s_{N}\,s_{p}+14s_{N}^{2}+s_{N}^{3}$\\

16 & 3 & R & $g_{1},p_{1},N_{1}$ & $84+4s_{g}+4s_{p}+57s_{N}+s_{N}\,s_{g}+s_{N}\,s_{p}+13s_{N}^{2}+s_{N}^{3}$\\

16 & 4 & R & $g_{1},p_{1},N_{1}$ & $79+4s_{g}+3s_{p}+56s_{N}+s_{N}\,s_{g}+s_{N}\,s_{p}+13s_{N}^{2}+s_{N}^{3}$\\

17 & 1 & R & $(h_{1})_{1},f_{1},p_{1},N_{1}$ & $25+s_{h_{1}}+s_{f}+s_{p}+11s_{N}+s_{N}^{2}$\\

17 & 2 & R & $(h_{1})_{1},f_{1},p_{1},N_{1}$ & $23+s_{h_{1}}+s_{f}+s_{p}+10s_{N}+s_{N}^{2}$\\

17 & 3 & R & $(h_{1})_{1},f_{1},p_{1},N_{1}$ & $12+s_{h_{1}}+s_{f}+s_{p}+3s_{N}$\\

17 & 4 & R & $(h_{1})_{1},f_{1},N_{1}$ & $7+s_{h_{1}}+s_{f}+2s_{N}$\\

18 & 1 & R & $(h_{2})_{1},p_{1},N_{1}$ & $459+s_{h_{2}}+43s_{p}+387s_{N}+s_{p}^{2}+18s_{N}\,s_{p}+124s_{N}^{2}+2s_{N}^{2}\,s_{p}+18s_{N}^{3}+s_{N}^{4}$\\

18 & 2 & R & $(h_{2})_{1},p_{1},N_{1}$ & $457+s_{h_{2}}+43s_{p}+386s_{N}+s_{p}^{2}+18s_{N}\,s_{p}+124s_{N}^{2}+2s_{N}^{2}\,s_{p}+18s_{N}^{3}+s_{N}^{4}$\\

18 & 3 & R & $(h_{2})_{1},p_{1},N_{1}$ & $446+s_{h_{2}}+43s_{p}+379s_{N}+s_{p}^{2}+18s_{N}\,s_{p}+123s_{N}^{2}+2s_{N}^{2}\,s_{p}+18s_{N}^{3}+s_{N}^{4}$\\

18 & 4 & R & $(h_{2})_{1},p_{1},N_{1}$ & $441+s_{h_{2}}+42s_{p}+378s_{N}+s_{p}^{2}+18s_{N}\,s_{p}+123s_{N}^{2}+2s_{N}^{2}\,s_{p}+18s_{N}^{3}+s_{N}^{4}$\\

19 & 1 & R & $(h_{3})_{1},f_{1},p_{1},N_{1}$ & $25+s_{h_{3}}+s_{f}+s_{p}+11s_{N}+s_{N}^{2}$\\

19 & 2 & R & $(h_{3})_{1},f_{1},p_{1},N_{1}$ & $23+s_{h_{3}}+s_{f}+s_{p}+10s_{N}+s_{N}^{2}$\\

19 & 3 & R & $(h_{3})_{1},f_{1},p_{1},N_{1}$ & $12+s_{h_{3}}+s_{f}+s_{p}+3s_{N}$\\

19 & 4 & R & $(h_{3})_{1},f_{1},N_{1}$ & $7+s_{h_{3}}+s_{f}+2s_{N}$\\

20 & 1 & R & $(h_{4})_{1},g_{1},p_{1},N_{1}$ & $25+s_{h_{4}}+s_{g}+s_{p}+11s_{N}+s_{N}^{2}$\\

20 & 2 & R & $(h_{4})_{1},g_{1},p_{1},N_{1}$ & $23+s_{h_{4}}+s_{g}+s_{p}+10s_{N}+s_{N}^{2}$\\

20 & 3 & R & $(h_{4})_{1},g_{1},p_{1},N_{1}$ & $12+s_{h_{4}}+s_{g}+s_{p}+3s_{N}$\\*

20 & 4 & R & $(h_{4})_{1},g_{1},N_{1}$ & $7+s_{h_{4}}+s_{g}+2s_{N}$\\*

\bottomrule\caption{All 78 certified fixed comparisons in the weight-5+3 matrix.}\end{longtable}\endgroup

\Needspace{13\baselineskip}\section{The weight-6+3 matrix}
\index{height!comparison!weight $6+3$}\index{certificate!matrix!weight $6+3$}

\begingroup\footnotesize\renewcommand{\arraystretch}{1.32}\setlength{\tabcolsep}{3pt}
\begin{longtable}{@{}rrcp{.21\textwidth}p{.54\textwidth}@{}}
\toprule Row & Col. & Winner & Substitutions & Nonnegative difference polynomial\\\midrule
\endfirsthead
\multicolumn{5}{c}{\textit{Exact comparison certificates, continued}}\\
\toprule Row & Col. & Winner & Substitutions & Nonnegative difference polynomial\\\midrule\endhead
\midrule\multicolumn{5}{r}{\textit{Continued on the next page}}\\\endfoot
\endlastfoot

1 & 1 & R & $\text{none}$ & $3$\\

1 & 3 & C & $N_{1}$ & $10+8s_{N}+s_{N}^{2}$\\

1 & 4 & C & $p_{1},N_{1}$ & $15+s_{p}+9s_{N}+s_{N}^{2}$\\

2 & 1 & R & $N_{1}$ & $5+s_{N}$\\

2 & 2 & R & $\text{none}$ & $3$\\

2 & 3 & C & $N_{1}$ & $8+7s_{N}+s_{N}^{2}$\\

2 & 4 & C & $p_{1},N_{1}$ & $13+s_{p}+8s_{N}+s_{N}^{2}$\\

3 & 1 & R & $N_{1}$ & $7+2s_{N}$\\

3 & 2 & R & $N_{1}$ & $5+s_{N}$\\

3 & 3 & C & $N_{1}$ & $6+6s_{N}+s_{N}^{2}$\\

3 & 4 & C & $p_{1},N_{1}$ & $11+s_{p}+7s_{N}+s_{N}^{2}$\\

4 & 1 & R & $N_{1}$ & $9+3s_{N}$\\

4 & 2 & R & $N_{1}$ & $7+2s_{N}$\\

4 & 3 & C & $N_{1}$ & $4+5s_{N}+s_{N}^{2}$\\

4 & 4 & C & $p_{1},N_{1}$ & $9+s_{p}+6s_{N}+s_{N}^{2}$\\

5 & 1 & R & $N_{1}$ & $16+8s_{N}+s_{N}^{2}$\\

5 & 2 & R & $N_{1}$ & $14+7s_{N}+s_{N}^{2}$\\

5 & 3 & R & $\text{none}$ & $3$\\

5 & 4 & C & $p_{1},N_{1}$ & $2+s_{p}+s_{N}$\\

6 & 1 & R & $p_{1},N_{1}$ & $21+s_{p}+9s_{N}+s_{N}^{2}$\\

6 & 2 & R & $p_{1},N_{1}$ & $19+s_{p}+8s_{N}+s_{N}^{2}$\\

6 & 3 & R & $p_{1},N_{1}$ & $8+s_{p}+s_{N}$\\

6 & 4 & R & $\text{none}$ & $3$\\

7 & 1 & R & $N_{1}$ & $18+9s_{N}+s_{N}^{2}$\\

7 & 2 & R & $N_{1}$ & $16+8s_{N}+s_{N}^{2}$\\

7 & 3 & R & $N_{1}$ & $5+s_{N}$\\

7 & 4 & C & $p_{1}$ & $s_{p}$\\

8 & 1 & R & $p_{1},N_{1}$ & $23+s_{p}+10s_{N}+s_{N}^{2}$\\

8 & 2 & R & $p_{1},N_{1}$ & $21+s_{p}+9s_{N}+s_{N}^{2}$\\

8 & 3 & R & $p_{1},N_{1}$ & $10+s_{p}+2s_{N}$\\

8 & 4 & R & $N_{1}$ & $5+s_{N}$\\

9 & 1 & R & $N_{1}$ & $29+16s_{N}+2s_{N}^{2}$\\

9 & 2 & R & $N_{1}$ & $27+15s_{N}+2s_{N}^{2}$\\

9 & 3 & R & $N_{1}$ & $16+8s_{N}+s_{N}^{2}$\\

10 & 1 & R & $p_{1},N_{1}$ & $34+s_{p}+17s_{N}+2s_{N}^{2}$\\

10 & 2 & R & $p_{1},N_{1}$ & $32+s_{p}+16s_{N}+2s_{N}^{2}$\\

10 & 3 & R & $p_{1},N_{1}$ & $21+s_{p}+9s_{N}+s_{N}^{2}$\\

10 & 4 & R & $N_{1}$ & $16+8s_{N}+s_{N}^{2}$\\

11 & 1 & R & $p_{1},N_{1}$ & $39+2s_{p}+18s_{N}+2s_{N}^{2}$\\

11 & 2 & R & $p_{1},N_{1}$ & $37+2s_{p}+17s_{N}+2s_{N}^{2}$\\

11 & 3 & R & $p_{1},N_{1}$ & $26+2s_{p}+10s_{N}+s_{N}^{2}$\\

11 & 4 & R & $p_{1},N_{1}$ & $21+s_{p}+9s_{N}+s_{N}^{2}$\\

12 & 1 & R & $N_{1}$ & $63+48s_{N}+12s_{N}^{2}+s_{N}^{3}$\\

12 & 2 & R & $N_{1}$ & $61+47s_{N}+12s_{N}^{2}+s_{N}^{3}$\\

12 & 3 & R & $N_{1}$ & $50+40s_{N}+11s_{N}^{2}+s_{N}^{3}$\\

13 & 1 & R & $p_{1},N_{1}$ & $83+4s_{p}+57s_{N}+s_{N}\,s_{p}+13s_{N}^{2}+s_{N}^{3}$\\

13 & 2 & R & $p_{1},N_{1}$ & $81+4s_{p}+56s_{N}+s_{N}\,s_{p}+13s_{N}^{2}+s_{N}^{3}$\\

13 & 3 & R & $p_{1},N_{1}$ & $70+4s_{p}+49s_{N}+s_{N}\,s_{p}+12s_{N}^{2}+s_{N}^{3}$\\

13 & 4 & R & $p_{1},N_{1}$ & $65+3s_{p}+48s_{N}+s_{N}\,s_{p}+12s_{N}^{2}+s_{N}^{3}$\\

14 & 1 & R & $f_{1},p_{1},N_{1}$ & $24+s_{f}+s_{p}+10s_{N}+s_{N}^{2}$\\

14 & 2 & R & $f_{1},p_{1},N_{1}$ & $22+s_{f}+s_{p}+9s_{N}+s_{N}^{2}$\\

14 & 3 & R & $f_{1},p_{1},N_{1}$ & $11+s_{f}+s_{p}+2s_{N}$\\

14 & 4 & R & $f_{1},N_{1}$ & $6+s_{f}+s_{N}$\\

15 & 1 & R & $g_{1},p_{1},N_{1}$ & $24+s_{g}+s_{p}+10s_{N}+s_{N}^{2}$\\

15 & 2 & R & $g_{1},p_{1},N_{1}$ & $22+s_{g}+s_{p}+9s_{N}+s_{N}^{2}$\\

15 & 3 & R & $g_{1},p_{1},N_{1}$ & $11+s_{g}+s_{p}+2s_{N}$\\

15 & 4 & R & $g_{1},N_{1}$ & $6+s_{g}+s_{N}$\\

16 & 1 & R & $N_{1}$ & $65+49s_{N}+12s_{N}^{2}+s_{N}^{3}$\\

16 & 2 & R & $N_{1}$ & $63+48s_{N}+12s_{N}^{2}+s_{N}^{3}$\\

16 & 3 & R & $N_{1}$ & $52+41s_{N}+11s_{N}^{2}+s_{N}^{3}$\\

17 & 1 & R & $p_{1},N_{1}$ & $85+4s_{p}+58s_{N}+s_{N}\,s_{p}+13s_{N}^{2}+s_{N}^{3}$\\

17 & 2 & R & $p_{1},N_{1}$ & $83+4s_{p}+57s_{N}+s_{N}\,s_{p}+13s_{N}^{2}+s_{N}^{3}$\\

17 & 3 & R & $p_{1},N_{1}$ & $72+4s_{p}+50s_{N}+s_{N}\,s_{p}+12s_{N}^{2}+s_{N}^{3}$\\

17 & 4 & R & $p_{1},N_{1}$ & $67+3s_{p}+49s_{N}+s_{N}\,s_{p}+12s_{N}^{2}+s_{N}^{3}$\\

18 & 1 & R & $f_{1},p_{1},N_{1}$ & $26+s_{f}+s_{p}+11s_{N}+s_{N}^{2}$\\

18 & 2 & R & $f_{1},p_{1},N_{1}$ & $24+s_{f}+s_{p}+10s_{N}+s_{N}^{2}$\\

18 & 3 & R & $f_{1},p_{1},N_{1}$ & $13+s_{f}+s_{p}+3s_{N}$\\

18 & 4 & R & $f_{1},N_{1}$ & $8+s_{f}+2s_{N}$\\

19 & 1 & R & $g_{1},p_{1},N_{1}$ & $26+s_{g}+s_{p}+11s_{N}+s_{N}^{2}$\\

19 & 2 & R & $g_{1},p_{1},N_{1}$ & $24+s_{g}+s_{p}+10s_{N}+s_{N}^{2}$\\

19 & 3 & R & $g_{1},p_{1},N_{1}$ & $13+s_{g}+s_{p}+3s_{N}$\\

19 & 4 & R & $g_{1},N_{1}$ & $8+s_{g}+2s_{N}$\\

20 & 1 & R & $N_{1}$ & $254+256s_{N}+96s_{N}^{2}+16s_{N}^{3}+s_{N}^{4}$\\

20 & 2 & R & $N_{1}$ & $252+255s_{N}+96s_{N}^{2}+16s_{N}^{3}+s_{N}^{4}$\\

20 & 3 & R & $N_{1}$ & $241+248s_{N}+95s_{N}^{2}+16s_{N}^{3}+s_{N}^{4}$\\

21 & 1 & R & $p_{1},N_{1}$ & $334+16s_{p}+312s_{N}+8s_{N}\,s_{p}+109s_{N}^{2}+s_{N}^{2}\,s_{p}+17s_{N}^{3}+s_{N}^{4}$\\

21 & 2 & R & $p_{1},N_{1}$ & $332+16s_{p}+311s_{N}+8s_{N}\,s_{p}+109s_{N}^{2}+s_{N}^{2}\,s_{p}+17s_{N}^{3}+s_{N}^{4}$\\

21 & 3 & R & $p_{1},N_{1}$ & $321+16s_{p}+304s_{N}+8s_{N}\,s_{p}+108s_{N}^{2}+s_{N}^{2}\,s_{p}+17s_{N}^{3}+s_{N}^{4}$\\

21 & 4 & R & $p_{1},N_{1}$ & $316+15s_{p}+303s_{N}+8s_{N}\,s_{p}+108s_{N}^{2}+s_{N}^{2}\,s_{p}+17s_{N}^{3}+s_{N}^{4}$\\

22 & 1 & R & $p_{1},N_{1}$ & $439+42s_{p}+378s_{N}+s_{p}^{2}+18s_{N}\,s_{p}+123s_{N}^{2}+2s_{N}^{2}\,s_{p}+18s_{N}^{3}+s_{N}^{4}$\\

22 & 2 & R & $p_{1},N_{1}$ & $437+42s_{p}+377s_{N}+s_{p}^{2}+18s_{N}\,s_{p}+123s_{N}^{2}+2s_{N}^{2}\,s_{p}+18s_{N}^{3}+s_{N}^{4}$\\

22 & 3 & R & $p_{1},N_{1}$ & $426+42s_{p}+370s_{N}+s_{p}^{2}+18s_{N}\,s_{p}+122s_{N}^{2}+2s_{N}^{2}\,s_{p}+18s_{N}^{3}+s_{N}^{4}$\\

22 & 4 & R & $p_{1},N_{1}$ & $421+41s_{p}+369s_{N}+s_{p}^{2}+18s_{N}\,s_{p}+122s_{N}^{2}+2s_{N}^{2}\,s_{p}+18s_{N}^{3}+s_{N}^{4}$\\

23 & 1 & R & $f_{1},p_{1},N_{1}$ & $98+4s_{f}+4s_{p}+65s_{N}+s_{N}\,s_{f}+s_{N}\,s_{p}+14s_{N}^{2}+s_{N}^{3}$\\

23 & 2 & R & $f_{1},p_{1},N_{1}$ & $96+4s_{f}+4s_{p}+64s_{N}+s_{N}\,s_{f}+s_{N}\,s_{p}+14s_{N}^{2}+s_{N}^{3}$\\

23 & 3 & R & $f_{1},p_{1},N_{1}$ & $85+4s_{f}+4s_{p}+57s_{N}+s_{N}\,s_{f}+s_{N}\,s_{p}+13s_{N}^{2}+s_{N}^{3}$\\

23 & 4 & R & $f_{1},p_{1},N_{1}$ & $80+4s_{f}+3s_{p}+56s_{N}+s_{N}\,s_{f}+s_{N}\,s_{p}+13s_{N}^{2}+s_{N}^{3}$\\

24 & 1 & R & $g_{1},p_{1},N_{1}$ & $98+4s_{g}+4s_{p}+65s_{N}+s_{N}\,s_{g}+s_{N}\,s_{p}+14s_{N}^{2}+s_{N}^{3}$\\

24 & 2 & R & $g_{1},p_{1},N_{1}$ & $96+4s_{g}+4s_{p}+64s_{N}+s_{N}\,s_{g}+s_{N}\,s_{p}+14s_{N}^{2}+s_{N}^{3}$\\

24 & 3 & R & $g_{1},p_{1},N_{1}$ & $85+4s_{g}+4s_{p}+57s_{N}+s_{N}\,s_{g}+s_{N}\,s_{p}+13s_{N}^{2}+s_{N}^{3}$\\

24 & 4 & R & $g_{1},p_{1},N_{1}$ & $80+4s_{g}+3s_{p}+56s_{N}+s_{N}\,s_{g}+s_{N}\,s_{p}+13s_{N}^{2}+s_{N}^{3}$\\

25 & 1 & R & $(h_{1})_{1},f_{1},p_{1},N_{1}$ & $26+s_{h_{1}}+s_{f}+s_{p}+11s_{N}+s_{N}^{2}$\\

25 & 2 & R & $(h_{1})_{1},f_{1},p_{1},N_{1}$ & $24+s_{h_{1}}+s_{f}+s_{p}+10s_{N}+s_{N}^{2}$\\

25 & 3 & R & $(h_{1})_{1},f_{1},p_{1},N_{1}$ & $13+s_{h_{1}}+s_{f}+s_{p}+3s_{N}$\\

25 & 4 & R & $(h_{1})_{1},f_{1},N_{1}$ & $8+s_{h_{1}}+s_{f}+2s_{N}$\\

26 & 1 & R & $(h_{2})_{1},p_{1},N_{1}$ & $460+s_{h_{2}}+43s_{p}+387s_{N}+s_{p}^{2}+18s_{N}\,s_{p}+124s_{N}^{2}+2s_{N}^{2}\,s_{p}+18s_{N}^{3}+s_{N}^{4}$\\

26 & 2 & R & $(h_{2})_{1},p_{1},N_{1}$ & $458+s_{h_{2}}+43s_{p}+386s_{N}+s_{p}^{2}+18s_{N}\,s_{p}+124s_{N}^{2}+2s_{N}^{2}\,s_{p}+18s_{N}^{3}+s_{N}^{4}$\\

26 & 3 & R & $(h_{2})_{1},p_{1},N_{1}$ & $447+s_{h_{2}}+43s_{p}+379s_{N}+s_{p}^{2}+18s_{N}\,s_{p}+123s_{N}^{2}+2s_{N}^{2}\,s_{p}+18s_{N}^{3}+s_{N}^{4}$\\

26 & 4 & R & $(h_{2})_{1},p_{1},N_{1}$ & $442+s_{h_{2}}+42s_{p}+378s_{N}+s_{p}^{2}+18s_{N}\,s_{p}+123s_{N}^{2}+2s_{N}^{2}\,s_{p}+18s_{N}^{3}+s_{N}^{4}$\\

27 & 1 & R & $(h_{3})_{1},f_{1},p_{1},N_{1}$ & $26+s_{h_{3}}+s_{f}+s_{p}+11s_{N}+s_{N}^{2}$\\

27 & 2 & R & $(h_{3})_{1},f_{1},p_{1},N_{1}$ & $24+s_{h_{3}}+s_{f}+s_{p}+10s_{N}+s_{N}^{2}$\\

27 & 3 & R & $(h_{3})_{1},f_{1},p_{1},N_{1}$ & $13+s_{h_{3}}+s_{f}+s_{p}+3s_{N}$\\

27 & 4 & R & $(h_{3})_{1},f_{1},N_{1}$ & $8+s_{h_{3}}+s_{f}+2s_{N}$\\

28 & 1 & R & $(h_{4})_{1},g_{1},p_{1},N_{1}$ & $26+s_{h_{4}}+s_{g}+s_{p}+11s_{N}+s_{N}^{2}$\\

28 & 2 & R & $(h_{4})_{1},g_{1},p_{1},N_{1}$ & $24+s_{h_{4}}+s_{g}+s_{p}+10s_{N}+s_{N}^{2}$\\

28 & 3 & R & $(h_{4})_{1},g_{1},p_{1},N_{1}$ & $13+s_{h_{4}}+s_{g}+s_{p}+3s_{N}$\\

28 & 4 & R & $(h_{4})_{1},g_{1},N_{1}$ & $8+s_{h_{4}}+s_{g}+2s_{N}$\\

29 & 1 & R & $N_{1}$ & $1021+1280s_{N}+640s_{N}^{2}+160s_{N}^{3}+20s_{N}^{4}+s_{N}^{5}$\\

29 & 2 & R & $N_{1}$ & $1019+1279s_{N}+640s_{N}^{2}+160s_{N}^{3}+20s_{N}^{4}+s_{N}^{5}$\\

29 & 3 & R & $N_{1}$ & $1008+1272s_{N}+639s_{N}^{2}+160s_{N}^{3}+20s_{N}^{4}+s_{N}^{5}$\\

30 & 1 & R & $p_{1},N_{1}$ & $1341+64s_{p}+1584s_{N}+48s_{N}\,s_{p}+748s_{N}^{2}+12s_{N}^{2}\,s_{p}+177s_{N}^{3}+s_{N}^{3}\,s_{p}+21s_{N}^{4}+s_{N}^{5}$\\

30 & 2 & R & $p_{1},N_{1}$ & $1339+64s_{p}+1583s_{N}+48s_{N}\,s_{p}+748s_{N}^{2}+12s_{N}^{2}\,s_{p}+177s_{N}^{3}+s_{N}^{3}\,s_{p}+21s_{N}^{4}+s_{N}^{5}$\\

30 & 3 & R & $p_{1},N_{1}$ & $1328+64s_{p}+1576s_{N}+48s_{N}\,s_{p}+747s_{N}^{2}+12s_{N}^{2}\,s_{p}+177s_{N}^{3}+s_{N}^{3}\,s_{p}+21s_{N}^{4}+s_{N}^{5}$\\

30 & 4 & R & $p_{1},N_{1}$ & $1323+63s_{p}+1575s_{N}+48s_{N}\,s_{p}+747s_{N}^{2}+12s_{N}^{2}\,s_{p}+177s_{N}^{3}+s_{N}^{3}\,s_{p}+21s_{N}^{4}+s_{N}^{5}$\\

31 & 1 & R & $p_{1},N_{1}$ & $1761+168s_{p}+1953s_{N}+4s_{p}^{2}+114s_{N}\,s_{p}+870s_{N}^{2}+s_{N}\,s_{p}^{2}+26s_{N}^{2}\,s_{p}+195s_{N}^{3}+2s_{N}^{3}\,s_{p}+22s_{N}^{4}+s_{N}^{5}$\\

31 & 2 & R & $p_{1},N_{1}$ & $1759+168s_{p}+1952s_{N}+4s_{p}^{2}+114s_{N}\,s_{p}+870s_{N}^{2}+s_{N}\,s_{p}^{2}+26s_{N}^{2}\,s_{p}+195s_{N}^{3}+2s_{N}^{3}\,s_{p}+22s_{N}^{4}+s_{N}^{5}$\\

31 & 3 & R & $p_{1},N_{1}$ & $1748+168s_{p}+1945s_{N}+4s_{p}^{2}+114s_{N}\,s_{p}+869s_{N}^{2}+s_{N}\,s_{p}^{2}+26s_{N}^{2}\,s_{p}+195s_{N}^{3}+2s_{N}^{3}\,s_{p}+22s_{N}^{4}+s_{N}^{5}$\\

31 & 4 & R & $p_{1},N_{1}$ & $1743+167s_{p}+1944s_{N}+4s_{p}^{2}+114s_{N}\,s_{p}+869s_{N}^{2}+s_{N}\,s_{p}^{2}+26s_{N}^{2}\,s_{p}+195s_{N}^{3}+2s_{N}^{3}\,s_{p}+22s_{N}^{4}+s_{N}^{5}$\\

32 & 1 & R & $f_{1},p_{1},N_{1}$ & $397+16s_{f}+16s_{p}+360s_{N}+8s_{N}\,s_{f}+8s_{N}\,s_{p}+121s_{N}^{2}+s_{N}^{2}\,s_{f}+s_{N}^{2}\,s_{p}+18s_{N}^{3}+s_{N}^{4}$\\

32 & 2 & R & $f_{1},p_{1},N_{1}$ & $395+16s_{f}+16s_{p}+359s_{N}+8s_{N}\,s_{f}+8s_{N}\,s_{p}+121s_{N}^{2}+s_{N}^{2}\,s_{f}+s_{N}^{2}\,s_{p}+18s_{N}^{3}+s_{N}^{4}$\\

32 & 3 & R & $f_{1},p_{1},N_{1}$ & $384+16s_{f}+16s_{p}+352s_{N}+8s_{N}\,s_{f}+8s_{N}\,s_{p}+120s_{N}^{2}+s_{N}^{2}\,s_{f}+s_{N}^{2}\,s_{p}+18s_{N}^{3}+s_{N}^{4}$\\

32 & 4 & R & $f_{1},p_{1},N_{1}$ & $379+16s_{f}+15s_{p}+351s_{N}+8s_{N}\,s_{f}+8s_{N}\,s_{p}+120s_{N}^{2}+s_{N}^{2}\,s_{f}+s_{N}^{2}\,s_{p}+18s_{N}^{3}+s_{N}^{4}$\\

33 & 1 & R & $g_{1},p_{1},N_{1}$ & $397+16s_{g}+16s_{p}+360s_{N}+8s_{N}\,s_{g}+8s_{N}\,s_{p}+121s_{N}^{2}+s_{N}^{2}\,s_{g}+s_{N}^{2}\,s_{p}+18s_{N}^{3}+s_{N}^{4}$\\

33 & 2 & R & $g_{1},p_{1},N_{1}$ & $395+16s_{g}+16s_{p}+359s_{N}+8s_{N}\,s_{g}+8s_{N}\,s_{p}+121s_{N}^{2}+s_{N}^{2}\,s_{g}+s_{N}^{2}\,s_{p}+18s_{N}^{3}+s_{N}^{4}$\\

33 & 3 & R & $g_{1},p_{1},N_{1}$ & $384+16s_{g}+16s_{p}+352s_{N}+8s_{N}\,s_{g}+8s_{N}\,s_{p}+120s_{N}^{2}+s_{N}^{2}\,s_{g}+s_{N}^{2}\,s_{p}+18s_{N}^{3}+s_{N}^{4}$\\

33 & 4 & R & $g_{1},p_{1},N_{1}$ & $379+16s_{g}+15s_{p}+351s_{N}+8s_{N}\,s_{g}+8s_{N}\,s_{p}+120s_{N}^{2}+s_{N}^{2}\,s_{g}+s_{N}^{2}\,s_{p}+18s_{N}^{3}+s_{N}^{4}$\\

34 & 1 & R & $f_{1},p_{1},N_{1}$ & $522+21s_{f}+46s_{p}+435s_{N}+s_{p}\,s_{f}+s_{p}^{2}+9s_{N}\,s_{f}+19s_{N}\,s_{p}+136s_{N}^{2}+s_{N}^{2}\,s_{f}+2s_{N}^{2}\,s_{p}+19s_{N}^{3}+s_{N}^{4}$\\

34 & 2 & R & $f_{1},p_{1},N_{1}$ & $520+21s_{f}+46s_{p}+434s_{N}+s_{p}\,s_{f}+s_{p}^{2}+9s_{N}\,s_{f}+19s_{N}\,s_{p}+136s_{N}^{2}+s_{N}^{2}\,s_{f}+2s_{N}^{2}\,s_{p}+19s_{N}^{3}+s_{N}^{4}$\\

34 & 3 & R & $f_{1},p_{1},N_{1}$ & $509+21s_{f}+46s_{p}+427s_{N}+s_{p}\,s_{f}+s_{p}^{2}+9s_{N}\,s_{f}+19s_{N}\,s_{p}+135s_{N}^{2}+s_{N}^{2}\,s_{f}+2s_{N}^{2}\,s_{p}+19s_{N}^{3}+s_{N}^{4}$\\

34 & 4 & R & $f_{1},p_{1},N_{1}$ & $504+21s_{f}+45s_{p}+426s_{N}+s_{p}\,s_{f}+s_{p}^{2}+9s_{N}\,s_{f}+19s_{N}\,s_{p}+135s_{N}^{2}+s_{N}^{2}\,s_{f}+2s_{N}^{2}\,s_{p}+19s_{N}^{3}+s_{N}^{4}$\\

35 & 1 & R & $g_{1},p_{1},N_{1}$ & $522+21s_{g}+46s_{p}+435s_{N}+s_{p}\,s_{g}+s_{p}^{2}+9s_{N}\,s_{g}+19s_{N}\,s_{p}+136s_{N}^{2}+s_{N}^{2}\,s_{g}+2s_{N}^{2}\,s_{p}+19s_{N}^{3}+s_{N}^{4}$\\

35 & 2 & R & $g_{1},p_{1},N_{1}$ & $520+21s_{g}+46s_{p}+434s_{N}+s_{p}\,s_{g}+s_{p}^{2}+9s_{N}\,s_{g}+19s_{N}\,s_{p}+136s_{N}^{2}+s_{N}^{2}\,s_{g}+2s_{N}^{2}\,s_{p}+19s_{N}^{3}+s_{N}^{4}$\\

35 & 3 & R & $g_{1},p_{1},N_{1}$ & $509+21s_{g}+46s_{p}+427s_{N}+s_{p}\,s_{g}+s_{p}^{2}+9s_{N}\,s_{g}+19s_{N}\,s_{p}+135s_{N}^{2}+s_{N}^{2}\,s_{g}+2s_{N}^{2}\,s_{p}+19s_{N}^{3}+s_{N}^{4}$\\

35 & 4 & R & $g_{1},p_{1},N_{1}$ & $504+21s_{g}+45s_{p}+426s_{N}+s_{p}\,s_{g}+s_{p}^{2}+9s_{N}\,s_{g}+19s_{N}\,s_{p}+135s_{N}^{2}+s_{N}^{2}\,s_{g}+2s_{N}^{2}\,s_{p}+19s_{N}^{3}+s_{N}^{4}$\\

36 & 1 & R & $(h_{1})_{1},f_{1},p_{1},N_{1}$ & $109+4s_{h_{1}}+4s_{f}+4s_{p}+72s_{N}+s_{N}\,s_{h_{1}}+s_{N}\,s_{f}+s_{N}\,s_{p}+15s_{N}^{2}+s_{N}^{3}$\\

36 & 2 & R & $(h_{1})_{1},f_{1},p_{1},N_{1}$ & $107+4s_{h_{1}}+4s_{f}+4s_{p}+71s_{N}+s_{N}\,s_{h_{1}}+s_{N}\,s_{f}+s_{N}\,s_{p}+15s_{N}^{2}+s_{N}^{3}$\\

36 & 3 & R & $(h_{1})_{1},f_{1},p_{1},N_{1}$ & $96+4s_{h_{1}}+4s_{f}+4s_{p}+64s_{N}+s_{N}\,s_{h_{1}}+s_{N}\,s_{f}+s_{N}\,s_{p}+14s_{N}^{2}+s_{N}^{3}$\\

36 & 4 & R & $(h_{1})_{1},f_{1},p_{1},N_{1}$ & $91+4s_{h_{1}}+4s_{f}+3s_{p}+63s_{N}+s_{N}\,s_{h_{1}}+s_{N}\,s_{f}+s_{N}\,s_{p}+14s_{N}^{2}+s_{N}^{3}$\\

37 & 1 & R & $(h_{2})_{1},p_{1},N_{1}$ & $1845+4s_{h_{2}}+172s_{p}+2010s_{N}+4s_{p}^{2}+s_{N}\,s_{h_{2}}+115s_{N}\,s_{p}+883s_{N}^{2}+s_{N}\,s_{p}^{2}+26s_{N}^{2}\,s_{p}+196s_{N}^{3}+2s_{N}^{3}\,s_{p}+22s_{N}^{4}+s_{N}^{5}$\\

37 & 2 & R & $(h_{2})_{1},p_{1},N_{1}$ & $1843+4s_{h_{2}}+172s_{p}+2009s_{N}+4s_{p}^{2}+s_{N}\,s_{h_{2}}+115s_{N}\,s_{p}+883s_{N}^{2}+s_{N}\,s_{p}^{2}+26s_{N}^{2}\,s_{p}+196s_{N}^{3}+2s_{N}^{3}\,s_{p}+22s_{N}^{4}+s_{N}^{5}$\\

37 & 3 & R & $(h_{2})_{1},p_{1},N_{1}$ & $1832+4s_{h_{2}}+172s_{p}+2002s_{N}+4s_{p}^{2}+s_{N}\,s_{h_{2}}+115s_{N}\,s_{p}+882s_{N}^{2}+s_{N}\,s_{p}^{2}+26s_{N}^{2}\,s_{p}+196s_{N}^{3}+2s_{N}^{3}\,s_{p}+22s_{N}^{4}+s_{N}^{5}$\\

37 & 4 & R & $(h_{2})_{1},p_{1},N_{1}$ & $1827+4s_{h_{2}}+171s_{p}+2001s_{N}+4s_{p}^{2}+s_{N}\,s_{h_{2}}+115s_{N}\,s_{p}+882s_{N}^{2}+s_{N}\,s_{p}^{2}+26s_{N}^{2}\,s_{p}+196s_{N}^{3}+2s_{N}^{3}\,s_{p}+22s_{N}^{4}+s_{N}^{5}$\\

38 & 1 & R & $(h_{3})_{1},f_{1},p_{1},N_{1}$ & $109+4s_{h_{3}}+4s_{f}+4s_{p}+72s_{N}+s_{N}\,s_{h_{3}}+s_{N}\,s_{f}+s_{N}\,s_{p}+15s_{N}^{2}+s_{N}^{3}$\\

38 & 2 & R & $(h_{3})_{1},f_{1},p_{1},N_{1}$ & $107+4s_{h_{3}}+4s_{f}+4s_{p}+71s_{N}+s_{N}\,s_{h_{3}}+s_{N}\,s_{f}+s_{N}\,s_{p}+15s_{N}^{2}+s_{N}^{3}$\\

38 & 3 & R & $(h_{3})_{1},f_{1},p_{1},N_{1}$ & $96+4s_{h_{3}}+4s_{f}+4s_{p}+64s_{N}+s_{N}\,s_{h_{3}}+s_{N}\,s_{f}+s_{N}\,s_{p}+14s_{N}^{2}+s_{N}^{3}$\\

38 & 4 & R & $(h_{3})_{1},f_{1},p_{1},N_{1}$ & $91+4s_{h_{3}}+4s_{f}+3s_{p}+63s_{N}+s_{N}\,s_{h_{3}}+s_{N}\,s_{f}+s_{N}\,s_{p}+14s_{N}^{2}+s_{N}^{3}$\\

39 & 1 & R & $(h_{4})_{1},g_{1},p_{1},N_{1}$ & $109+4s_{h_{4}}+4s_{g}+4s_{p}+72s_{N}+s_{N}\,s_{h_{4}}+s_{N}\,s_{g}+s_{N}\,s_{p}+15s_{N}^{2}+s_{N}^{3}$\\

39 & 2 & R & $(h_{4})_{1},g_{1},p_{1},N_{1}$ & $107+4s_{h_{4}}+4s_{g}+4s_{p}+71s_{N}+s_{N}\,s_{h_{4}}+s_{N}\,s_{g}+s_{N}\,s_{p}+15s_{N}^{2}+s_{N}^{3}$\\

39 & 3 & R & $(h_{4})_{1},g_{1},p_{1},N_{1}$ & $96+4s_{h_{4}}+4s_{g}+4s_{p}+64s_{N}+s_{N}\,s_{h_{4}}+s_{N}\,s_{g}+s_{N}\,s_{p}+14s_{N}^{2}+s_{N}^{3}$\\

39 & 4 & R & $(h_{4})_{1},g_{1},p_{1},N_{1}$ & $91+4s_{h_{4}}+4s_{g}+3s_{p}+63s_{N}+s_{N}\,s_{h_{4}}+s_{N}\,s_{g}+s_{N}\,s_{p}+14s_{N}^{2}+s_{N}^{3}$\\

40 & 1 & R & $(u_{1})_{1},p_{1},N_{1}$ & $19+s_{u_{1}}+s_{p}+9s_{N}+s_{N}^{2}$\\

40 & 2 & R & $(u_{1})_{1},p_{1},N_{1}$ & $17+s_{u_{1}}+s_{p}+8s_{N}+s_{N}^{2}$\\

40 & 3 & R & $(u_{1})_{1},p_{1},N_{1}$ & $6+s_{u_{1}}+s_{p}+s_{N}$\\

40 & 4 & R & $(u_{1})_{1}$ & $1+s_{u_{1}}$\\

41 & 1 & R & $(u_{2})_{1},p_{1},N_{1}$ & $19+s_{u_{2}}+s_{p}+9s_{N}+s_{N}^{2}$\\

41 & 2 & R & $(u_{2})_{1},p_{1},N_{1}$ & $17+s_{u_{2}}+s_{p}+8s_{N}+s_{N}^{2}$\\

41 & 3 & R & $(u_{2})_{1},p_{1},N_{1}$ & $6+s_{u_{2}}+s_{p}+s_{N}$\\

41 & 4 & R & $(u_{2})_{1}$ & $1+s_{u_{2}}$\\

42 & 1 & R & $(u_{3})_{1},p_{1},N_{1}$ & $19+s_{u_{3}}+s_{p}+9s_{N}+s_{N}^{2}$\\

42 & 2 & R & $(u_{3})_{1},p_{1},N_{1}$ & $17+s_{u_{3}}+s_{p}+8s_{N}+s_{N}^{2}$\\

42 & 3 & R & $(u_{3})_{1},p_{1},N_{1}$ & $6+s_{u_{3}}+s_{p}+s_{N}$\\

42 & 4 & R & $(u_{3})_{1}$ & $1+s_{u_{3}}$\\

43 & 1 & R & $(u_{4})_{1},p_{1},N_{1}$ & $19+s_{u_{4}}+s_{p}+9s_{N}+s_{N}^{2}$\\

43 & 2 & R & $(u_{4})_{1},p_{1},N_{1}$ & $17+s_{u_{4}}+s_{p}+8s_{N}+s_{N}^{2}$\\

43 & 3 & R & $(u_{4})_{1},p_{1},N_{1}$ & $6+s_{u_{4}}+s_{p}+s_{N}$\\

43 & 4 & R & $(u_{4})_{1}$ & $1+s_{u_{4}}$\\

44 & 1 & R & $(u_{5})_{1},p_{1},N_{1}$ & $19+s_{u_{5}}+s_{p}+9s_{N}+s_{N}^{2}$\\

44 & 2 & R & $(u_{5})_{1},p_{1},N_{1}$ & $17+s_{u_{5}}+s_{p}+8s_{N}+s_{N}^{2}$\\

44 & 3 & R & $(u_{5})_{1},p_{1},N_{1}$ & $6+s_{u_{5}}+s_{p}+s_{N}$\\

44 & 4 & R & $(u_{5})_{1}$ & $1+s_{u_{5}}$\\

45 & 1 & R & $(u_{6})_{1},p_{1},N_{1}$ & $19+s_{u_{6}}+s_{p}+9s_{N}+s_{N}^{2}$\\

45 & 2 & R & $(u_{6})_{1},p_{1},N_{1}$ & $17+s_{u_{6}}+s_{p}+8s_{N}+s_{N}^{2}$\\

45 & 3 & R & $(u_{6})_{1},p_{1},N_{1}$ & $6+s_{u_{6}}+s_{p}+s_{N}$\\

45 & 4 & R & $(u_{6})_{1}$ & $1+s_{u_{6}}$\\

46 & 1 & R & $(u_{7})_{1},p_{1},N_{1}$ & $19+s_{u_{7}}+s_{p}+9s_{N}+s_{N}^{2}$\\

46 & 2 & R & $(u_{7})_{1},p_{1},N_{1}$ & $17+s_{u_{7}}+s_{p}+8s_{N}+s_{N}^{2}$\\

46 & 3 & R & $(u_{7})_{1},p_{1},N_{1}$ & $6+s_{u_{7}}+s_{p}+s_{N}$\\

46 & 4 & R & $(u_{7})_{1}$ & $1+s_{u_{7}}$\\

47 & 1 & R & $(u_{8})_{1},p_{1},N_{1}$ & $19+s_{u_{8}}+s_{p}+9s_{N}+s_{N}^{2}$\\

47 & 2 & R & $(u_{8})_{1},p_{1},N_{1}$ & $17+s_{u_{8}}+s_{p}+8s_{N}+s_{N}^{2}$\\

47 & 3 & R & $(u_{8})_{1},p_{1},N_{1}$ & $6+s_{u_{8}}+s_{p}+s_{N}$\\

47 & 4 & R & $(u_{8})_{1}$ & $1+s_{u_{8}}$\\

48 & 1 & R & $(u_{9})_{1},p_{1},N_{1}$ & $19+s_{u_{9}}+s_{p}+9s_{N}+s_{N}^{2}$\\

48 & 2 & R & $(u_{9})_{1},p_{1},N_{1}$ & $17+s_{u_{9}}+s_{p}+8s_{N}+s_{N}^{2}$\\

48 & 3 & R & $(u_{9})_{1},p_{1},N_{1}$ & $6+s_{u_{9}}+s_{p}+s_{N}$\\*

48 & 4 & R & $(u_{9})_{1}$ & $1+s_{u_{9}}$\\*

\bottomrule\caption{All 186 certified fixed comparisons in the weight-6+3 matrix.}\end{longtable}\endgroup

\Needspace{13\baselineskip}\section{The weight-5+4 matrix}
\index{height!comparison!weight $5+4$}\index{certificate!matrix!weight $5+4$}

\begingroup\footnotesize\renewcommand{\arraystretch}{1.32}\setlength{\tabcolsep}{3pt}
\begin{longtable}{@{}rrcp{.21\textwidth}p{.54\textwidth}@{}}
\toprule Row & Col. & Winner & Substitutions & Nonnegative difference polynomial\\\midrule
\endfirsthead
\multicolumn{5}{c}{\textit{Exact comparison certificates, continued}}\\
\toprule Row & Col. & Winner & Substitutions & Nonnegative difference polynomial\\\midrule\endhead
\midrule\multicolumn{5}{r}{\textit{Continued on the next page}}\\\endfoot
\endlastfoot

1 & 1 & R & $\text{none}$ & $1$\\

1 & 2 & C & $N_{1}$ & $1+s_{N}$\\

1 & 3 & C & $N_{1}$ & $3+2s_{N}$\\

1 & 4 & C & $N_{1}$ & $12+8s_{N}+s_{N}^{2}$\\

1 & 5 & C & $p_{1},N_{1}$ & $17+s_{p}+9s_{N}+s_{N}^{2}$\\

1 & 6 & C & $N_{1}$ & $59+48s_{N}+12s_{N}^{2}+s_{N}^{3}$\\

1 & 7 & C & $p_{1},N_{1}$ & $79+4s_{p}+57s_{N}+s_{N}\,s_{p}+13s_{N}^{2}+s_{N}^{3}$\\

1 & 8 & C & $f_{1},p_{1},N_{1}$ & $20+s_{f}+s_{p}+10s_{N}+s_{N}^{2}$\\

1 & 9 & C & $g_{1},p_{1},N_{1}$ & $20+s_{g}+s_{p}+10s_{N}+s_{N}^{2}$\\

2 & 1 & R & $N_{1}$ & $3+s_{N}$\\

2 & 2 & R & $\text{none}$ & $1$\\

2 & 3 & C & $N_{1}$ & $1+s_{N}$\\

2 & 4 & C & $N_{1}$ & $10+7s_{N}+s_{N}^{2}$\\

2 & 5 & C & $p_{1},N_{1}$ & $15+s_{p}+8s_{N}+s_{N}^{2}$\\

2 & 6 & C & $N_{1}$ & $57+47s_{N}+12s_{N}^{2}+s_{N}^{3}$\\

2 & 7 & C & $p_{1},N_{1}$ & $77+4s_{p}+56s_{N}+s_{N}\,s_{p}+13s_{N}^{2}+s_{N}^{3}$\\

2 & 8 & C & $f_{1},p_{1},N_{1}$ & $18+s_{f}+s_{p}+9s_{N}+s_{N}^{2}$\\

2 & 9 & C & $g_{1},p_{1},N_{1}$ & $18+s_{g}+s_{p}+9s_{N}+s_{N}^{2}$\\

3 & 1 & R & $N_{1}$ & $5+2s_{N}$\\

3 & 2 & R & $N_{1}$ & $3+s_{N}$\\

3 & 3 & R & $\text{none}$ & $1$\\

3 & 4 & C & $N_{1}$ & $8+6s_{N}+s_{N}^{2}$\\

3 & 5 & C & $p_{1},N_{1}$ & $13+s_{p}+7s_{N}+s_{N}^{2}$\\

3 & 6 & C & $N_{1}$ & $55+46s_{N}+12s_{N}^{2}+s_{N}^{3}$\\

3 & 7 & C & $p_{1},N_{1}$ & $75+4s_{p}+55s_{N}+s_{N}\,s_{p}+13s_{N}^{2}+s_{N}^{3}$\\

3 & 8 & C & $f_{1},p_{1},N_{1}$ & $16+s_{f}+s_{p}+8s_{N}+s_{N}^{2}$\\

3 & 9 & C & $g_{1},p_{1},N_{1}$ & $16+s_{g}+s_{p}+8s_{N}+s_{N}^{2}$\\

4 & 1 & R & $N_{1}$ & $14+8s_{N}+s_{N}^{2}$\\

4 & 2 & R & $N_{1}$ & $12+7s_{N}+s_{N}^{2}$\\

4 & 3 & R & $N_{1}$ & $10+6s_{N}+s_{N}^{2}$\\

4 & 4 & R & $\text{none}$ & $1$\\

4 & 5 & C & $p_{1},N_{1}$ & $4+s_{p}+s_{N}$\\

4 & 6 & C & $N_{1}$ & $46+40s_{N}+11s_{N}^{2}+s_{N}^{3}$\\

4 & 7 & C & $p_{1},N_{1}$ & $66+4s_{p}+49s_{N}+s_{N}\,s_{p}+12s_{N}^{2}+s_{N}^{3}$\\

4 & 8 & C & $f_{1},p_{1},N_{1}$ & $7+s_{f}+s_{p}+2s_{N}$\\

4 & 9 & C & $g_{1},p_{1},N_{1}$ & $7+s_{g}+s_{p}+2s_{N}$\\

5 & 1 & R & $p_{1},N_{1}$ & $19+s_{p}+9s_{N}+s_{N}^{2}$\\

5 & 2 & R & $p_{1},N_{1}$ & $17+s_{p}+8s_{N}+s_{N}^{2}$\\

5 & 3 & R & $p_{1},N_{1}$ & $15+s_{p}+7s_{N}+s_{N}^{2}$\\

5 & 4 & R & $p_{1},N_{1}$ & $6+s_{p}+s_{N}$\\

5 & 5 & R & $\text{none}$ & $1$\\

5 & 7 & C & $p_{1},N_{1}$ & $61+3s_{p}+48s_{N}+s_{N}\,s_{p}+12s_{N}^{2}+s_{N}^{3}$\\

5 & 8 & C & $f_{1},N_{1}$ & $2+s_{f}+s_{N}$\\

5 & 9 & C & $g_{1},N_{1}$ & $2+s_{g}+s_{N}$\\

6 & 1 & R & $N_{1}$ & $16+9s_{N}+s_{N}^{2}$\\

6 & 2 & R & $N_{1}$ & $14+8s_{N}+s_{N}^{2}$\\

6 & 3 & R & $N_{1}$ & $12+7s_{N}+s_{N}^{2}$\\

6 & 4 & R & $N_{1}$ & $3+s_{N}$\\

6 & 5 & C & $p_{1}$ & $2+s_{p}$\\

6 & 6 & C & $N_{1}$ & $44+39s_{N}+11s_{N}^{2}+s_{N}^{3}$\\

6 & 7 & C & $p_{1},N_{1}$ & $64+4s_{p}+48s_{N}+s_{N}\,s_{p}+12s_{N}^{2}+s_{N}^{3}$\\

6 & 8 & C & $f_{1},p_{1},N_{1}$ & $5+s_{f}+s_{p}+s_{N}$\\

6 & 9 & C & $g_{1},p_{1},N_{1}$ & $5+s_{g}+s_{p}+s_{N}$\\

7 & 1 & R & $p_{1},N_{1}$ & $21+s_{p}+10s_{N}+s_{N}^{2}$\\

7 & 2 & R & $p_{1},N_{1}$ & $19+s_{p}+9s_{N}+s_{N}^{2}$\\

7 & 3 & R & $p_{1},N_{1}$ & $17+s_{p}+8s_{N}+s_{N}^{2}$\\

7 & 4 & R & $p_{1},N_{1}$ & $8+s_{p}+2s_{N}$\\

7 & 5 & R & $N_{1}$ & $3+s_{N}$\\

7 & 7 & C & $p_{1},N_{1}$ & $59+3s_{p}+47s_{N}+s_{N}\,s_{p}+12s_{N}^{2}+s_{N}^{3}$\\

7 & 8 & C & $f_{1}$ & $s_{f}$\\

7 & 9 & C & $g_{1}$ & $s_{g}$\\

8 & 1 & R & $N_{1}$ & $61+48s_{N}+12s_{N}^{2}+s_{N}^{3}$\\

8 & 2 & R & $N_{1}$ & $59+47s_{N}+12s_{N}^{2}+s_{N}^{3}$\\

8 & 3 & R & $N_{1}$ & $57+46s_{N}+12s_{N}^{2}+s_{N}^{3}$\\

8 & 4 & R & $N_{1}$ & $48+40s_{N}+11s_{N}^{2}+s_{N}^{3}$\\

8 & 6 & R & $\text{none}$ & $1$\\

8 & 7 & C & $p_{1},N_{1}$ & $19+4s_{p}+9s_{N}+s_{N}\,s_{p}+s_{N}^{2}$\\

8 & 9 & C & $g_{2},p_{1},N_{1}$ & $20+s_{g}+4s_{p}+9s_{N}+s_{N}\,s_{p}+s_{N}^{2}$\\

9 & 1 & R & $p_{1},N_{1}$ & $81+4s_{p}+57s_{N}+s_{N}\,s_{p}+13s_{N}^{2}+s_{N}^{3}$\\

9 & 2 & R & $p_{1},N_{1}$ & $79+4s_{p}+56s_{N}+s_{N}\,s_{p}+13s_{N}^{2}+s_{N}^{3}$\\

9 & 3 & R & $p_{1},N_{1}$ & $77+4s_{p}+55s_{N}+s_{N}\,s_{p}+13s_{N}^{2}+s_{N}^{3}$\\

9 & 4 & R & $p_{1},N_{1}$ & $68+4s_{p}+49s_{N}+s_{N}\,s_{p}+12s_{N}^{2}+s_{N}^{3}$\\

9 & 5 & R & $p_{1},N_{1}$ & $63+3s_{p}+48s_{N}+s_{N}\,s_{p}+12s_{N}^{2}+s_{N}^{3}$\\

9 & 6 & R & $p_{1},N_{1}$ & $21+4s_{p}+9s_{N}+s_{N}\,s_{p}+s_{N}^{2}$\\

9 & 7 & R & $\text{none}$ & $1$\\

9 & 9 & C & $g_{2}$ & $s_{g}$\\

10 & 1 & R & $f_{1},p_{1},N_{1}$ & $22+s_{f}+s_{p}+10s_{N}+s_{N}^{2}$\\

10 & 2 & R & $f_{1},p_{1},N_{1}$ & $20+s_{f}+s_{p}+9s_{N}+s_{N}^{2}$\\

10 & 3 & R & $f_{1},p_{1},N_{1}$ & $18+s_{f}+s_{p}+8s_{N}+s_{N}^{2}$\\

10 & 4 & R & $f_{1},p_{1},N_{1}$ & $9+s_{f}+s_{p}+2s_{N}$\\

10 & 5 & R & $f_{1},N_{1}$ & $4+s_{f}+s_{N}$\\

10 & 8 & R & $\text{none}$ & $1$\\

11 & 1 & R & $g_{1},p_{1},N_{1}$ & $22+s_{g}+s_{p}+10s_{N}+s_{N}^{2}$\\

11 & 2 & R & $g_{1},p_{1},N_{1}$ & $20+s_{g}+s_{p}+9s_{N}+s_{N}^{2}$\\

11 & 3 & R & $g_{1},p_{1},N_{1}$ & $18+s_{g}+s_{p}+8s_{N}+s_{N}^{2}$\\

11 & 4 & R & $g_{1},p_{1},N_{1}$ & $9+s_{g}+s_{p}+2s_{N}$\\

11 & 5 & R & $g_{1},N_{1}$ & $4+s_{g}+s_{N}$\\

11 & 6 & R & $g_{2},p_{1},N_{1}$ & $22+s_{g}+4s_{p}+9s_{N}+s_{N}\,s_{p}+s_{N}^{2}$\\

11 & 7 & R & $g_{2}$ & $2+s_{g}$\\

11 & 9 & R & $\text{none}$ & $1$\\

12 & 1 & R & $N_{1}$ & $252+256s_{N}+96s_{N}^{2}+16s_{N}^{3}+s_{N}^{4}$\\

12 & 2 & R & $N_{1}$ & $250+255s_{N}+96s_{N}^{2}+16s_{N}^{3}+s_{N}^{4}$\\

12 & 3 & R & $N_{1}$ & $248+254s_{N}+96s_{N}^{2}+16s_{N}^{3}+s_{N}^{4}$\\

12 & 4 & R & $N_{1}$ & $239+248s_{N}+95s_{N}^{2}+16s_{N}^{3}+s_{N}^{4}$\\

12 & 6 & R & $N_{1}$ & $192+208s_{N}+84s_{N}^{2}+15s_{N}^{3}+s_{N}^{4}$\\

13 & 1 & R & $p_{1},N_{1}$ & $332+16s_{p}+312s_{N}+8s_{N}\,s_{p}+109s_{N}^{2}+s_{N}^{2}\,s_{p}+17s_{N}^{3}+s_{N}^{4}$\\

13 & 2 & R & $p_{1},N_{1}$ & $330+16s_{p}+311s_{N}+8s_{N}\,s_{p}+109s_{N}^{2}+s_{N}^{2}\,s_{p}+17s_{N}^{3}+s_{N}^{4}$\\

13 & 3 & R & $p_{1},N_{1}$ & $328+16s_{p}+310s_{N}+8s_{N}\,s_{p}+109s_{N}^{2}+s_{N}^{2}\,s_{p}+17s_{N}^{3}+s_{N}^{4}$\\

13 & 4 & R & $p_{1},N_{1}$ & $319+16s_{p}+304s_{N}+8s_{N}\,s_{p}+108s_{N}^{2}+s_{N}^{2}\,s_{p}+17s_{N}^{3}+s_{N}^{4}$\\

13 & 5 & R & $p_{1},N_{1}$ & $314+15s_{p}+303s_{N}+8s_{N}\,s_{p}+108s_{N}^{2}+s_{N}^{2}\,s_{p}+17s_{N}^{3}+s_{N}^{4}$\\

13 & 6 & R & $p_{1},N_{1}$ & $272+16s_{p}+264s_{N}+8s_{N}\,s_{p}+97s_{N}^{2}+s_{N}^{2}\,s_{p}+16s_{N}^{3}+s_{N}^{4}$\\

13 & 7 & R & $p_{1},N_{1}$ & $252+12s_{p}+255s_{N}+7s_{N}\,s_{p}+96s_{N}^{2}+s_{N}^{2}\,s_{p}+16s_{N}^{3}+s_{N}^{4}$\\

14 & 1 & R & $p_{1},N_{1}$ & $437+42s_{p}+378s_{N}+s_{p}^{2}+18s_{N}\,s_{p}+123s_{N}^{2}+2s_{N}^{2}\,s_{p}+18s_{N}^{3}+s_{N}^{4}$\\

14 & 2 & R & $p_{1},N_{1}$ & $435+42s_{p}+377s_{N}+s_{p}^{2}+18s_{N}\,s_{p}+123s_{N}^{2}+2s_{N}^{2}\,s_{p}+18s_{N}^{3}+s_{N}^{4}$\\

14 & 3 & R & $p_{1},N_{1}$ & $433+42s_{p}+376s_{N}+s_{p}^{2}+18s_{N}\,s_{p}+123s_{N}^{2}+2s_{N}^{2}\,s_{p}+18s_{N}^{3}+s_{N}^{4}$\\

14 & 4 & R & $p_{1},N_{1}$ & $424+42s_{p}+370s_{N}+s_{p}^{2}+18s_{N}\,s_{p}+122s_{N}^{2}+2s_{N}^{2}\,s_{p}+18s_{N}^{3}+s_{N}^{4}$\\

14 & 5 & R & $p_{1},N_{1}$ & $419+41s_{p}+369s_{N}+s_{p}^{2}+18s_{N}\,s_{p}+122s_{N}^{2}+2s_{N}^{2}\,s_{p}+18s_{N}^{3}+s_{N}^{4}$\\

14 & 6 & R & $p_{1},N_{1}$ & $377+42s_{p}+330s_{N}+s_{p}^{2}+18s_{N}\,s_{p}+111s_{N}^{2}+2s_{N}^{2}\,s_{p}+17s_{N}^{3}+s_{N}^{4}$\\

14 & 7 & R & $p_{1},N_{1}$ & $357+38s_{p}+321s_{N}+s_{p}^{2}+17s_{N}\,s_{p}+110s_{N}^{2}+2s_{N}^{2}\,s_{p}+17s_{N}^{3}+s_{N}^{4}$\\

15 & 1 & R & $f_{1},p_{1},N_{1}$ & $96+4s_{f}+4s_{p}+65s_{N}+s_{N}\,s_{f}+s_{N}\,s_{p}+14s_{N}^{2}+s_{N}^{3}$\\

15 & 2 & R & $f_{1},p_{1},N_{1}$ & $94+4s_{f}+4s_{p}+64s_{N}+s_{N}\,s_{f}+s_{N}\,s_{p}+14s_{N}^{2}+s_{N}^{3}$\\

15 & 3 & R & $f_{1},p_{1},N_{1}$ & $92+4s_{f}+4s_{p}+63s_{N}+s_{N}\,s_{f}+s_{N}\,s_{p}+14s_{N}^{2}+s_{N}^{3}$\\

15 & 4 & R & $f_{1},p_{1},N_{1}$ & $83+4s_{f}+4s_{p}+57s_{N}+s_{N}\,s_{f}+s_{N}\,s_{p}+13s_{N}^{2}+s_{N}^{3}$\\

15 & 5 & R & $f_{1},p_{1},N_{1}$ & $78+4s_{f}+3s_{p}+56s_{N}+s_{N}\,s_{f}+s_{N}\,s_{p}+13s_{N}^{2}+s_{N}^{3}$\\

15 & 6 & R & $f_{1},p_{1},N_{1}$ & $36+4s_{f}+4s_{p}+17s_{N}+s_{N}\,s_{f}+s_{N}\,s_{p}+2s_{N}^{2}$\\

15 & 7 & R & $f_{1},N_{1}$ & $16+4s_{f}+8s_{N}+s_{N}\,s_{f}+s_{N}^{2}$\\

15 & 8 & R & $f_{1},p_{1},N_{1}$ & $75+3s_{f}+3s_{p}+55s_{N}+s_{N}\,s_{f}+s_{N}\,s_{p}+13s_{N}^{2}+s_{N}^{3}$\\

16 & 1 & R & $g_{1},p_{1},N_{1}$ & $96+4s_{g}+4s_{p}+65s_{N}+s_{N}\,s_{g}+s_{N}\,s_{p}+14s_{N}^{2}+s_{N}^{3}$\\

16 & 2 & R & $g_{1},p_{1},N_{1}$ & $94+4s_{g}+4s_{p}+64s_{N}+s_{N}\,s_{g}+s_{N}\,s_{p}+14s_{N}^{2}+s_{N}^{3}$\\

16 & 3 & R & $g_{1},p_{1},N_{1}$ & $92+4s_{g}+4s_{p}+63s_{N}+s_{N}\,s_{g}+s_{N}\,s_{p}+14s_{N}^{2}+s_{N}^{3}$\\

16 & 4 & R & $g_{1},p_{1},N_{1}$ & $83+4s_{g}+4s_{p}+57s_{N}+s_{N}\,s_{g}+s_{N}\,s_{p}+13s_{N}^{2}+s_{N}^{3}$\\

16 & 5 & R & $g_{1},p_{1},N_{1}$ & $78+4s_{g}+3s_{p}+56s_{N}+s_{N}\,s_{g}+s_{N}\,s_{p}+13s_{N}^{2}+s_{N}^{3}$\\

16 & 6 & R & $g_{1},p_{1},N_{1}$ & $36+4s_{g}+4s_{p}+17s_{N}+s_{N}\,s_{g}+s_{N}\,s_{p}+2s_{N}^{2}$\\

16 & 7 & R & $g_{1},N_{1}$ & $16+4s_{g}+8s_{N}+s_{N}\,s_{g}+s_{N}^{2}$\\

16 & 9 & R & $g_{1},p_{1},N_{1}$ & $75+3s_{g}+3s_{p}+55s_{N}+s_{N}\,s_{g}+s_{N}\,s_{p}+13s_{N}^{2}+s_{N}^{3}$\\

17 & 1 & R & $(h_{1})_{1},f_{1},p_{1},N_{1}$ & $24+s_{h_{1}}+s_{f}+s_{p}+11s_{N}+s_{N}^{2}$\\

17 & 2 & R & $(h_{1})_{1},f_{1},p_{1},N_{1}$ & $22+s_{h_{1}}+s_{f}+s_{p}+10s_{N}+s_{N}^{2}$\\

17 & 3 & R & $(h_{1})_{1},f_{1},p_{1},N_{1}$ & $20+s_{h_{1}}+s_{f}+s_{p}+9s_{N}+s_{N}^{2}$\\

17 & 4 & R & $(h_{1})_{1},f_{1},p_{1},N_{1}$ & $11+s_{h_{1}}+s_{f}+s_{p}+3s_{N}$\\

17 & 5 & R & $(h_{1})_{1},f_{1},N_{1}$ & $6+s_{h_{1}}+s_{f}+2s_{N}$\\

17 & 8 & R & $(h_{1})_{1},N_{1}$ & $3+s_{h_{1}}+s_{N}$\\

18 & 1 & R & $(h_{2})_{1},p_{1},N_{1}$ & $458+s_{h_{2}}+43s_{p}+387s_{N}+s_{p}^{2}+18s_{N}\,s_{p}+124s_{N}^{2}+2s_{N}^{2}\,s_{p}+18s_{N}^{3}+s_{N}^{4}$\\

18 & 2 & R & $(h_{2})_{1},p_{1},N_{1}$ & $456+s_{h_{2}}+43s_{p}+386s_{N}+s_{p}^{2}+18s_{N}\,s_{p}+124s_{N}^{2}+2s_{N}^{2}\,s_{p}+18s_{N}^{3}+s_{N}^{4}$\\

18 & 3 & R & $(h_{2})_{1},p_{1},N_{1}$ & $454+s_{h_{2}}+43s_{p}+385s_{N}+s_{p}^{2}+18s_{N}\,s_{p}+124s_{N}^{2}+2s_{N}^{2}\,s_{p}+18s_{N}^{3}+s_{N}^{4}$\\

18 & 4 & R & $(h_{2})_{1},p_{1},N_{1}$ & $445+s_{h_{2}}+43s_{p}+379s_{N}+s_{p}^{2}+18s_{N}\,s_{p}+123s_{N}^{2}+2s_{N}^{2}\,s_{p}+18s_{N}^{3}+s_{N}^{4}$\\

18 & 5 & R & $(h_{2})_{1},p_{1},N_{1}$ & $440+s_{h_{2}}+42s_{p}+378s_{N}+s_{p}^{2}+18s_{N}\,s_{p}+123s_{N}^{2}+2s_{N}^{2}\,s_{p}+18s_{N}^{3}+s_{N}^{4}$\\

18 & 6 & R & $(h_{2})_{1},p_{1},N_{1}$ & $398+s_{h_{2}}+43s_{p}+339s_{N}+s_{p}^{2}+18s_{N}\,s_{p}+112s_{N}^{2}+2s_{N}^{2}\,s_{p}+17s_{N}^{3}+s_{N}^{4}$\\

18 & 7 & R & $(h_{2})_{1},p_{1},N_{1}$ & $378+s_{h_{2}}+39s_{p}+330s_{N}+s_{p}^{2}+17s_{N}\,s_{p}+111s_{N}^{2}+2s_{N}^{2}\,s_{p}+17s_{N}^{3}+s_{N}^{4}$\\

18 & 8 & R & $(h_{2})_{2}$ & $1+s_{h_{2}}$\\

18 & 9 & R & $(h_{2})_{3}$ & $1+s_{h_{2}}$\\

19 & 1 & R & $(h_{3})_{1},f_{1},p_{1},N_{1}$ & $24+s_{h_{3}}+s_{f}+s_{p}+11s_{N}+s_{N}^{2}$\\

19 & 2 & R & $(h_{3})_{1},f_{1},p_{1},N_{1}$ & $22+s_{h_{3}}+s_{f}+s_{p}+10s_{N}+s_{N}^{2}$\\

19 & 3 & R & $(h_{3})_{1},f_{1},p_{1},N_{1}$ & $20+s_{h_{3}}+s_{f}+s_{p}+9s_{N}+s_{N}^{2}$\\

19 & 4 & R & $(h_{3})_{1},f_{1},p_{1},N_{1}$ & $11+s_{h_{3}}+s_{f}+s_{p}+3s_{N}$\\

19 & 5 & R & $(h_{3})_{1},f_{1},N_{1}$ & $6+s_{h_{3}}+s_{f}+2s_{N}$\\

19 & 6 & R & $(h_{3})_{2},g_{1},p_{1},N_{1}$ & $37+s_{h_{3}}+4s_{g}+4s_{p}+17s_{N}+s_{N}\,s_{g}+s_{N}\,s_{p}+2s_{N}^{2}$\\

19 & 7 & R & $(h_{3})_{2},g_{1},N_{1}$ & $17+s_{h_{3}}+4s_{g}+8s_{N}+s_{N}\,s_{g}+s_{N}^{2}$\\

19 & 8 & R & $(h_{3})_{1},N_{1}$ & $3+s_{h_{3}}+s_{N}$\\

19 & 9 & R & $(h_{3})_{2},g_{1},p_{1},N_{1}$ & $76+s_{h_{3}}+3s_{g}+3s_{p}+55s_{N}+s_{N}\,s_{g}+s_{N}\,s_{p}+13s_{N}^{2}+s_{N}^{3}$\\

20 & 1 & R & $(h_{4})_{1},g_{1},p_{1},N_{1}$ & $24+s_{h_{4}}+s_{g}+s_{p}+11s_{N}+s_{N}^{2}$\\

20 & 2 & R & $(h_{4})_{1},g_{1},p_{1},N_{1}$ & $22+s_{h_{4}}+s_{g}+s_{p}+10s_{N}+s_{N}^{2}$\\

20 & 3 & R & $(h_{4})_{1},g_{1},p_{1},N_{1}$ & $20+s_{h_{4}}+s_{g}+s_{p}+9s_{N}+s_{N}^{2}$\\

20 & 4 & R & $(h_{4})_{1},g_{1},p_{1},N_{1}$ & $11+s_{h_{4}}+s_{g}+s_{p}+3s_{N}$\\

20 & 5 & R & $(h_{4})_{1},g_{1},N_{1}$ & $6+s_{h_{4}}+s_{g}+2s_{N}$\\

20 & 6 & R & $(h_{4})_{1},g_{2},p_{1},N_{1}$ & $24+s_{h_{4}}+s_{g}+4s_{p}+10s_{N}+s_{N}\,s_{p}+s_{N}^{2}$\\

20 & 7 & R & $(h_{4})_{1},g_{2},N_{1}$ & $4+s_{h_{4}}+s_{g}+s_{N}$\\

20 & 8 & R & $(h_{4})_{2},f_{1},p_{1},N_{1}$ & $76+s_{h_{4}}+3s_{f}+3s_{p}+55s_{N}+s_{N}\,s_{f}+s_{N}\,s_{p}+13s_{N}^{2}+s_{N}^{3}$\\*

20 & 9 & R & $(h_{4})_{1},N_{1}$ & $3+s_{h_{4}}+s_{N}$\\*

\bottomrule\caption{All 158 certified fixed comparisons in the weight-5+4 matrix.}\end{longtable}\endgroup

\chapter{The reproducible audit and manuscript provenance}\label{app:audit}

\section{Files and execution}
\index{audit!files}\index{audit!execution}\index{verification!files}
The clean distribution contains the single source
\texttt{AA\_Monograph.tex}, its compiled PDF, a~README, and an
\texttt{audit} folder. That folder contains the two original standalone
verifiers, their regenerated result tables, and a~cumulative audit that
normalizes every binary parenthesization at every count from one
through twelve.

The commands are \texttt{python3 verify\_eleven.py --output-dir results\_11}
and \texttt{python3 n12\_verify.py --output-dir results\_12}, run from
the audit folder. The additional cumulative audit is run with
\texttt{python3 verify\_all.py}. Python~3.9 or later suffices, and no
third-party package or network connection is used. Assertions are part
of the audit; do not run it with Python's optimization option.
\index{verification!scope}\index{audit!reproducible}

The eleven-occurrence verifier uses the basic comparison classification
of Chapter~\ref{b:S:classification}. The twelve-occurrence verifier uses
the sharper classification of Chapter~\ref{t:S:classification}.
Their different proof assignments are intentional. The cumulative
record uses the first classification through eleven and the second
at twelve, exactly as in Table~\ref{tab:cumulative}.

\section{What each check establishes}
\index{verification!scope of checks}\index{audit!scope}
The weighted-multiset generator builds every unordered rooted tree of
the required weight. A~second recursive generator builds full binary
trees without using that construction. Normalizing those binary trees
checks equality of the two complete normal-form sets and every
individual multiplicity in~\eqref{E:multiplicity}. Root-partition
counts provide an additional aggregate check.

The proof classifier checks that every normal form has an assigned
uniform, regular, or named exceptional argument. The forest audit
checks its fixed reductions and the exceptional signature differences.
The arithmetic engine performs exact integer-polynomial substitutions
and tests every coefficient of the resulting difference. Appendix~
\ref{app:certificates} prints the resulting certificates. The five
non-strict certificates are accompanied by mathematical tie proofs.

These checks do not prove the canonical-coordinate theorem, verify
that every informal argument is correct in a~formal proof assistant,
or enumerate ordered-set inputs of any fixed cardinality. Those are
different tasks. The mathematical proof applies to inputs of arbitrary
finite cardinality through its structural arguments.

\section{Source consolidation}
\index{manuscript!provenance!source consolidation}\index{source!consolidation}
The supplied material consisted of the foundational joint manuscript,
the manuscripts treating five through eleven occurrences, a~general
induction note, and the separate twelve-occurrence archive. The
consolidation uses the self-contained cumulative strong-induction
argument through eleven and incorporates the foundational $A^A$
reconstruction directly. It adds the unbounded polynomial-perturbation
result from the induction note and all twelve-occurrence proofs.

Repeated preliminary sections have been replaced by the common
constructions of Part~I. References to the earlier manuscript sequence
have been replaced by internal theorem, chapter, and equation
references. The row-by-row mixed-root tables at six through nine,
the cumulative normal-form catalogue, and the printed exact certificate
appendix were generated for this consolidated exposition. Original
user archives were left unchanged.

The original eleven- and twelve-occurrence audit programs are retained
with their mathematical scope statements. Their output directories
in this distribution were regenerated, rather than merely copied from
the supplied results. A~separate cumulative audit records the checks
for every occurrence count.

\backmatter
\printindex

\begin{thebibliography}{9}
\bibitem{Farley}
J.\,D. Farley,
Does the endomorphism poset $P^P$ determine whether a~finite poset $P$
is connected? An issue Duffus raised in 1978.
\emph{Mathematica Bohemica} \textbf{148} (2023), 435--446.
DOI: 10.21136/MB.2022.0010-22.

\bibitem{FarleyPreprint}
J.\,D. Farley,
An issue raised in 1978 by a~then-future editor-in-chief of the journal
\emph{Order}: Does the endomorphism poset of a~finite connected poset
tell us that the poset is connected?
Preprint, arXiv:2005.03255v1, 2020.

\bibitem{Hashimoto}
J. Hashimoto,
On direct product decomposition of partially ordered sets.
\emph{Annals of Mathematics} (2) \textbf{54} (1951), 315--318.

\bibitem{McKenzie}
R. N. McKenzie,
Arithmetic of finite ordered sets: Cancellation of exponents, II.
\emph{Order} \textbf{17} (2000), 309--332.
DOI: 10.1023/A:1006449213916.
\bibitem{McKenzieZigzag}
R. N. McKenzie,
The zig-zag property and exponential cancellation of ordered sets.
\emph{Order} \textbf{20} (2003), 185--221.
\end{thebibliography}
\end{document}